\documentclass[12pt,a4paper]{article}

\usepackage[utf8]{inputenc}

\usepackage{amsthm}
\usepackage{amssymb}
\usepackage[tbtags]{amsmath}
\usepackage{mathtools}
\usepackage{a4wide}
\usepackage{amsfonts}
\usepackage{enumitem}
\usepackage{hyperref}
\usepackage{dsfont}
\usepackage{color}
\usepackage{authblk}

\newcommand{\lin}{\operatorname{span}}

\newcommand{\Lip}{\operatorname{Lip}}

\newcommand{\tr}{\operatorname{tr}}

\renewcommand{\div}{\operatorname{div}}

\newcommand{\hook}{\hookrightarrow}
\newcommand{\harp}{\rightharpoonup}

\newcommand{\grad}{\nabla}
\newcommand{\laplace}{\Delta}

\newcommand{\Rb}{\mathbb{R}}
\newcommand{\Nb}{\mathbb{N}}

\newcommand{\Eb}{\mathbb{E}}
\newcommand{\Pb}{\mathbb{P}}

\newcommand{\Fb}{\mathbb{F}}
\newcommand{\Fc}{\mathcal{F}}
\newcommand{\Fct}{\left( \mathcal{F}_t \right)_{t \geq 0}}
\newcommand{\Vc}{\mathcal{V}}
\newcommand{\Hc}{\mathcal{H}}
\newcommand{\Ac}{\mathcal{A}}
\newcommand{\Rc}{\mathcal{R}}
\newcommand{\Kc}{\mathcal{K}}
\newcommand{\Uc}{\mathcal{U}}
\newcommand{\Xc}{\mathcal{X}}
\newcommand{\Tc}{\mathcal{T}}
\newcommand{\Sc}{\mathcal{S}}
\newcommand{\Mc}{\mathcal{M}}
\newcommand{\Ho}{\overline{H}}
\newcommand{\Vo}{\overline{V}}
\newcommand{\Vtwo}{V_{(2)}}

\newcommand{\vb}{\bar{v}}
\newcommand{\vt}{\tilde{v}}

\newcommand{\Apr}{A_{\rm{pr}}}
\newcommand{\bun}{\bar{U}^n}
\newcommand{\una}{U^n_\ast}
\newcommand{\uone}{U^{(1)}}
\newcommand{\utwo}{U^{(2)}}

\providecommand{\keywords}[1]
{
	{\small
  	\textbf{\textit{Keywords---}} #1
  	}
}

\providecommand{\msc}[1]
{
	{\small
  	\textbf{\textit{MSC2010---}} #1
  	}
}

\newtheorem{theorem}{Theorem}[section]
\newtheorem{lemma}[theorem]{Lemma}
\newtheorem{corollary}[theorem]{Corollary}
\newtheorem{proposition}[theorem]{Proposition}
\newtheorem{definition}[theorem]{Definition}

\theoremstyle{definition}
\newtheorem{example}{Example}

\theoremstyle{remark}

\numberwithin{equation}{section}

\title{Well-Posedness of the 3D Stochastic Primitive Equations with Transport Noise}
\author[1]{Zdzis{\l}aw Brze\'{z}niak\footnote{Corresponding author; email: \href{mailto:zdzislaw.brzezniak@york.ac.uk}{zdzislaw.brzezniak@york.ac.uk}.}}
\author[2]{Jakub Slav\'{i}k\footnote{email: \href{mailto:slavik@utia.cas.cz}{slavik@utia.cas.cz}}}
\affil[1]{University of York, Department of Mathematics \protect\\James College, Campus West, YO10 5DD, United Kingdom}
\affil[2]{The Czech Academy of Sciences, Institute of Information Theory and Automation\protect\\Pod~Vod\'{a}renskou v\v{e}\v{z}\'{i} 4, 182 00 Prague 8, Czechia}
\date{}

\begin{document}

\maketitle

\begin{abstract}
We show that that the stochastic 3D primitive equations with either the physical boundary conditions or Neumann boundary conditions on the top and bottom and Dirichlet boundary condition on the sides driven by multiplicative gradient-dependent white noise have unique maximal strong solutions both in stochastic and PDE sense under certain assumptions on the growth of the noise. For the latter boundary conditions global existence is established using an argument based on decomposition of vertical velocity to barotropic and baroclinic modes and an iterated stopping time argument. An explicit example of non-trivial infinite dimensional noise depending on the vertical average of the horizontal gradient of horizontal velocity is presented.
\end{abstract}


\msc{35Q86, 60H15, 76M35, 86A05, 86A10}

\tableofcontents

\section{Introduction}

The primitive equations are one of the fundamental models in geophysical fluid dynamics. They can be derived from the Navier-Stokes equations using the hydrostatic approximation and the Boussinesq approximation, see e.g.\ \cite{pedlosky} or \cite{vallis}. In 3D, the primitive equations of the ocean perturbed by a multiplicative white noise read as
\begin{gather}
	\label{eq:pe.v.orig}
	\begin{multlined}
	\partial_t v + \left(v\cdot\grad\right)v + w\partial_z v + \tfrac{1}{\rho_0} \grad p + f \vec{k} \times v - \mu_v \laplace v - \nu_v \partial_{zz} v\\
	= F_v + \sigma_1\left(v, \grad_3 v, T, \grad_3 T \right) \dot{W}_1,
	\end{multlined}
	\\
	\partial_z p = -\rho g,\\
	\div v + \partial_z w = 0,\\
	\partial_t T + \left(v\cdot\grad\right)T + w\partial_z T - \mu_T \laplace T - \nu_T \partial_{zz} T = F_T + \sigma_2\left(v, \grad_3 v, T, \grad_3 T\right) \dot{W}_2,\\
	\label{eq:pe.rho.orig}
	\rho = \rho_0 \left(1 - \beta_T\left(T-T_r\right) \right),
\end{gather}
where $v = (v_1, v_2)$, $w$, $T$, $p$, $\rho$ are the horizontal velocity, vertical velocity, temperature, salinity, pressure and density, respectively, $f$ represents the Coriolis acceleration, $g$ is gravity and the coefficients $\mu_v$, $\nu_v$ and $\mu_T$, $\nu_T$ are the horizontal and vertical viscosity and heat diffusion coefficients, respectively. The system is driven by deterministic non-autonomous forces $F_v$, $F_T$ and stochastic terms with multiplicative white noise in time. The primitive equations usually contain salinity as well. However, it is omitted here since it does not introduce any additional mathematical problems.

Motivated by the papers \cite{brzezniak1991, brzezniak1992, Brz+Motyl-2013, capinski1991, capinski1993, mikulevicius2004, mikulevicius2005}, where stochastic Navier-Stokes equations with noise term depending on the gradient of velocity are considered, we aim to establish global well-posedness of the 3D stochastic primitive equations with a gradient-dependent noise. For physical justification of such noise terms see \cite{mikulevicius2001,mikulevicius2004} and the references therein. A similar result has been established in \cite{gao2012} for 2D stochastic primitive equations with only mild assumptions on the noise term.

A rigorous mathematical treatment of the primitive equations started in \cite{lions1992a, lions1992b} where global existence of weak solutions has been established. In 2007, the global existence of strong solutions was shown by \cite{cao2007,kobelkov2007, kukavica2007} with different boundary conditions and assumptions on the regularity of the domain. Recent results include global well-posedness in $L^p$ spaces \cite{hieber2016} and global existence of models with partial diffusivity and/or viscosity, see the review paper \cite{li2018} and a related result \cite{hussein2019}.

The existence theory for the 3D stochastic primitive is slightly less developed. In \cite{guo2009} global existence of pathwise (i.e.\ strong in stochastic sense) strong (in PDE sense) solutions of primitive equations driven by additive noise in the setting of \cite{cao2007} has been established. The existence of maximal pathwise strong solutions for primitive equations with so-called physical boundary conditions, that is the setting of e.g.\ \cite{kukavica2007}, driven by multiplicative noise has been established in \cite{debussche2011} with the assumptions
\begin{equation}
	\label{eq:sigma.debussche.local}
	\sigma \in \Lip\left(H, L_2\left(\Uc, H \right) \right) \cap \Lip\left(V, L_2\left(\Uc, V \right) \right) \cap \Lip\left(D(A), L_2\left(\Uc, D(A) \right) \right),
\end{equation}
where $\Uc$ is the reproducing kernel Hilbert space of the cylindrical Wiener process $W$ and $L_2(X, Y)$ and $\Lip(X, Y)$ are the spaces of all Hilbert-Schmidt operators and Lipschitz continuous mappings from $X$ to $Y$, respectively, for Hilbert spaces $X$ and $Y$. The spaces $H$ and $V$ are defined in Section \ref{sect:fspaces}. The maximal existence result has been later expanded to global well-posedness in \cite{debussche2012} with noise term $\sigma$ satisfying
\begin{equation}
	\label{eq:sigma.debussche.global}
	\sigma \in \Lip\left(H, L_2\left(\Uc, H\right) \right) \cap \Lip\left(V, L_2\left(\Uc, V \right) \right) \cap \Lip\left(V, L_2\left(\Uc, D(A) \right) \right).
\end{equation}
Clearly, the assumption \eqref{eq:sigma.debussche.global} excludes noise terms depending on gradients. Other results include existence of invariant measures \cite{glatt-holtz2014}, large deviations principle \cite{dong2017.1} and Markov selections result \cite{dong2017.2}, both under the assumption \eqref{eq:sigma.debussche.global}.

We present two results. First, we show local existence and uniqueness of maximal pathwise strong solutions to the stochastic 3D primitive equations \eqref{eq:pe.v.orig}-\eqref{eq:pe.rho.orig} for a general class of noise terms $\sigma$ assuming that the growth of $\sigma$ w.r.t.\ to the gradient is sufficiently small, see Theorem \ref{thm:maximal.existence} and Example \ref{example.1} in Section \ref{sect:main.results}. In particular using the notation above we require
\begin{equation}
	\label{eq:sigma.assumptions}
	\sigma \in \Lip \left( V, L_2\left(\Uc, H\right) \right) \cap \Lip \left( D(A), L_2\left(\Uc, V\right) \right)	
\end{equation}
with explicit control of some of the growth constants involved. This result holds both for the boundary conditions from \cite{debussche2012} and the physical boundary conditions from e.g.\ \cite{debussche2011}. Secondly, considering the boundary conditions from \cite{debussche2012} we establish global existence for a smaller class of noise terms still allowing the presence of e.g.\ vertical average of horizontal gradients of vertical velocity, see Theorem \ref{thm:global.existence} and Example \ref{example.2} in Section \ref{sect:main.results}. Similarly as above we require that $\sigma$ satisfies \eqref{eq:sigma.assumptions} together with explicit control of additional constants resulting from the decomposition technique of \cite{cao2007}. The maximal existence result is obtained by an adjustment of the technique from \cite{debussche2011}, while the global existence result relies on the decomposition method of \cite{cao2007} combined with an iterated version of the stopping time argument from \cite{debussche2012} and the stochastic Gronwall lemma from \cite{glatt-holtz2009}.

The maximal existence result directly improves the one in \cite{debussche2011} to a less regular noise terms and provides additional information about stochastic integrability of the solutions. Although the higher integrability is more or less a cosmetic improvement by itself, it is necessary for the global existence argument of \cite{cao2007} to go through. Since the argument closely follows the method from \cite{debussche2011}, we provide details mostly in the parts where the different regularity of the noise term plays any role and in parts omitted from \cite{debussche2011} where we believe that a more detailed exposition might be in order, in particular in the proof of the blow-up of maximal solutions.

The global existence result directly improves the well-posedness results for both the additive noise \cite{guo2009} and for the multiplicative noise \cite{debussche2012} to a less regular noise setting. A well-posedness result for stochastic 3D primitive equations with non-homogeneous physical boundary conditions and varying topography, which is the setting of \cite{kukavica2007} in the deterministic case, and possibly gradient dependent noise still remains an open problem. Global existence of martingale weak solutions in this setting has been established in \cite{glatt-holtz17}.

The proof of the existence of global solutions differs from the corresponding proofs for stochastic Navier-Stokes Equations, see \cite{Brz+Motyl-2013} and/or \cite{mikulevicius2005}. In those papers the global solutions are constructed directly without using local solutions. Moreover, we only use the classical Skorokhod Theorem \cite[Theorem 2.4]{dpz} without invoking a generalization to nonmetric spaces as it has been done in 
\cite{Brz+Motyl-2013}.

The rest of the paper is organized as follows: In Section 2 we recall the standard reformulation of the primitive equations, define the necessary function spaces, operators and the various notions of solutions. We also state the main results and provide explicit examples of admissible noise terms. In Section 3 we establish the maximal existence result. Section 4 consists of the proof of global well-posedness and necessary apriori estimates in the spirit of \cite{cao2007}. In the Appendix the reader can find a version of the It\^{o} Lemma used in Sections 3 and 4.

\section{Mathematical setting}

\subsection{Reformulation of the problem}

Let $\Mc_0 \subseteq \Rb^2$ be a bounded domain with $C^2$ boundary and let $\Mc = \Mc_0 \times (-h, 0)$ for $h > 0$ fixed. The boundary $\partial \Mc$ is partitioned into the the top part $\Gamma_i$, the lateral part $\Gamma_l$ and the bottom part $\Gamma_b$ defined respectively by
\[
	\Gamma_i = \overline{\Mc_0} \times \lbrace 0 \rbrace, \qquad \Gamma_l = \partial \Mc_0 \times (-h, 0), \qquad \Gamma_b = \overline{\Mc_0} \times \lbrace -h \rbrace.
\]
We emphasize that the operators $\div$, $\grad$ and $\laplace$ are acting only on the horizontal coordinates, i.e.\ for a sufficiently smooth function $v: \Mc \to \Rb^2$
\[
	\div v = \partial_x v_1 + \partial_y v_2, \qquad \grad v = \begin{pmatrix} \partial_x v_1 & \partial_y v_2\\ \partial_x v_2 & \partial_y v_2 \end{pmatrix}, \qquad \laplace v = \begin{pmatrix} \partial_{xx} v_1 + \partial_{yy} v_1\\ \partial_{xx} v_2 + \partial_{yy} v_2 \end{pmatrix}.
\]
The full gradient will be denoted by $\grad_3$. For simplicity of the notation we will sometimes use $\partial_1$ and $\partial_2$ instead of $\partial_x$ and $\partial_y$.

The primitive equations \eqref{eq:pe.v.orig}--\eqref{eq:pe.rho.orig} are supplemented by the initial conditions
\begin{equation*}
	v(0) = v_0, \qquad T(0) = T_0,
\end{equation*}
and the following boundary conditions
\begin{align*}
	\text{on} &\ \Gamma_i: & &\partial_z v = 0, \quad w = 0, \quad \nu \partial_z T + \alpha T = 0,\\
	\text{on} &\ \Gamma_l: & &v = 0, \quad w = 0,\quad \partial_{\vec{n}_H} T = 0,\\
	\text{on} &\ \Gamma_b: & &\partial_z v = 0, \quad w = 0, \quad \partial_z T = 0.
\end{align*}
where $\vec{n}_H \in \Rb^2$ is the horizontal part of the outer unit normal to $\partial \Mc$. The maximal existence result in Theorem \ref{thm:maximal.existence} will also hold for the physical boundary conditions, see Section \ref{sect:main.results}, since the problem can be formulated in the same abstract functional way.

Following a standard argument, see \cite[Section 2.1]{petcu2009}, we may reformulate equations \eqref{eq:pe.v.orig}--\eqref{eq:pe.rho.orig} as follows
\begin{gather}
	\label{eq:pe.v.reform}
	\begin{multlined}
		\partial_t v + \left(v\cdot\grad\right)v + w(v)\partial_z v + \tfrac{1}{\rho_0} \grad p_S - \beta_T g \grad \int_z^0 T \, dz' + f \vec{k} \times v\\
		- \mu_v \laplace v - \nu_v \partial_{zz} v = F_v + \sigma_1\left(v, \grad v, T, \grad T\right) \dot{W}_1,
	\end{multlined}\\
	\div \int_{-h}^0 v(x, y, z') \, dz' = 0,\\
	\label{eq:pe.T.reform}
	\partial_t T + \left(v\cdot\grad\right)T + w\partial_z T - \mu_T \laplace T - \nu_T \partial_{zz} T = F_T + \sigma_2\left(v, \grad v, T, \grad T\right) \dot{W}_2,
\end{gather}
where
\begin{gather}
	\label{eq:pe.w.definition}
	w(v)(x, y) = - \int_{-h}^z \div v(x, y, z') \, dz'\\
	\label{eq:pe.p.reform}
	p_S = p - P, \qquad P = P(T) = g \int_z^0 \rho \, dz',\\
	\label{eq:pe.rho.reform}
	\rho = \rho_0 \left(1 - \beta_T(T-T_r) \right),
\end{gather}
and $p_S$ is the surface pressure. Here we interpret $\vec{k} \times v = \vec{k} \times (v_1, v_2) = (-v_2, v_1)$.  A given number $T_r$ is  the reference value of the temperature.

We remark that in what follows we will be considering only the couple of \emph{prognostic} variables $U = (v, T)$ since the remaining \emph{diagnostic} variables $w$, $p$, $\rho$ can be inferred from the prognostic variables using the equations \eqref{eq:pe.w.definition}, \eqref{eq:pe.p.reform} and \eqref{eq:pe.rho.reform} and the hydrostatic Helmholtz-Leray projection $P_H$ defined below. We will sometimes write $U = (U_1, U_2)$ instead of $U = (v, T)$. Also, for simplicity, we will from now on assume that $\mu_v = \mu_t \equiv \mu$ and $\nu_T = \nu_v \equiv \nu$.

\subsection{Function spaces and operators}
\label{sect:fspaces}

For $p \in [1, \infty]$ we will denote the Lebesgue spaces on the domain $\Mc$ by $L^p = L^p \left(\Mc\right)$.

Let us define
\begin{gather*}
	\Vc_1 = \left\{ v \in C^\infty\left(\overline{\Mc}; \Rb^2\right) \mid \div \int_{-h}^0 v \, dz' = 0, v = 0 \ \text{in some neigbourhood of} \ \Gamma_l \right\},\\
	\Vc_2 = C^\infty\left( \overline{\Mc} \right), \qquad \Vc = \Vc_1 \times \Vc_2.
\end{gather*}
Let $H_1$ and $H_2$ be the closure of $\Vc_1$ and $\Vc_2$ in $L^2\left(\Mc; \Rb^2 \right)$ and in $L^2\left(\Mc \right)$, respectively, and let $H = H_1 \times H_2$. Equipped with the inner product of $L^2\left(\Mc; \Rb^3\right)$, the space $H$ is a separable Hilbert space. The norms and the inner products on $H_i$ and $H$ will be denoted by $\vert \cdot \vert$ and $\left( \cdot, \cdot\right)$, respectively. Let $P_{H_1}: L^2\left(\Mc, \Rb^2 \right) \to H_1$ denote the hydrostatic Helmholtz-Leray projection, see e.g.\ \cite[Lemma 2.2]{lions1992b} and \cite[Section 4]{hieber2016}. More details on $P_{H_1}$ will be below. The projection $P_H$ is then defined by
\[
	P_H U = \begin{pmatrix}
		P_{H_1} v\\
		T
	\end{pmatrix}, \qquad U = (v, T) \in L^2\left( \Mc, \Rb^3 \right).
\]

Let $V_1$ and $V_2$ be the closure of $\Vc_1$ and $\Vc_2$ in $H^1\left( \Mc; \Rb^2 \right)$ and $H^1\left( \Mc \right)$, respectively, and let $V = V_1 \times V_2$. The space $V$ equipped with the inner product of $H^1\left( \Mc; \Rb^3 \right)$ is a separable Hilbert space. We will denote the norms on the spaces $V_i$ and $V$ by $\Vert \cdot \Vert$. We remark that by e.g.\ \cite[Theorem 3.3]{alessandrini2008} the Poincaré inequality $\Vert v \Vert \leq C \vert \grad v \vert$ holds for $v \in V_1$, which implies the equivalence of norm $\Vert u \Vert \simeq \vert \grad u\vert$. Finally, let $\Vtwo$ be the closure of $\Vc$ in $H^2(\Mc; \Rb^3)$ equipped with the inner product of $H^2(\Mc; \Rb^3)$.

Let $a: V \times V \to \Rb$, $a_i: V_i \times V_i \to \Rb$ be the bilinear forms defined by
\begin{align*}
	a_1\left(v, v^\sharp\right) &= \int_{\Mc} \mu \grad v \cdot \grad v^\sharp + \nu \partial_z v \partial_z v^\sharp \, d\Mc, & &v, v^\sharp \in V_1,\\
	a_2\left(T, T^\sharp\right) &= \int_{\Mc} \mu \grad T \cdot \grad T^\sharp + \nu \partial_z T \partial_z T^\sharp \, d\Mc + \alpha \int_{\Gamma_i} T T^\sharp \, d\Gamma_i, & &T, T^\sharp \in V_2,\\
	a\left(U, U^\sharp\right) &= a_1\left(v, v^\sharp\right) + a_2\left(T, T^\sharp\right), & &U, U^\sharp \in V,
\end{align*}
where $d\Mc = dx\, dy\, dz$ denotes the Lebesgue measure on $\Mc$. In a similar fashion we will denote the Lebesgue measure on $\Mc_0$ by $d\Mc_0$, therefore $d\Mc_0 = dx\, dy$. By \cite[Lemma 2.4]{lions1992b}, the forms $a$ and $a_i$ are continuous and coercive, that is $a$ and $a_i$ satisfy
\begin{gather*}
	a\left(U, U^\sharp\right) \leq C \Vert U \Vert \Vert U^\sharp \Vert, \quad a\left(U, U\right) \geq c \Vert U \Vert^2,  \qquad U, U^\sharp \in V,\\
	a_i\left(U_i, U^\sharp_i\right) \leq C \Vert U_i \Vert \Vert U^\sharp_i \Vert, \quad a_i\left(U_i, U_i\right) \geq c \Vert U_i \Vert^2, \qquad U_i, U_i^\sharp \in V_i.
\end{gather*}
By the Riesz Theorem there exists isomorphisms $\tilde{A}: V \to V'$ and $\tilde{A}_i: V_i \to V_i'$. The unbounded operators $A: H \to H$ and $A_i: H_i \to H_i$ are defined by
\begin{align*}
	D(A) &= \left\lbrace U \in V \mid \tilde{A}U \in H \right\rbrace, & D(A_i) &= \left\lbrace U_i \in V_i \mid \tilde{A}_i U_i \in H_i \right\rbrace,\\
	AU &= \tilde{A}U, \quad U \in D(A), & A_i U_i &= \tilde{A}_i U_i, \quad U_i \in D(A_i).
\end{align*}
The operator $A$ is called the hydrostatic Stokes operator. By \cite[Lemma 2.4]{lions1992b}, see also \cite[Section VI.\S 2]{kato}, the operators $A$ and $A_i$ are self-adjoint and the inverse operators $A^{-1}: V' \to V$ and $A_i^{-1}: V_i' \to V$ are compact. By a standard argument there exists an increasing sequence of positive eigenvalues $\lbrace \lambda_k \rbrace_{k = 1}^\infty$ and a corresponding orthonormal basis consisting of eigenvectors $\lbrace h_k \rbrace_{k=1}^\infty$ of $A$.  Let us recall that  the fractional power $A^s$,  for $s > 0$,   of $A$ is defined by
\begin{align*}
	D\left( A^s \right) &= \left\lbrace U \in H : \sum_{k = 1}^\infty \lambda_k^{2s} \left| \left(U, h_k \right) \right| < \infty \right\rbrace, \quad A^s U &= \sum_{k=1}^\infty \lambda^s_k \left( U, h_k \right) h_k, \mbox{ for } U \in D\left(A^s\right). 
\end{align*}
For $s>0$ let us put 
\[
	 \vert U \vert_s = \vert A^s U \vert = \left( \sum_{k=1}^\infty \lambda^{2s}_k \left| \left( U, h_k \right) \right|^2 \right)^{1/2}.
\]
We have $D\left(A^{1/2}\right) = V$ and $\Vert U \Vert = \vert U \vert_{1/2}$. For $n \in \Nb$ let $H^n = \lin \lbrace h_1, h_2, \dots, h_n \rbrace$. Let $P_n: H \to H^n$ and $Q_n = I - P_n$ denote the canonical projection operator and its complement. Note that for $0 < s_1 < s_2$ the following Poincaré type inequalities holds
\begin{equation}
	\label{eq:inverse.poincare}
	\vert P_n U \vert_{s_2} \leq \lambda_n^{s_2 - s_1} \vert P_n U \vert_{s_1}, \qquad \vert Q_n U \vert_{s_1} \leq \lambda_n^{-(s_2 - s_1)} \vert Q_n U \vert_{s_2}, \qquad U \in D(A^{s_2}).	
\end{equation}
For a proof of \eqref{eq:inverse.poincare} see e.g.\ \cite[Lemma 2.1]{glatt-holtz2009}.

Let $b: V \times V \times \Vtwo \to \Rb$ be the trilinear form defined by
\[
	b(U, U^\sharp, U^\flat)	= \left( P_H
		\begin{pmatrix}
			v \cdot \grad v^\sharp + w(v) \partial_z v^\sharp\\
			v \cdot \grad T^\sharp + w(v) \partial_z T^\sharp
		\end{pmatrix}, U^\flat \right), \qquad U, U^\sharp \in V, U^\flat \in \Vtwo.
\]
Similarly as in \cite[Lemma 2.1 and Lemma 3.1]{petcu2009} we may show that $b$ is continuous on $V \times V \times \Vtwo$ and $V \times \Vtwo \times V$ and satisfies the following anti-symmetry property
\begin{equation}
	\label{eq:b.symmetry}
	b\left(U, U^\sharp, U^\flat\right) = - b\left(U, U^\flat, U^\sharp\right), \qquad U, U^\sharp, U^\flat \in V \ \text{and} \ U^\sharp \ \text{or} \ U^\flat \in \Vtwo.
\end{equation}
Moreover, $b$ satisfies the estimates
\begin{align}
	\label{eq:b.estimate1}
	\left|b\left(U, U^\sharp, U^\flat\right)\right| &\leq c_b \Vert v \Vert \Vert U^\sharp \Vert_{H^2} \Vert U^\flat \Vert, & U, U^\flat \in V, U^\sharp \in \Vtwo,\\
	\label{eq:b.estimate2}
	\left|b\left(U, U^\sharp, U^\flat\right)\right| &\leq c_b \Vert U \Vert^{1/2} \Vert U \Vert^{1/2}_{H^2} \Vert U^\sharp \Vert^{1/2} \Vert U^\sharp \Vert^{1/2}_{H^2} \vert U^\flat \vert & U, U^\sharp \in \Vtwo, U^\flat \in H,
\end{align}
and for $U, U^\sharp \in \Vtwo$, $U^\flat \in H$ we have
\begin{equation}
	\label{eq:b.estimate3}
	\left|b\left(U, U^\sharp, U^\flat\right)\right| \leq c_b \vert U^\flat \vert \left( \vert v \vert_{L^6} \Vert U^\sharp \Vert^{1/2} \Vert U^\sharp \Vert_{H^2}^{1/2} + \Vert v \Vert^{1/2} \vert v \vert_{H^2}^{1/2} \vert \partial_z U^\sharp \vert^{1/2} \Vert \partial_z U \Vert^{1/2} \right).
\end{equation}
Note that the anti-symmetry property \eqref{eq:b.symmetry} implies that
\begin{equation}
	\label{eq:b.zero}
	b\left(U, U^\sharp, U^\sharp\right) = 0 \qquad  \text{for all} \ U \in V, U^\sharp \in \Vtwo.
\end{equation}
We associate a bilinear operator $B: V \times V \to \Vtwo'$ to the trilinear form $b$ by
\[
	\left(B\left(U, U^\sharp\right), U^\flat\right) = b\left(U, U^\sharp, U^\flat\right), \qquad U, U^\sharp \in V, U^\flat \in \Vtwo,
\]
and as usual we write $B(U) = B(U, U)$. By the anti-symmetry property \eqref{eq:b.symmetry} one can also assume that $B: V \times \Vtwo \to V'$. We remark that by \eqref{eq:b.estimate1} and \eqref{eq:b.estimate2}, respectively, the operator $B$ is continuous as a map from $V \times \Vtwo$ to $V'$ and from $\Vtwo \times \Vtwo$ to $H$.

Following the notation of \cite{debussche2011} we define the operators $\Apr: V \to H$ and $E: H \to H$ by
\[
	\Apr U = P_H \begin{pmatrix}
		- \beta_T g \grad \int_z^0 T \, dz'\\
		0
	\end{pmatrix}, \quad
	E U^\sharp = P_H \begin{pmatrix}
		f \vec{k} \times v^\sharp\\
		0
	\end{pmatrix}, \qquad U \in V, U^{\sharp} \in H.
\]
Let $F_U$ be a progressively measurable process such that for all $t > 0$ we have
\begin{equation}
	\label{eq:f.source}
	F_U = P_H \begin{pmatrix}
			F_v\\
			F_T
		\end{pmatrix} \in L^2\left( \Omega, L^2\left(0, t; H\right) \right).
\end{equation}
Let
\begin{equation}
	\label{eq:F.definition}
	F(U) = \Apr U + EU - F_U, \qquad U \in V.	
\end{equation}
Then $F: L^2\left( 0, t; V \right) \to L^2\left( 0, t; H \right)$ satisfies the linear growth condition
\begin{equation}
	\label{eq:F.bounded}
	\int_0^t \vert F(U) \vert^2 \, ds \leq C \left(\Vert F_U \Vert^2_{L^2\left( 0, t; H \right)} + \int_0^t \Vert U \Vert^2 \, ds \right) \qquad \text{for all} \ t  > 0,
\end{equation}
and the Lipschitz continuity condition
\begin{equation}
	\label{eq:F.lipschitz}
	\vert F(U) - F(U^\sharp) \vert \leq C \Vert U - U^\sharp \Vert, \qquad U, U^\sharp \in V.
\end{equation}
Sometimes we will include the $L^2\left(0, t; H\right)$-norm of $F_U$ in \eqref{eq:F.bounded} in the constant $C = C_t$.

Let $\Ho$ and $\Vo$ be the Hilbert space defined by
\begin{align*}
	\Ho &= \left\lbrace v \in L^2\left(\Mc_0; \Rb^2\right) \mid \div v = 0 \ \text{in} \ \Mc_0, v \cdot \vec{n} = 0 \ \text{on} \ \partial \Mc_0 \right\rbrace,\\
	\Vo &= \left\lbrace v \in H^1(\Mc_0; \Rb^2) \cap H \mid v = 0 \ \text{on} \ \partial \Mc_0 \right\rbrace.
\end{align*}
The space $\Ho$ is equipped with the inner product of $L^2\left(\Mc_0, \Rb^2\right)$. On $\Vo$ we assume the iner product given by
\[
	\left( u, v \right)_V = \mu \int_{\Mc_0} \grad u \cdot \grad v \, d\Mc, \qquad u, v \in \Vo.
\]
We will denote the norms on $\Ho$ and $\Vo$ by $\vert \cdot \vert$ and $\Vert \cdot \Vert$, respectively, if there is no ambiguity. Let $A_S: D(A_S) \to \Ho$ be the 2D Stokes operator with Dirichlet boundary condition. It is well known that $A_S$ is a self-adjoint operator and by the result from \cite{giga1985} we have the equivalence of norms
\[
	\vert A_S v \vert \simeq \Vert v \Vert_{H^2\left(\Mc_0, \Rb^2 \right)}, \ v \in D\left(A_S\right).
\]

Let $\Ac_2$ and  $\Ac_3$ be the averaging operators defined for $v: \Mc \to \Rb^2$ by
\begin{equation}
	\label{eq:averaging.operators}
	\left(\Ac_2 v\right)(x, y) = \frac{1}{h} \int_{-h}^0 v(x, y, z') \, dz', \quad \left( \Ac_3 v \right)(x, y, z) = \left( \Ac_2 v\right)(x, y),
\end{equation}
It is straightforward to check that $\Vert \Ac_3 \Vert_{L(H_1)} \leq 1$ and $\Vert \Ac_2 \Vert_{L(H_1, \Ho)} \leq h^{-1/2}$. Also let $\Rc = I - \Ac_3$. Then clearly $\Rc: H_1 \to H_1$ and $\Vert \Rc \Vert_{L\left(H_1\right)} \leq 2$. Moreover, since the spaces $H$ and $\Ho$ have the norm of $L^2\left(\Mc; \Rb^2 \right)$ and $L^2\left( \Mc_0; \Rb^2 \right)$, respectively, we observe that the operators $\Ac_2$, $\Ac_3$ and $\Rc$ remain bounded also if considered with $L^2\left(\Mc\right)$ and $L^2\left(\Mc_0\right)$ in place of $H_1$ and $\Ho$, respectively.

Let us observe that  $v = \Ac v + \Rc v$ for $v \in H_1$. Moreover, following \cite{giga2017} one has  \[P_{H_1} v = P_{\Ho} \Ac v + \Rc v, \;\; v\in L^2\left(\Mc, \Rb^2 \right), \] where $P_{\Ho}$ is the standard 2D Helmholtz-Leray projection on $L^2\left(\Mc_0; \Rb^2 \right)$.

Let $\Uc$ be an auxiliary Hilbert space with an orthonormal basis $\lbrace e_k \rbrace_{k=1}^\infty$ and let $\sigma: V \to L_2\left(\Uc, H \right)$ be defined by
\[
	\sigma(U) = P_H \begin{pmatrix}
					\sigma_1(v, \grad_3 v, T, \grad_3 T)\\
					\sigma_2(v, \grad_3 v, T, \grad_3 T)
				\end{pmatrix}, \qquad U = (v, T) \in V.
\]
For simplicity we will from now on write $\sigma_i(U)$ instead of $\sigma_i(v, \grad_3 v, T, \grad_3 T)$. With a slight abuse of notation, let us define the operator $\Ac_2 \sigma_1: V \to L_2\left(\Uc, H\right)$ by
\[
	\left( \Ac_2 \sigma_1 \right)(U)\zeta = \Ac_2 \left( \sigma_1(U) \zeta \right), \quad U \in V, \zeta \in \Uc,
\]
and similarly for $\Ac_3$ and $\Rc$.

We will consider two sets of assumptions on the noise term $\sigma$. The assumptions can be naturally extended to time-dependent functions $\sigma$.
\begin{enumerate}
	\item First, for the local existence, we assume that $\sigma$ considered as a mapping from $V$ to $H$ and from $D(A)$ to $H$ is continuous and satisfies the following sub-linear growth conditions
	\begin{align}
		\label{eq:sigma.bnd.H}
		\Vert \sigma(U) \Vert_{L_2(\Uc, H)}^2 &\leq C\left( 1 + \vert U \vert^2 \right) + \eta_0 \Vert U \Vert^2, & U &\in V,\\
		\label{eq:sigma.bnd.V}
		\Vert \sigma(U) \Vert_{L_2(\Uc, V)}^2 &\leq C \left( 1 + \Vert U \Vert^2 \right) + \eta_1 \vert A U \vert^2, & U &\in \Vtwo,
	\end{align}
	for some $\eta_j > 0$, where $L_2(K, L)$ denotes	the space of Hilbert-Schmidt operators from a Hilbert space $K$ to another Hilbert space $L$. Moreover, for the local uniqueness in Proposition \ref{prop:pathwise.uniqueness} and the maximal existence in Theorem \ref{thm:maximal.existence} we assume that $\sigma$ is Lipschitz continuous and satisfies
	\begin{align}
		\label{eq:sigma.lip.H}
		\Vert \sigma(U) - \sigma(U^\sharp) \Vert_{L_2(\Uc, H)}^2 &	\leq C \Vert U - U^\sharp \Vert^2, & U, U^\sharp &\in V,\\
		\label{eq:sigma.lip.V}
		\Vert \sigma(U) - \sigma(U^\sharp) \Vert_{L^2(\Uc, V)}^2 &\leq C \Vert U - U^\sharp \Vert^2 + \gamma \vert AU - AU^\sharp \vert^2, & U, U^\sharp &\in \Vtwo,
	\end{align}
	for some $\gamma > 0$.
	\item Secondly, for the global existence, we assume that for $U = (v, T) \in \Vtwo$ the functions $\sigma_i$ moreover satisfy the following inequalities for some $\eta_j > 0$:
	\begin{align}
		\label{eq:sigma.rc.vt}
		\sum_{k=1}^\infty \left| \Rc \sigma_1(U) e_k \right|_{L^6}^2 &\leq C \left( 1 + \vert \Rc v \vert_{L^6}^2 \right) \left( 1 + \Vert U \Vert^2 \right),\\
		\label{eq:sigma.rc.t}
		\sum_{k=1}^\infty \left| \sigma_2(U) e_k \right|_{L^6}^2 &\leq C \left( 1 + \vert T \vert_{L^6}^2 \right) \left(1 + \Vert U \Vert^2 \right),\\
		\label{eq:sigma.grad.ac}
		\Vert \Ac_2 \sigma_1(U) \Vert_{L_2(\Uc, \Vo)}^2 &\leq C \left( 1 + \Vert U \Vert^2 \right) + \eta_2 \vert A_S \Ac_2 v \vert_{\Ho}^2,\\
		\label{eq:sigma.dz}
		\Vert \partial_z \sigma_i(U) \Vert_{L_2(\Uc, L^2)}^2 &\leq C \left( 1 + \Vert U \Vert^2 \right) + \eta_3 \vert \grad_3 \partial_z U_i \vert^2.
	\end{align}
\end{enumerate}

Let $X$ and $Y$ be Banach spaces and let $\Lip(X, Y)$ denote the set of all Lipschitz continuous maps from $X$ to $Y$. Recalling that the space of Hilbert-Schmidt operators have the ideal property, we observe that
\begin{equation}
	\label{eq:average.hs}
	\begin{gathered}
		\Ac_2 \sigma_1 \in \Lip\left(V, L_2\left( \Uc, H_f \right)\right) \cap \Lip \left( D(A), L_2\left(\Uc, H^1(\Mc_0\right) \right),\\
		\Ac_3 \sigma_1, \Rc \sigma_1 \in \Lip\left(V, L_2\left( \Uc, H \right) \right) \cap \Lip \left( D(A), L_2 \left( \Uc, H^1\left( \Mc \right) \right) \right).
	\end{gathered}
\end{equation}

Let us also recall the definitions of Sobolev spaces with fractional time derivative, see e.g.\ \cite{simon1990}. Let $X$ be a separable Hilbert space and let $t > 0$, $p > 1$ and $\alpha \in (0, 1)$. We define
\[
	W^{\alpha, p}(0, t; X) =  \left\{ u \in L^p(0, t; X) \mid \int_0^t \int_0^t \frac{\vert u(s) - u(r) \vert_X^p}{|s - r|^{1+\alpha p}} \, dr \, ds < \infty \right\}
\]
and equip it with the norm
\[
	\Vert u \Vert^{p}_{W^{\alpha, p}(0, t; X)} = \int_0^t \vert u(s) \vert^p_X \, ds + \int_0^t \int_0^t \frac{\vert u(s) - u(r) \vert_X^p}{|s - r|^{1+\alpha p}} \, dr \, ds.
\]

\subsection{Stochastic preliminaries}

Let $\Sc = \left(\Omega, \Fc, \Fb, \Pb\right)$ be a stochastic basis\footnote{In the whole text we always assume that the stochastic bases satisfy the usual conditions.} with filtration $\Fb = \Fct$. Let $\Uc$ be a separable Hilbert space and let $W$ be an $\Fb$-adapted cylindrical Wiener process with reproducing kernel Hilbert space $\Uc$ on $\Sc$. Let $\lbrace e_k \rbrace_{k = 1}^\infty$ be an orthonormal basis of $\Uc$. Let $X$ be another separable Hilbert space and assume that $\Phi \in L^2\left(\Omega, L^2\left(0, T; L_2\left(\Uc, X \right)\right)\right)$. Recall that the stochastic integral w.r.t.\ a cylindrical Wiener process $W$ is defined by
\[
	\int_0^T \Phi \, dW = \sum_{k = 1}^\infty \int_0^T \Phi e_k \, dW^k,
\]
where $W^k$ are independent one dimensional Wiener processes on $\Sc$ such that formally $W = \sum_{k = 1}^\infty e_k W_k$, see e.g.\ \cite[Section 4]{dpz}. The definition of the stochastic integral can be extended to processes satisfying
\begin{equation}
	\label{eq:stochastic.integral.extension}
	\int_0^T \Vert \Phi \Vert^2_{L_2\left( \Uc, X \right)} \, dt < \infty \qquad \Pb\text{-almost surely}.
\end{equation}
For more details see e.g.\ \cite[Section 4]{dpz}.

We will often use the Burkholder-Davis-Gundy inequality
\begin{equation}
	\label{eq:bdg}
	\Eb \sup_{t \in \left[0, T\right]} \left| \int_0^t \Phi \, dW \right|^r_X \leq C_{BDG, r} \, \Eb \left( \int_0^T \Vert \Phi \Vert_{L_2(\Uc, X)}^2 \, dt \right)^{r/2},
\end{equation}
where $\Phi \in L^2\left(\Omega, L^2\left(0, T; L_2\left(\Uc, X\right)\right)\right)$. For proof see e.g.\ \cite[Theorem 3.28, p.\ 166]{karatzas1991}. If $r = 1$, we omit part of the subscript and write $C_{BDG}$ instead of $C_{BDG, 1}$. Moreover, we will require a fractional variation of the above inequality from  \cite[Lemma 2.1]{flandoli1995}. Let $p \geq 2$ and $\alpha \in [0, 1/2)$, then
\begin{equation}
	\label{eq:bdg.frac}
	\Eb \left| \int_0^\cdot \Phi \, dW \right|^p_{W^{\alpha, p}(0, T; X)} \leq c_{BDG, p} \, \Eb \int_0^T \Vert \Phi \Vert_{L_2(\Uc, X)}^p \, dt,
\end{equation}
where $\Phi \in L^p\left(\Omega, L^p\left(0, T; L_2\left(\Uc, X\right)\right)\right)$.

Let $H$, $V$ and $\Ho$ be the Hilbert spaces defined in Section \ref{sect:fspaces} and let $\sigma$ satisfy \eqref{eq:sigma.bnd.H}. Given $U \in L^2\left(\Omega, L^2\left(0, T; V\right)\right)$, by \cite[Proposition 4.30]{dpz} and \eqref{eq:average.hs} we have
\begin{align*}
	\Ac_2 \int_0^T \sigma_1(U) \, dW_1 &= \int_0^T \Ac_2 \sigma_1(U) \, dW_1 \in \Ho,\\
	 \Ac_3 \int_0^T \sigma_1(U) \, dW_1 &= \int_0^T \Ac_3 \sigma_1(U) \, dW_1 \in H,
\end{align*}
both identities holding $\Pb$-almost surely. Similar argument holds in other spaces such as $V$ or $\Vtwo$.

In Section \ref{sect:global.solution} we will often need the stochastic Gronwall Lemma form \cite[Lemma 5.3]{glatt-holtz2009}.

\begin{proposition}
\label{prop:stochastic.gronwall}
Let $t > 0$ and $X, Y, Z, R: [0, \infty) \times \Omega \to [0, \infty)$ be stochastic processes. Let $\tau: \Omega \to [0, t)$ be a stopping time such that
\begin{equation*}
	\Eb \int_0^\tau RX + Z \, ds < \infty
\end{equation*}
and let $\kappa > 0$ be a constant for which
\begin{equation*}
	\int_0^\tau R \, ds < \kappa \qquad \Pb-\text{almost surely}.
\end{equation*}
Assume that there exists $C_0 > 0$ such that for all stopping times $\tau_a, \tau_b$ such that $0 \leq \tau_a \leq \tau_b \leq \tau$ one has
\begin{equation*}
	\Eb \left[ \sup_{s \in \left[\tau_a, \tau_b\right]} X + \int_{\tau_a}^{\tau_b} Y \, ds \right] \leq C_0 \Eb \left[ X(\tau_a) + \int_{\tau_a}^{\tau_b} RX + Z \, ds \right],
\end{equation*}
then
\begin{equation*}
	\Eb \left[ \sup_{s \in \left[0, \tau\right]} X + \int_{0}^{\tau} Y \, ds \right] \leq C(C_0, t, \kappa) \Eb \left[ X(0) + \int_0^{\tau} Z \, ds \right].
\end{equation*}
\end{proposition}

\subsection{Definition of solutions}

For martingale solutions (i.e.\ weak in stochastic sense), we consider the initial data to be given by a Borel measure $\mu_0$ on $V$ satisfying, for some $q \geq 2$,
\begin{equation}
	\label{eq:mu.zero}
	\int_V \Vert U \Vert^q \, d\mu_0(U) < \infty.
\end{equation}
Given a stochastic basis $\left( \Omega, \Fc, \Fb, \Pb \right)$, we may find an $\Fc_0$-measurable $V$-valued random variable $U_0$ such that the law of $U_0$ is $\mu_0$ and $U_0 \in L^q\left( \Omega; \Fc_0, V\right)$.

We may now reformulate the equations \eqref{eq:pe.v.reform}-\eqref{eq:pe.rho.reform} in an abstract form
\begin{equation}
	\label{eq:pe.abstract}
	dU + \left[ AU + B(U) + \Apr U + EU \right] \, dt =  F_U \, dt + \sigma(U) \, dW, \qquad U(0) = U_0.
\end{equation}

\begin{definition}
Let $\mu_0$ be a Borel probability measure on $V$ such that \eqref{eq:mu.zero} holds for some $q \geq 2$. Let $F_U$ satisfy \eqref{eq:f.source} and let $\sigma$ be such that \eqref{eq:sigma.bnd.H}-\eqref{eq:sigma.lip.V} holds. A quadruple $(\Sc, W, U, \tau)$ is called a \emph{local martingale solution} if $\Sc = \left(\Omega, \Fc, \Fb, \Pb\right)$ is a stochastic basis, $W$ is an $\Fb$-adapted cylindrical Wiener process with reproducing kernel Hilbert space $\Uc$, $\tau$ is an $\Fb$-stopping time and $U\left(\cdot \wedge \tau\right): \Omega \times [0, \infty) \to V$ is a progressively measurable process\footnote{We could define the solution to be measurable and adapted and then use the fact that for a measurable and adapted process  there is a progressively measurable modification. In the rest of the manuscript we will understand the progressively measurable processes up to a modification, in particular in Appendix \ref{app:ito}.} such that $\tau > 0$ $\Pb$-a.s.\ and for all $t \geq 0$
\begin{equation}
	\label{eq:solution.regularity}
	U\left( \cdot \wedge \tau \right) \in L^2\left(\Omega; C\left([0, t], V\right)\right), \qquad \mathds{1}_{[0, \tau]}(\cdot) U \in L^2\left(\Omega; L^2\left(0, t; D(A) \right)\right),
\end{equation}
the law of $U(0)$ is $\mu_0$ and $U$ satisfies the following equality in $H$
\begin{equation}
	\label{eq:solution.def}
	U\left(t \wedge \tau\right) + \int_0^{t \wedge \tau} AU + B(U) + \Apr U + EU - F_U \, ds = U(0) + \int_0^{t \wedge \tau} \sigma(U) \, dW
\end{equation}
for all $t \geq 0$.

Moreover, if $\tau = +\infty$ $\Pb$-a.s.\ we call the triple $(\Sc, W, U)$ a \emph{global martingale solution}.
\end{definition}

\begin{definition}
\label{def:pathwise.sol}
Let $F$ and $\sigma$ satisfy \eqref{eq:f.source} and \eqref{eq:sigma.bnd.H}-\eqref{eq:sigma.lip.V}. Let $\Sc = \left(\Omega, \Fc, \Fb, \Pb \right)$ be a stochastic basis and let $W$ be a given $\Fb$-adapted cylindrical Wiener process with reproducing kernel Hilbert space $\Uc$. Let $U_0 \in L^2\left(\Omega; \Fc_0, V\right)$ be an $\Fc_0$-measurable random variable.
\begin{enumerate}
	\item A pair $(U, \tau)$ is called a \emph{local pathwise solution} if $\tau$ is an $\Fb$-stopping time and $U(\cdot \wedge \tau)$ is a $V$-valued progressively measurable stochastic process satisfying \eqref{eq:solution.regularity} and \eqref{eq:solution.def}.
	\item A couple $\left(U, \xi\right)$ is called a \emph{maximal pathwise solution} if there exists an increasing sequence $\lbrace \tau_N \rbrace_{N = 1}^\infty$ of $\Fb$-adapted stopping times such that $\tau_N \nearrow \xi$ $\Pb$-a.s., for all $N \in \Nb$ the couple $\left(U|_{[0, \tau_N]}, \tau_N \right)$ is a local pathwise solution and for all local pathwise solutions $(\hat{U}, \hat{\tau})$ we have $\hat{\tau} \leq \xi$ a.s.\ and the solutions $U$ and $\hat{U}$ coincide on $\left[0, \hat{\tau}\right]$, that is $U|_{\left[0, \hat{\tau}\right]} = \hat{U}$.
	\item A maximal pathwise solution $\left(U, \xi\right)$ is called \emph{global} if $\xi = +\infty$ $\Pb$-almost surely.
\end{enumerate}
\end{definition}

In the following Section we will study a modified problem  \eqref{eq:pe.modified} similar to \eqref{eq:pe.abstract}. The definitions of solutions of the modified equation \eqref{eq:pe.modified} remain the same as the definitions above with obvious adjustments.

\subsection{Main results and an example}
\label{sect:main.results}

We remark that due the abstract form of the proof of Theorem \ref{thm:maximal.existence} the maximal existence result also holds for the \emph{physical} boundary conditions from \cite{debussche2011}, see also \cite{kukavica2007}, \cite{petcu2009} and \cite{glatt-holtz17} for a more complicated setting. The physical boundary conditions read as follows:
\begin{align*}
	\text{on} &\ \Gamma_i: & &\nu_v \partial_z v = 0, \quad w = 0, \quad \nu_T \partial_z T + \alpha_T T = 0,\\
	\text{on} &\ \Gamma_l: & &v = 0, \quad w = 0,\quad \partial_{\vec{n}_H} T = 0,\\
	\text{on} &\ \Gamma_b: & &v = 0, \quad w = 0, \quad \partial_z T = 0.
\end{align*}

\begin{definition}
We say that $\sigma$ satisfies the hypothesis $H_p$ for $p \geq 2$ if \eqref{eq:sigma.bnd.H}-\eqref{eq:sigma.lip.V} hold with
\begin{equation}
	\label{eq:small.constants.maximal}
	\eta_1 < \frac{1}{p\left( 1 + C^2_{BDG} \right) - 1} \wedge 10^{2/p-1} \qquad \text{and} \qquad \gamma < \frac{2}{C_{BDG}^2},
\end{equation}
where $C_{BDG}$ is the constant from the Burkholder-Davis-Gundy inequality \eqref{eq:bdg}.
\end{definition}

In both theorems below let $\Sc = \left(\Omega, \Fc, \Fb, \Pb \right)$ be a stochastic basis and let $W$ be a given $\Fb$-adapted cylindrical Wiener process with reproducing kernel Hilbert space $\Uc$.

\begin{theorem}
\label{thm:maximal.existence}
Let $U_0 \in L^2\left(\Omega; \Fc_0, V\right)$ be a $V$-valued random variable and let $F_U$ satisfy \eqref{eq:f.source}. Let $\sigma$ satisfy Hypothesis $H_4$, see \eqref{eq:small.constants.maximal}. Then there exists a unique maximal pathwise solution $\left( U, \xi \right)$ of \eqref{eq:pe.abstract}. The maximal solution also satisfies
\begin{equation}
	\label{eq:explosion}
	\sup_{t \in \left[0, \xi\right)} \Vert U \Vert^2 + \int_0^{\xi} \vert AU \vert^2  \, dt = \infty \qquad \Pb\text{-a.s.\ on} \ \lbrace \xi < \infty \rbrace.
\end{equation}

Moreover, if $U_0 \in L^p\left( \Omega; \Fc_0, V \right)$ for some $p > 2$ and $\sigma$ satisfies Hypothesis $H_{\max \lbrace p, 4 \rbrace}$, then for all $N \in \Nb$ and $t \geq 0$
\begin{equation}
	\label{eq:solution.regularity.p}
	U\left( \cdot \wedge \tau_N \right) \in L^p \left( \Omega, C\left( [0, t], V \right) \right), \qquad \mathds{1}_{[0, \tau_N]} \vert AU \vert^2 \Vert U \Vert^{p-2} \in L^1\left( \Omega, L^1\left( 0, t \right) \right).
\end{equation}
\end{theorem}

\begin{theorem}
\label{thm:global.existence}
Let $U_0 \in L^6\left(\Omega; \Fc_0, V \right)$ be a $V$-valued random variable and let $F_U = (F_v, F_T)$ satisfy both \eqref{eq:f.source} and
\[
	F_T \in L^2\left( \Omega, L^2\left( 0, t; L^2\left( \Gamma_i \right) \right) \right)
\]
for all $t > 0$. Let $\sigma$ satisfy Hypothesis $H_6$, see \eqref{eq:small.constants.maximal}, and \eqref{eq:sigma.rc.vt}-\eqref{eq:sigma.dz} with constants $\eta_0$, $\eta_2$ and $\eta_3$ such that
\begin{equation}
	\label{eq:small.constants.global}
	\eta_0 < \frac{2}{3 + 2 C_{BDG}^2}, \qquad \eta_2  < \frac{1}{C^2_{BDG} + \tfrac32}, \qquad 2 C^2_{BDG} \eta_3 < \mu \wedge \nu.
\end{equation}
Then there exists a unique global patwhise solution $U$ of \eqref{eq:pe.abstract}.
\end{theorem}

The necessity of the additional regularity of $F_T$ will be made apparent in the proof of Proposition \ref{prop:T.estimates}. The proofs of the above Theorems can be found in Sections \ref{sec:proof.maximal.existence} and \ref{sec:proof.global.existence}, respectively.

Let us now give two explicit examples of non-trivial noise terms $\sigma$ that satisfy the assumptions of Theorem \ref{thm:maximal.existence} and Theorem \ref{thm:global.existence}. In the following examples let $\Uc = \ell^2(\Nb)$, let $\lbrace e_k \rbrace_{k = 1}^\infty$ be the canonical basis of $\ell^2(\Nb)$ and let $W$ be a cylindrical Wiener process $W$ with reproducing kernel Hilbert space $\Uc$. For simplicity we consider only the term $\sigma_1$.

\begin{example}
\label{example.1}
Let $\phi_k \in C^1\left(\overline{\Mc}, \Rb^{2}\right)$, $\psi_k \in C^1\left(\overline{\Mc}\right)$ and $\chi_k \in V_1$ be such that
\begin{equation*}
	\sum_{k = 1}^\infty \vert \phi_k \vert_{L^\infty}^2 + \vert \psi_k \vert^2_{L^\infty} = \theta_0^2, \quad \sum_{k = 1}^{\infty} \vert \grad_3 \phi_k \vert_{L^\infty}^2 + \vert \grad_3 \psi_k \vert^2_{L^\infty} = \theta_1^2, \quad \sum_{k=1}^\infty \Vert \chi_k \Vert^2 = \kappa^2,
\end{equation*}
for some $\theta_1, \theta_2, \kappa \geq 0$. Let $\alpha_k \geq 0$ be such that $\sum_{k=1}^\infty \alpha_k^2 = \alpha^2$ for some $\alpha \geq 0$. Let us define
\begin{equation}
	\label{example:sigma.maximal}
	\sigma_1(v) \zeta = \sum_{k=1}^\infty \zeta_k \left[ \left( \phi_k \cdot \grad \right) v + \psi_k \partial_z v + \alpha_k v + \chi_k \right], \qquad \zeta = \lbrace \zeta_k \rbrace_{k=1}^\infty \in \Uc.
\end{equation}
Then $\sigma_1$ satisfies
\[
	\Vert \sigma_1(v) \Vert^2_{L_2(\Uc, H_1)} = \sum_{k=1}^\infty \vert (\phi_k \cdot \grad) v + \psi_k \partial_z v + \alpha_k v + \chi_k \vert_{H_1}^2 \leq C \left( \theta_0^2 \Vert v \Vert^2 + \alpha^2 \vert v \vert^2 + \kappa^2 \right), \quad v \in V_1.
\]
Therefore \eqref{eq:sigma.bnd.H} holds and if $\theta_0$ is sufficiently small the condition on $\eta_0$ from \eqref{eq:small.constants.global} can be met as well. One can also show that $\sigma_1$ satisfies
\[
	\Vert \sigma_1(v) \Vert_{L_2(\Uc, V_1)}^2 \leq C \left[ \left( \theta_0^2 + \alpha^2 \right) \Vert v \Vert^2 + \theta_1^2 \vert A_1 v \vert^2 + \kappa^2\right], v \in D(A_1).
\]
Thus \eqref{eq:sigma.bnd.V} holds. If one assumes that $\theta_1$ is sufficiently small, the assumption on $\eta_1$ in \eqref{eq:small.constants.maximal} can be satisfied. The Lipschitz continuity properties can be checked in a straightforward manner.
\end{example}

We emphasize that although the smallness of coefficients may seem to be overly restrictive, it contains the class of noise terms for which the uniqueness of invariant measures is typically established, see also \cite[Remark 1.4]{glatt-holtz2014}.
	
\begin{example}
\label{example.2}
Due to the additional structural assumptions \eqref{eq:sigma.rc.vt}, \eqref{eq:sigma.rc.t} and \eqref{eq:sigma.grad.ac} that come from the estimates on the baroclinic and barotropic modes, see the proofs in Section 4, the previous example \eqref{example:sigma.maximal} does not in general meet the assumptions of Theorem \ref{thm:global.existence}. Let $\phi_k \in C^1\left(\overline{\Mc}, \Rb^2 \right)$ be such that $\Ac_3 \phi_k = \phi_k$ for all $k \in \Nb$, that is $\phi_k$ is independent of $z$. Moreover, let $\chi_k \in V$ and let
\[
	\sum_{k = 1}^\infty \vert \psi_k \vert_{L^\infty}^2 = \theta_0^2, \qquad \sum_{k = 1}^\infty \vert \grad \phi_k \vert^2_{L^\infty} \leq \theta_1^2, \qquad \sum_{k = 1}^\infty \Vert \chi_k \Vert^2 = \kappa^2,
\]
for some $\theta_0, \theta_1, \kappa > 0$. Let $\alpha, \alpha_k \geq 0$ be as in Example \ref{example.1} and let us define
\begin{equation*}
	\sigma_1(v) \zeta = \sum_{k = 1}^\infty \zeta_k \left( \phi_k \grad \Ac_3 v + \alpha_k v + \chi_k \right), \qquad \zeta = \lbrace \zeta_k \rbrace_{k = 1}^\infty \in \Uc.	
\end{equation*}
Then since for all $k \in \Nb$ we have $\Rc \phi_k = (I - \Ac_3) \phi_k = 0$, we readily observe that for $\zeta \in \Uc$
\begin{align*}
	\Rc \sigma_1(v) \zeta &= \sum_{k=1}^\infty \zeta_k \Rc \left[ \left( \phi_k \cdot \grad \right) \Ac_3 v  + \alpha_k v + \chi_k \right] = \sum_{k=1}^\infty \zeta_k \left[ \alpha_k \Rc v + \Rc \chi_k \right]\\
	\Ac_2 \sigma_1(v) \zeta &= \sum_{k=1}^\infty \zeta_k \Ac_2 \left[ \left( \phi_k \cdot \grad \right) \Ac_2 v + \alpha_k v + \chi_k \right] = \sum_{k=1}^\infty \zeta_k \left[ \left( \phi_k \cdot \grad \right) \Ac_2 v + \alpha_k \Ac_2 v + \Ac_2 \chi_k \right]
\end{align*}
and thus
\begin{align*}
	\sum_{k = 1}^\infty \vert \Rc \sigma_1(v) e_k \vert^2_{L^6} &= \sum_{k=1}^\infty \vert \alpha_k \Rc v + \Rc \chi_k \vert_{L^6}^2 \leq C \left( \alpha^2 \vert \Rc v \vert_{L^6}^2 + \kappa^2 \right),\\
	\Vert \Ac \sigma_1(v) \Vert^2_{L_2(\Uc, \Vo)} &= \sum_{k=1}^\infty \vert A^{1/2} P_{H_1} \left[ \left( \phi_k \cdot \grad \right) \Ac v + \alpha_k \Ac v + \Ac \chi_k \right] \vert_{H_1}^2\\
	&\leq C \left[ \left( \theta_1^2 + \alpha^2 \right) \Vert \Ac v \Vert_{\Vo}^2 + \theta_0^2 \vert A_S \Ac v \vert_{\Ho}^2 + \kappa^2 \right],
\end{align*}
which immediately gives \eqref{eq:sigma.rc.vt}. We observe that choosing $\alpha$, $\theta_0$ and $\theta_1$ sufficiently small, \eqref{eq:sigma.grad.ac}, then the condition for $\eta_2$ in \eqref{eq:small.constants.global} can be satisfied. Checking that \eqref{eq:sigma.dz} and the rest of \eqref{eq:small.constants.global} can be met by a suitable choice of $\alpha$, $\theta_0$ and $\theta_1$ is straightforward, although in this particular case we have $\eta_3 = 0$.
\end{example}

\section{Existence of maximal solutions}
\label{sect:maximal.solution}

In the whole section we assume that the stochastic basis $\Sc$, the cylindrical Wiener process $W$, the force term $F_U$ and the noise term $\sigma$ are as in Theorem \ref{thm:maximal.existence}.

\subsection{Approximation scheme}

Employing the technique from \cite{debussche2011}, we first focus on following the modified equation
\begin{equation}
	\label{eq:pe.modified}
	dU + [AU + \theta(\Vert U - U_\ast \Vert)B(U, U) + F(U)] \, dt = \sigma(U) \, dW, \qquad U(0) = U_0,
\end{equation}
where $U_\ast$ is the solution of the random linear differential equation
\begin{equation}
	\label{eq:linear}
	\tfrac{d}{dt} U_\ast(t) + AU_\ast(t) = 0, \qquad U_\ast(0) = U_0,
\end{equation}
and $\theta \in C^\infty(\Rb)$ is such that, for some $\kappa > 0$ determined later,
\begin{equation}
	\label{eq:theta}
	\mathds{1}_{[-\kappa/2, \kappa/2]} \leq \theta \leq \mathds{1}_{[-\kappa, \kappa]}.
\end{equation}
Similarly as in \cite{debussche2011}, we could state the result in an abstract setting. For the existence of a local martingale solution, it would be sufficient to assume that $F$ satisfies only the linear growth assumption \eqref{eq:F.bounded}.

The Galerkin approximations of the modified equation  \eqref{eq:pe.modified} solve the equation
\begin{gather}
	\label{eq:galerkin}
	dU^n + \left[AU^n + \theta\left(\Vert U^n - U^n_\ast \Vert\right) B^n\left(U^n\right) + F^n\left(U^n\right)\right] \, dt = \sigma^n\left(U^n\right) \, dW,\\
	U^n(0) = P_n U_0 \equiv U^n_0,
\end{gather}
where $U^n_\ast$ is the solution of the linear equation
\begin{equation}
	\label{eq:linear.n}
	\tfrac{d}{dt}U^n_\ast(t) + AU^n_\ast(t) = 0, \qquad U^n_\ast(0) = U^n_0.
\end{equation}
and
\[
	F^n = P_n F, \qquad B^n = P_n B, \qquad \sigma^n = P_n \sigma.
\]
Since the equation \eqref{eq:galerkin} is as an SDE in a finite dimensional Hilbert space and all the non-linear terms in \eqref{eq:galerkin} are locally Lipschitz, local existence and uniqueness of $U^n$ follow by a standard argument. The fact that the solutions are global will follow from estimates from Lemma \ref{lemma:approximation.estimates}. The linear equation \eqref{eq:linear.n} is well-posed, see e.g.\ \cite{lions1972}, and the solution $\una$ satisfies for $p \geq 2$
\begin{gather}
	\nonumber
	\vert \una \vert_{W^{1, 2}\left(0, t; H\right)}^2 \leq C \Vert U_0 \Vert^2,\\
	\label{eq:una.estimate1}
	\sup_{s \in \left[0, t\right]} \Vert \una \Vert^p + \int_0^t \vert A\una \vert^2 \Vert \una \Vert^{p-2} \, ds + \left( \int_0^t |A \una|^2 \, ds \right)^{p/2} \leq C \Vert U_0 \Vert^p,
\end{gather}
where the constants are independent of $n \in \Nb$ and therefore has the regularity
\[
	\una \in C\left([0, t]; V\right) \cap L^2\left(0, t; D(A)\right), \quad \tfrac{d}{dt} \una \in L^2\left(0, t; H \right), \qquad t > 0.
\]

In the next lemma we will establish the boundedness of the Galerkin approximations $U^n$. The lemma and its proof are similar to \cite[Lemma 3.1]{debussche2011}. We include a full proof to demonstrate the dependence on $\eta_1$ in \eqref{eq:small.constants.maximal}.

\begin{lemma}
\label{lemma:approximation.estimates}
Let $t > 0$, $p \geq 2$, $\alpha \in [0, 1/2)$ and let $U_0 \in L^q\left( \Omega, V \right)$ be an $\Fc_0$-measurable random variable with $q \geq \max \lbrace 2p, 4 \rbrace$. Let $F$ satisfy \eqref{eq:F.bounded} and let $\sigma$ satisfy the Hypothesis $H_p$, see \eqref{eq:small.constants.maximal}. Then there exist $\kappa > 0$, see \eqref{eq:theta}, and $K > 0$ such that for all $n \in \Nb$ we have
\begin{gather}
	\label{eq:un.Vp}
	\Eb \left[ \sup_{s \in \left[0, t \right]} \Vert U^n \Vert^p + \int_0^t \vert A U^n \vert^2 \Vert U^n \Vert^{p-2} \, ds + \left( \int_0^t \vert A U^n \vert^2 \, ds \right)^{p/2}  \right] \leq K,\\
	\label{eq:un.fract}
	\Eb \left| \int_0^\cdot \sigma^n(U^n) \, dW \right|^p_{W^{\alpha, p}(0, t; H)} \leq K.
\end{gather}
Moreover, if $p \geq 4$ then also for all $n \in \Nb$
\begin{equation}
	\label{eq:un.diff.est}
	\Eb\left| U^n(\cdot) - \int_0^\cdot \sigma^n(U^n) \, dW \right|^2_{W^{1, 2}\left(0, t; H\right)} \leq K.
\end{equation}
\end{lemma}

\begin{proof}
Let $\bun = U^n - \una$ and therefore $\bun$ satisfies the following SDE
\begin{gather*}
	d\bun + \left[A\bun + \theta\left(\Vert \bun \Vert\right)B^n\left(\bun + \una\right) + F^n\left(\bun+\una\right)\right] \, dt = \sigma^n\left(\bun+\una\right) \, dW,\\
	\bun(0) = 0.
\end{gather*}

First, for $p \geq 2$ let us prove that there exists $C > 0$ such that for all $n \in \Nb$
\begin{equation}
	\label{eq:bun.estimate}
	\Eb \left[ \sup_{s \in \left[0, t \right]} \Vert \bun \Vert^p + \int_0^t \Vert \bun \Vert^{p-2} \vert A \bun \vert^2 \, ds \right] \leq C \left( 1 + \Eb \Vert U_0 \Vert^{\max \lbrace p, 4 \rbrace} \right).
\end{equation}
Recalling that $A$ is self-adjoint, from the finite dimensional It\^{o} Lemma we infer that
\begin{align}
	\nonumber
	d\Vert& \bun \Vert^p + p \Vert \bun \Vert^{p-2} \vert A\bun \vert^2 \, dt\\
	\nonumber
	&= -p\Vert \bun \Vert^{p-2} \left(A \bun,  \theta\left(\Vert \bun \Vert\right) B^n\left(\bun + \una\right)\right) \, dt\\
	\nonumber
	&\hphantom{=\ } -p\Vert \bun \Vert^{p-2} \left(A\bun, F^n\left(\bun + \una\right)\right) \, dt + p\Vert \bun \Vert^{p-2} \left(A^{1/2} \bun, A^{1/2}\sigma\left(\bun+\una\right) \, dW\right)\\
	\nonumber
	&\hphantom{=\ } + \tfrac{p}{2} \Vert \bun \Vert^{p-2} \Vert\sigma^n\left(\bun + \una\right) \Vert^2_{L_2\left(\Uc, V\right)} \, dt\\
	\nonumber
	&\hphantom{=\ } + \tfrac{p(p-2)}{2}\Vert \bun \Vert^{p-4} \tr \left[ \left(A^{1/2} \bun \otimes A^{1/2} \bun\right) \left(A^{1/2} \sigma^n\left(\bun + \una\right)\right) \left(A^{1/2} \sigma^n\left(\bun + \una\right)\right)^\ast \right] \, dt\\
	\label{eq:difference.ito}
	&= J^p_1 \, dt + J^p_2 \, dt + J^p_3 \, dW + J^p_4 \, dt + J^p_5 \, dt.
\end{align}
To estimate $J^p_1$, we first use the bilinearity of $B$ to get
\begin{align*}
	\left|J^p_1 \right| &= p \Vert \bun \Vert^{p-2} \left| \left(A\bun, \theta\left(\Vert \bun \Vert\right) \left[B^n\left(\una\right) +  B^n\left(\bun, \una\right) + B^n\left(\una, \bun\right) + B^n\left(\bun\right)\right]\right) \right|\\
	&\leq J^p_{1, 1} + J^p_{1, 2} + J^p_{1, 3} + J^p_{1, 4}.
\end{align*}
By the bound \eqref{eq:b.estimate2} on $B$ , the Young inequality and the property \eqref{eq:theta} of the function $\theta$ we have
\begin{align}
	\nonumber
	J^p_{1, 1} &\leq p c_b \Vert \bun \Vert^{p-2} \theta\left(\Vert \bun \Vert\right) \Vert \una \Vert \vert A\una \vert \ \vert A \bun \vert\\
	\nonumber
	&\leq \tfrac{p\varepsilon}{4} \Vert \bun \Vert^{p-2} \vert A\bun \vert^2 + C_\varepsilon \theta\left(\Vert \bun \Vert\right)^2 \Vert \bun \Vert^{p-2} \Vert \una \Vert^2 \vert A \una \vert^2\\
	\label{eq:Jp11}
	&\leq \tfrac{p\varepsilon}{4} \Vert \bun \Vert^{p-2} \vert A\bun \vert^2 + C_\varepsilon \kappa^{p-2} \Vert \una \Vert^2 \vert A\una \vert^2
\end{align}
for some $\varepsilon > 0$ small determined later. We deal with $J^p_{1, 2}$ and $J^p_{1, 3}$ in a similar way by estimating
\begin{align}
	\nonumber
	J^p_{1, 2} + J^p_{1, 3} &\leq 2 p c_b \theta\left(\Vert \bun \Vert\right) \Vert \bun \Vert^{p-2} \Vert \bun \Vert^{1/2} \vert A \bun \vert^{3/2} \Vert \una \Vert^{1/2} \vert A \una \vert^{1/2}\\
	\nonumber
	&\leq \tfrac{p\varepsilon}{4} \Vert \bun \Vert^{p-2} \vert A \bun \vert^2 + C_\varepsilon \theta\left(\Vert \bun \Vert\right)^4 \Vert \bun \Vert^p \Vert \una \Vert^2 \vert A \una \vert^2\\
	&\leq \tfrac{p\varepsilon}{4} \Vert \bun \Vert^{p-2} \vert A \bun \vert^2 + C_\varepsilon \kappa^p \Vert \una \Vert^2 \vert A \una \vert^2.
\end{align}
Using the estimate \eqref{eq:b.estimate1} on $B$ and the property \eqref{eq:theta} of the function $\theta$ we deduce
\begin{equation}
	\label{eq:Jp14}
	J^p_{1, 4} \leq p c_b \theta\left(\Vert \bun \Vert\right) \Vert \bun \Vert^{p-1} \vert A \bun \vert^2 \leq p c_b \kappa \Vert \bun \Vert^{p-2} \vert A\bun \vert^2.
\end{equation}
Collecting the estimates \eqref{eq:Jp11}--\eqref{eq:Jp14} we get
\begin{equation}
	\label{eq:Jp1}
	\left| J^p_1 \right| \leq \left(\tfrac{p\varepsilon}{2} + p c_b \kappa\right)\Vert \bun \Vert^{p-2} \vert A \bun \vert^2 + C_\varepsilon\kappa^p \Vert \una \Vert^2 \vert A \una \vert^2.
\end{equation}
Concerning $J^p_2$, employing the growth assumption \eqref{eq:F.bounded} on $F$ and the Young inequality we infer
\begin{align}
	\nonumber
	\int_0^t \left| J^p_2 \right| \, ds &\leq p C \int_0^t \Vert \bun \Vert^{p-2} \vert A \bun \vert \left(1+ \Vert\bun \Vert + \Vert \una \Vert \right) \, ds\\
	\label{eq:Jp2}
	&\leq \frac{p\varepsilon}{2} \int_0^t \vert A \bun \vert^2 \Vert \bun \Vert^{p-2} \, ds + C_\varepsilon \int_0^t 1 + \Vert \bun \Vert^p + \Vert \una \Vert^p \, ds.
\end{align}
We estimate $J^p_4$ and $J^p_5$ together by the assumption \eqref{eq:sigma.bnd.V} on the growth of $\sigma(U)$ in $L_2\left( \Uc, V \right)$ and the Young inequality. We obtain
\begin{align}
	\nonumber
	\int_0^t \left| J^p_4 + J^p_5 \right| \, ds &\leq \frac{p(p-1)}{2} \int_0^t \Vert \bun \Vert^{p-2} \Vert \sigma^n\left(\bun + \una\right) \Vert^2_{L_2\left(\Uc, V\right)} \, ds\\
	\nonumber
	&\leq \frac{p(p-1)}{2} \int_0^t \Vert \bun \Vert^{p-2} \left[ C \left( 1 + \Vert \bun + \una \Vert^2 \right) + \eta_1 \vert A\bun + A \una \vert^2 \right] \, ds\\
	\nonumber
	&\leq C_{\varepsilon}\left( \int_0^t 1 + \Vert \bun \Vert^p + \Vert \una \Vert^p \, ds \right) + p(p-1) \eta_1 \int_0^t \Vert \bun \Vert^{p-2} \vert A \bun \vert^2 \, ds\\
	&\hphantom{\leq\ } + \frac{\varepsilon}{3} \sup_{s \in \left[0, t\right]} \Vert \bun \Vert^p + C_\varepsilon \left( \int_0^t \vert A \una \vert^2 \, ds \right)^{p/2}
\end{align}
for some $\varepsilon > 0$ precisely determined later. The stochastic term is dealt with using the Burkholder-Davis-Gundy inequality \eqref{eq:bdg}, the Young inequality and the bound \eqref{eq:sigma.bnd.V} on the growth of $\sigma(U)$ in $L_2\left( \Uc, V \right)$. For a similar kind of argument see \cite[Theorem 1.1]{brzezniak2000}. We deduce
\begin{align*}
	\Eb \sup_{s \in \left[0, t \right]} \left| \int_0^s J^p_3 \, dW \right| &\leq p C_{BDG} \Eb \left( \int_0^t \Vert \bun \Vert^{2(p-1)} \Vert \sigma\left(\bun + \una\right) \Vert_{L_2\left(\Uc, V\right)}^2 \, ds \right)^{1/2}\\
	&\leq C \Eb \left( \int_0^t \Vert \bun \Vert^{2(p-1)} \left(1 + \Vert \bun \Vert^2 + \Vert \una \Vert^2 \right) \, ds \right)^{1/2}\\
	&\hphantom{\leq \ } + p C_{BDG} \sqrt{\eta_1} \Eb \left( \int_0^t \Vert \bun \Vert^{2(p-1)} \vert A\bun + A \una \vert^2 \, ds \right)^{1/2}\\
	&= \Eb J_{3, 1}^p + \Eb J_{3, 2}^2.
\end{align*}
By the Young inequality it is straightforward to obtain
\begin{equation}
	\Eb J_{3, 1}^p \leq C_\varepsilon \Eb \left[ \int_0^t 1 + \Vert \bun \Vert^p \, ds + \sup_{s \in [0, t]} \Vert \una \Vert^p \right] + \frac{\varepsilon}{3} \Eb \sup_{s \in \left[0, t \right]} \Vert \bun \Vert^p.
\end{equation}
Regarding $J_{3, 2}^p$ we again use the Young inequality and get
\begin{equation*}
	\begin{split}	
		\Eb J_{3, 2}^p &\leq p C_{BDG} \sqrt{2 \eta_1} \Eb \left( \int_0^t \Vert \bun \Vert^{2(p-1)} \left( \vert A \una \vert^2 + \vert A \bun \vert^2 \right) \, ds \right)^{1/2}\\
		&\leq \frac{\varepsilon}{3} \Eb \sup_{s \in \left[0, t\right]} \Vert \bun \Vert^p + C_\varepsilon \Eb \left( \int_0^t \vert A \una \vert^2 \, ds \right)^{p/2}\\
		&\hphantom{\leq \ } + \frac{p C_{BDG} \sqrt{2 \eta_1} \hat{\varepsilon}}{2} \Eb \sup_{s \in [0, t]} \Vert \bun \Vert^p + \frac{p C_{BDG} \sqrt{2 \eta_1}}{2 \hat{\varepsilon}} \Eb \int_0^t \Vert \bun \Vert^p \vert A \bun \vert^2 \, ds
	\end{split}
\end{equation*}
for some $\hat{\varepsilon} > 0$ small determined later. Denoting
\[
	\frac{p C_{BDG}\sqrt{2 \eta_1} \hat{\varepsilon}}{2} = 1 - \delta
\]
for some $\delta \in (\varepsilon, 1)$ precisely determined later we rewrite the above as
\begin{equation}
	\label{eq:Jp32}
	\begin{split}
		\Eb J^p_{3, 2} &\leq \left( \frac{\varepsilon}{3} + 1 - \delta \right) \Eb \sup_{s \in \left[0, t \right]} \Vert \bun \Vert^p + C_\varepsilon \Eb \left( \int_0^t \vert A \una \vert^2 \, ds \right)^{p/2}\\
		&\hphantom{\leq \ }  + \frac{p^2 C_{BDG}^2 \eta_1}{1-\delta} \Eb \int_0^t \Vert \bun \Vert^{p-2} \vert A \bun \vert^2 \, ds.
	\end{split}
\end{equation}
Collecting the estimates \eqref{eq:Jp1}-\eqref{eq:Jp32}, we integrate the equation \eqref{eq:difference.ito} in time and apply the expected value to get
\begin{multline*}
	\left( \delta - \varepsilon \right) \Eb  \sup_{s \in \left[0, t \right]} \Vert \bun \Vert^p\\
	+ p\left(1 - \varepsilon - c_b \kappa - (p-1) \eta_1 - \frac{p C_{BDG}^2 \eta_1}{1-\delta} \right) \Eb \int_0^t \vert A \bun \vert^2 \Vert \bun \Vert^{p-2} \, ds\\
	\leq C \Eb \left[ \Vert U_0^n \Vert^p + \sup_{s \in \left[0, t \right]} \Vert \una \Vert^p + \int_0^t 1 + \Vert \bun \Vert^p + \Vert \una \Vert^2 \vert A \una \vert^2 \, ds + \left( \int_0^t \vert A \una \vert^2 \, ds \right)^{p/2} \right].
\end{multline*}
Recalling that $\sigma$ satisfies Hypothesis $H_p$, see \eqref{eq:small.constants.maximal}, and the estimate \eqref{eq:una.estimate1} on $\una$ holds, we choose $\delta, \kappa = \kappa_1,\varepsilon > 0$ sufficiently small (in this order) and invoke the standard Gronwall Lemma to obtain
\begin{align*}
	\Eb \bigg[ &\sup_{s \in \left[0, t\right]} \Vert \bun \Vert^p + \int_0^t \Vert \bun \Vert^{p-2} \vert A \bun \vert^2 \, ds \bigg]\\
	&\leq C \Eb \left[ 1 + \Vert U^n_0 \Vert^p + \sup_{s \in \left[0, t\right]} \Vert \una \Vert^p + \int_0^t \Vert \una \Vert^2 \vert A\una \vert^2 \, ds + \left( \int_0^t \vert A \una \vert^2 \, ds \right)^{p/2} \right]\\
	&\leq C\Eb \left[ 1 + \Vert U_0 \Vert^{\max \lbrace p, 4 \rbrace} \right],
\end{align*}
which finishes the proof of \eqref{eq:bun.estimate}.

Next, we want to prove that for $p \geq 2$ there exists $C > 0$ such that for all $n \in \Nb$
\begin{equation}
	\label{eq:bun.estimate.exp}
	\Eb \left( \int_0^t \vert A \bun \vert^2 \, ds \right)^{p/2} \leq C \Eb \left[ 1 + \Vert U_0 \Vert^{\max \lbrace p, 4 \rbrace} + \Vert U_0 \Vert^{2p} \right].
\end{equation}
Returning to \eqref{eq:difference.ito} with $p=2$, we integrate in time and apply the expected value to get
\begin{align}
	\nonumber
	\Eb \left( \int_0^t \vert A \bun \vert^2 \, ds \right)^{p/2} &\leq 5^{(p-2)/2} \Eb \Bigg[ \Vert \bun \Vert^p + \left( \int_0^t \left| \theta\left(\Vert \bun \Vert\right) \left(B^n\left(\bun + \una\right), A \bun\right) \right| \, ds \right)^{p/2}\\
	\nonumber
	&\hphantom{\qquad}+\left( \int_0^t \left| \left( F\left(\bun + \una\right), A \bun \right) \right| \, ds \right)^{p/2} + \sup_{s \in [0, t]} \left| \int_0^s J^2_3 \, dW \right|^{p/2}\\
	\nonumber
	&\hphantom{\qquad}+ \left( \int_0^t \tfrac12 \Vert \sigma_n\left(\bun + \una\right) \Vert^2_{L_2\left(\Uc, V\right)} \, ds \right)^{p/2} \Bigg]\\
	\label{eq:AU2.1}
	&= 5^{(p-2)/2} \Eb \left[ \Vert \bun_0 \Vert^p + I_1 + I_2 + I_3 + I_4 \right].
\end{align}
Using the previously established bound \eqref{eq:Jp1} we get
\begin{equation}
	|I_1| \leq 2^{(p-2)/2} \left(\varepsilon + c_b \kappa\right)^{p/2} \left( \int_0^t \vert A \bun \vert^2 \, ds \right)^{p/2} + C_{\varepsilon} \left( \int_0^t \Vert \una \Vert^2 \vert A \una \vert^2 \, ds \right)^{p/2}
\end{equation}
for some $\varepsilon > 0$ precisely determined later. Similarly to \eqref{eq:Jp2} we estimate

\begin{equation}
	|I_2| \leq \frac{\varepsilon}{2} \left( \int_0^t \vert A \bun \vert^2 \, ds \right)^{p/2} + C_{\varepsilon} \left( 1 + \int_0^t \Vert \bun \Vert^p + \Vert \una \Vert^p \, ds \right).
\end{equation}
By the assumption \eqref{eq:sigma.bnd.V} on the $L_2\left( \Uc, V\right)$ norm of $\sigma(U)$ we have
\begin{align}
	\nonumber
	|I_4| &\leq 2^{(p-2)/2} \left| \int_0^t C \left(1 + \Vert \bun \Vert^p + \Vert \una \Vert^p \right) + 2 \eta_1 \vert A \una \vert^2 \, ds \right|^{p/2}\\
	\nonumber
	&\hphantom{\leq \ } + 2^{(p-2)/2} \left| \int_0^t 2 \eta_1 \vert A \bun \vert^2 \, ds \right|^{p/2}\\
	&\leq C \left[ \int_0^t 1 + \Vert \bun \Vert^p + \Vert \una \Vert^p \, ds + \left( \int_0^t \vert A \una \vert^2 \, ds \right)^{p/2} \right]+ 2^{p-1} \eta_1 \left( \int_0^t \vert A \bun \vert^2 \, ds \right)^{p/2}.
\end{align}
Regarding the stochastic term $I_3$, we use the Burkholder-Davis inequality \eqref{eq:bdg}, the Young inequality and the bound \eqref{eq:sigma.bnd.V} again to obtain the estimate
\begin{align}
	\nonumber
	\Eb &\sup_{s \in \left[0, t\right]} \left| \int_0^s J^2_3 \, dW \right|^{p/2} \leq C_{BDG, p/2} \Eb \left( \int_0^t \Vert \bun \Vert^2 \Vert \sigma^n\left(\bun + \una\right) \Vert^{2}_{L_2\left(\Uc, V\right)} \, ds \right)^{p/4}\\
	\nonumber
	&\leq 2^{p/2} C_{BDG, p/2} \Eb \left( \int_0^t \Vert \bun \Vert^p \left( C\left(1 + \Vert \bun \Vert^2 + \Vert \una \Vert^2 \right) + 2 \eta_1 \left( \vert A \una \vert^2 + \vert A \bun \vert^2 \right) \right) \, ds \right)^{p/4}\\
	\label{eq:J2_3}
	&\leq \frac{\varepsilon}{2} \Eb \left( \int_0^t \vert A \bun \vert^2 \, ds \right)^{p/2} + C_{\varepsilon} \Eb \left[ 1 + \sup_{s \in \left[0, t \right]} \Vert \bun \Vert^{p} + \sup_{s \in [0, t]} \Vert \una \Vert^{p} + \left( \int_0^t \vert A \una \vert^2 \, ds \right)^{p/2} \right].
\end{align}
More precisely, if $p < 4$ we use the fact that the concave function $x \to x^{p/4}$ is sublinear, otherwise we employ discrete H\"{o}lder's inequality. Collecting the estimates \eqref{eq:AU2.1}-\eqref{eq:J2_3} we obtain
\begin{multline*}
	2^{p/2} \Eb \left( \int_0^t \vert A \bun \vert^2 \, ds \right)^{p/2} \leq 5^{(p-2)/2} \Eb \Vert \bun(0) \Vert^p\\
	+ 5^{(p-2)/2} \left( 2^{(p-2)/2} \left( \varepsilon + c_b \kappa \right)^{p/2} + \varepsilon + 2^{p-1} \eta_1^{p/2} \right) \Eb \left( \int_0^t \vert \bun \vert^2 \, ds \right)^{p/2}\\
	+ C_\varepsilon \Eb \left[ 1 + \sup_{s \in \left[0, t\right]} \Vert \bun \Vert^p + \sup_{s \in [0, t]} \Vert \una \Vert^p + \left( \int_0^t \Vert \una \Vert^2 \vert A \una \vert^2 \, ds \right)^{p/2} + \left( \int_0^t \vert A \una \vert^2 \, ds \right)^{p/2}\right].
\end{multline*}
Recalling that $\sigma$ satisfies Hypothesis $H_p$, see \eqref{eq:small.constants.maximal}, and assuming that $\kappa = \kappa_2 > 0$ is sufficiently small, we may choose $\varepsilon > 0$ sufficiently small and invoke the estimate \eqref{eq:una.estimate1} on $\una$ and the previously established bound \eqref{eq:bun.estimate} to prove \eqref{eq:bun.estimate.exp}. Let $\kappa = \kappa_1 \wedge \kappa_2$.

Finally we are ready to prove the original claim \eqref{eq:un.Vp}. Recalling $q \geq \max \lbrace 2p, 4 \rbrace$ we estimate
\begin{align*}
	\Eb \Bigg[ \sup_{s \in \left[0, t \right]} &\Vert U^n \Vert^p + \int_0^t \vert A U^n \vert^2 \Vert U^n \Vert^{p-2} \, ds \Bigg]\\
	&\leq \Eb \left[ \sup_{s \in \left[0, t \right]} \Vert U^n \Vert^p + \left( \sup_{s \in \left[0, t \right]}\Vert U^n \Vert^{p-2} \right) \int_0^t \vert A U^n \vert^2 \, ds \right]\\
	&\leq C \Eb \left[ \sup_{s \in \left[0, t\right]} \Vert U^n \Vert^p + \left( \int_0^t \vert A U^n \vert^2 \, ds \right)^{p/2} \right]\\
	&\leq C \Eb \left[  \sup_{s \in \left[0, t\right]} \Vert \bun \Vert^p + \sup_{s \in \left[0, t\right]} \Vert \una \Vert^p + \left( \int_0^t \vert A \bun \vert^2 \, ds \right)^{p/2} + \left( \int_0^t \vert A \una \vert^2 \, ds \right)^{p/2} \right]\\
	&\leq C \Eb \left[ 1 + \Vert U_0 \Vert^{\max \lbrace p, 4 \rbrace} + \Vert U_0 \Vert^{2p} \right].
\end{align*}

The bound \eqref{eq:un.fract} is proved using the fractional Burkholder-Davis-Gundy inequality \eqref{eq:bdg.frac}, the bound \eqref{eq:sigma.bnd.H} on $\sigma(U)$ in $L_2\left( \Uc, H \right)$ and the estimate \eqref{eq:un.Vp} by
\[
	\Eb \left| \int_0^\cdot \sigma^n\left(U^n\right) \, dW \right|^p_{W^{\alpha, p}\left(0, t; H\right)} \leq C_t \Eb \int_0^t \Vert \sigma^n\left(U^n\right) \Vert^p_{L_2\left(\Uc, H\right)} \, ds \leq  C_t \Eb \left[ 1 + \int_0^t \Vert U^n \Vert^p \, ds \right].
\]

The remaining claim \eqref{eq:un.diff.est} follows from
\[
	U^n(s) - \int_0^s \sigma^n\left(U^n\right) \, dW = U^n(0) - \int_0^s AU^n + \theta\left(\Vert U^n - U^n_\ast \Vert\right) B^n\left(U^n\right) - F^n\left(U^n\right) \, dr,
\]
the bound \eqref{eq:b.estimate2} on $B$ and the assumption \eqref{eq:F.bounded} on $F$ by the estimate
\[
	 \Eb \left\Vert U^n - \int_0^\cdot \sigma^n\left(U^n\right) \, dW \right\Vert^2_{W^{1, 2}\left(0, t; H\right)} \leq C \Eb \left[ 1 + \Vert U^n(0) \Vert^2 + \int_0^t \vert A U^n \vert^2 \left(1 + \Vert U^n \Vert^2\right) \, ds \right],
\]
where the right-hand side is finite by \eqref{eq:un.Vp} with $p = 4$.
\end{proof}

\subsection{Existence of local martingale solutions}
\label{sect:local.martingale.sol}

After obtaining the same bounds on the finite dimensional approximations of \eqref{eq:pe.modified} as in \cite[Lemma 3.1]{debussche2011}, the compactness argument follows similarly. Thus, we concentrate mainly on the differences and omitted parts.

Given an initial distribution $\mu_0$ on $V$, let $U_0$ be an $\Fc_0$-measurable $V$-valued random variable with law $\mu_0$ satisfying the assumptions of Lemma \ref{lemma:approximation.estimates}. Recall that $W$ is an $\Fb$-cylindrical Wiener process with reproducing kernel Hilbert space $\Uc$ and let $\Uc_0$ be an auxiliary Hilbert space such that the embedding $\Uc \hook \Uc_0$ is Hilbert-Schmidt. Let $U^n$ be the solutions to the approximating system \eqref{eq:galerkin} relative to the basis $\Sc$, the cylindrical Wiener process $W$ and the initial condition $U_0$. We define
\begin{equation*}
	\Xc_U = L^2\left(0, t; V\right) \cap C\left(\left[0, t\right]; V'\right), \qquad \Xc_W = C\left(\left[0, t\right]; \Uc_0\right), \qquad \Xc = \Xc_U \times \Xc_W.
\end{equation*}
Let $\mu_U^n$, $\mu^n_W$ and $\mu^n$ be laws of $U^n$, $W$ and $(U^n, W)$ on $\Xc_U$, $\Xc_W$ and $\Xc$, respectively, in other words
\begin{equation}
	\label{eq:mu.measure}
	\mu^n_U(\cdot) = \Pb\left( \left\lbrace U^n \in \cdot \right\rbrace\right), \qquad \mu_W^n(\cdot) = \Pb\left( \left\lbrace W \in \cdot \right\rbrace \right), \qquad \mu^n = \mu_U^n \otimes \mu_W^n.
\end{equation}

The proof of the existence of a global martingale solution to the modified problem \eqref{eq:pe.modified} will be shown once we prove the following two propositions, cf.\ \cite[Proposition 4.1 and Proposition 7.1]{debussche2011}, in the setting of gradient-dependent noise. The first proposition can be proved similarly as in the referenced paper, the second requires a minor modification of the argument.

\begin{proposition}
\label{prop:approximating.sequence}
Let $\mu_0$ be a probability measure on $V$ satisfying \eqref{eq:mu.zero} with $q \geq 8$ and let $(\mu^n)$ be the measures defined in \eqref{eq:mu.measure}. Then there exists a probability space $(\tilde{\Omega}, \tilde{\Fc}, \tilde{\Pb})$, a subsequence $n_k \to \infty$ and a sequence of $\Xc$-valued random variables $(\tilde{U}^{n_k}, \tilde{W}^{n_k})$ such that
\begin{enumerate}
	\item $(\tilde{U}^{n_k}, \tilde{W}^{n_k})$ converges almost surely in $\Xc$ to $(\tilde{U}, \tilde{W}) \in \Xc$,
	\item $\tilde{W}^{n_k}$ is a cylindrical Wiener process with reproducing kernel Hilbert space $\Uc$ adapted to the filtration $\left( \Fc_t^{n_k} \right)_{t \geq 0}$, where $\left( \Fc_t^{n_k} \right)_{t \geq 0}$ is the completion of $\sigma(\tilde{W}^{n_k}, \tilde{U}^{n_k}; s \leq t)$,
	\item each pair $(\tilde{U}^{n_k}, \tilde{W}^{n_k})$ satisfies the equation
	\begin{equation}
	\label{eq:appr.after.skorohod}
		d\tilde{U}^{n_k} + \left[A\tilde{U}^{n_k} + \theta(\Vert \tilde{U}^{n_k} - \tilde{U}^{n_k}_\ast \Vert) B^{n_k}(\tilde{U}^{n_k}) + F^{n_k}(\tilde{U}^{n_k})\right] \, dt = \sigma^{n_k}(\tilde{U}^{n_k}) \, d\tilde{W}^{n_k},
	\end{equation}
	with initial condition
	\[
		\tilde{U}^{n_k}(0) = \tilde{U}^{n_k}_0 := P^{n_k} \tilde{U}_0,
	\]
	where $\tilde{U}^{n_k}_\ast$ is the solution of
	\[
		\tfrac{d}{dt} \tilde{U}^{n_k}_\ast + A \tilde{U}^{n_k}_\ast = 0, \qquad \tilde{U}^{n_k}_\ast(0) = \tilde{U}^{n_k}_0.
	\]
\end{enumerate}
\end{proposition}

\begin{proof}
Using the estimates established in Lemma \ref{lemma:approximation.estimates} and the compactness of the embeddings
\begin{gather*}
	L^2\left(0, t; D(A) \right) \cap W^{1/4, 2}\left(0, t; H\right) \hook \hook L^2\left(0, t; V \right),\\
	W^{1, 2}\left(0, t; H \right) \hook \hook C\left(\left[0, t\right], V'\right), \qquad W^{\alpha, q}\left(0, t; H \right) \hook \hook C\left(\left[0, t \right], V'\right)
\end{gather*}
for some $\alpha \in \left(1/q, 1/2\right)$, we may repeat the proof of \cite[Lemma 4.1]{debussche2011} to show that the sequence of measures $\lbrace \mu^n \rbrace_{n=1}^\infty$ is tight in $\Xc$. The first assertion then follows immediately by the Skorokhod Theorem, see e.g.\ \cite[Theorem 2.4]{dpz}.

Regarding the second claim, we recall that cylindrical process is fully determined by its law, see e.g.\ \cite[Lemma 2.1.35]{bfh}. To establish that a cylindrical Wiener process is $\left( \Fc_t^{n_k} \right)_{t \geq 0}$-adapted, by e.g.\ \cite[Corollary 2.1.36]{bfh} it suffices to show that the process $W^{n_k}(s+h)-W^{n_k}(s)$ is non-anticipative w.r.t.\ the filtration $\left( \Fc_t^{n_k} \right)_{t \geq 0}$, which again follows from the equality in law.

The final assertion follows using an infinite dimensional version of the mollification method by Bensoussan in \cite[Section 4.3.4]{bensoussan1995}, see also the proof of \cite[Lemma 2.1]{debussche2011}, using estimates for convolution in Banach spaces from e.g.\ \cite[Section 1.2]{hytonen2016}. Let
\begin{equation}
	\label{eq:xn}
	X^{n_k} = \int_0^t \bigg\Vert U^{n_k} + \int_0^s AU^{n_k} +  B^{n_k}\left(U^{n_k}\right) + F^{n_k}\left(U^{n_k}\right) \, dr - U^{n_k}(0) - \int_0^s \sigma^{n_k}\left(U^{n_k}\right) \, dW \bigg\Vert_{V'}^2 \, ds.
\end{equation}
Since $U^{n_k}$ are the solutions of \eqref{eq:galerkin}, clearly
\begin{equation}
	\label{eq:xn.zero.as}
	X^{n_k} = 0 \quad \text{$\Pb$-a.s.} \qquad \text{and thus} \quad \Eb \left[ \frac{X^{n_k}}{1+X^{n_k}} \right] = 0.
\end{equation}
Let $\tilde{X}^{n_k}$ be the analogue of \eqref{eq:xn} with $\left(\tilde{U}^{n_k}, \tilde{W}^{n_k}\right)$ instead of $\left(U^{n_k}, W\right)$, that is
\begin{equation}
	\label{eq:xn.tilde}
	\tilde{X}^{n_k} = \int_0^t \bigg\Vert \tilde{U}^{n_k} + \int_0^s A\tilde{U}^{n_k} + B^{n_k}\left(\tilde{U}^{n_k}\right) + F^{n_k}\left(\tilde{U}^{n_k}\right) \, dr
	- \tilde{U}^{n_k}(0) - \int_0^s \sigma^{n_k}\left(\tilde{U}^{n_k}\right) \, d\tilde{W}^{n_k} \bigg\Vert_{V'}^2 \, ds.
\end{equation}
To prove \eqref{eq:appr.after.skorohod}, it suffices to establish that
\[
	\tilde{\Eb} \left[ \frac{\tilde{X}^{n_k}}{1+\tilde{X}^{n_k}} \right] = 0.
\]
The difficulty arises from the presence of stochastic integral in \eqref{eq:xn.tilde}. For $\varepsilon > 0$ let $K_\varepsilon: \Rb \to [0, \infty)$ be defined by
\[
	K_\varepsilon(r) = \mathds{1}_{[0, \infty)}(r) \frac{\exp\left( -r/\varepsilon\right)}{\varepsilon}, \qquad r \in \Rb.
\]
Let $\Kc_\varepsilon: L^2\left(0, t; L_2\left(\Uc, H\right)\right) \to L^2\left(0, t; L_2\left(\Uc, H\right)\right)$ be the convolution operator
\begin{equation}
	\label{eq:convolution.operator}
	\Kc_\varepsilon(q) = K_\varepsilon \ast q, \qquad q \in L^2\left(0, t; L_2\left(\Uc, H\right)\right).
\end{equation}
By the Young inequality for convolutions in Banach spaces, see e.g.\ \cite[Lemma 1.2.30]{hytonen2016}, we have
\begin{equation}
	\label{eq:est.expected}
	\Eb \Vert \Kc_\varepsilon\left(\sigma^{n_k}\left(U^{n_k}\right)\right) \Vert_{L^2\left(0, t; L_2\left(\Uc, H\right)\right)}^2 \leq \Eb \Vert \sigma^{n_k}\left(U^{n_k}\right) \Vert_{L^2\left(0, t; L_2\left(\Uc, H\right)\right)}^2,
\end{equation}
and from \cite[Proposition 1.2.32]{hytonen2016} we obtain that for fixed $q \in L^2\left(0, T; L_2\left(\Uc, H\right)\right)$ we have
\begin{equation}
	\label{eq:modif.convergence}
	\Kc_\varepsilon(q) \to q \quad \text{in} \quad L^2\left(0, t; L_2\left(\Uc, H\right)\right) \quad \text{as} \quad \varepsilon \to 0+.
\end{equation}

Let $X^{n_k, \varepsilon}, \tilde{X}^{n_k, \varepsilon}$ be the equivalents of $X^{n_k}, \tilde{X}^{n_k}$ with $\Kc_\varepsilon\left(\sigma^{n_k}\left(U^{n_k}\right)\right)$ and $\Kc_\varepsilon(\sigma^{n_k}(\tilde{U}^{n_k}))$ instead of $\sigma^{n_k}\left(U^{n_k}\right)$ and $\sigma^{n_k}(\tilde{U}^{n_k})$, respectively, i.e.
\begin{align}
	\nonumber
	X^{{n_k}, \varepsilon} &= \int_0^t \bigg\Vert U^{n_k} + \int_0^\cdot AU^{n_k} + B^{n_k}\left(U^{n_k}\right) + F^{n_k}\left(U^{n_k}\right) \, dr - U^{n_k}(0)\\
	\label{eq:xne}
	&\hphantom{= \int_0^T \bigg\Vert} - \int_0^\cdot \int_0^r K_\varepsilon(r-u) \sigma^{n_k}\left(U^{n_k}(u)\right) \, du \, dW(r) \bigg\Vert_{V'}^2 \, ds,\\
	\nonumber
	\tilde{X}^{{n_k}, \varepsilon} &= \int_0^t \bigg\Vert \tilde{U}^{n_k} + \int_0^\cdot A\tilde{U}^{n_k} + B^{n_k}(\tilde{U}^{n_k}) + F^{n_k}(\tilde{U}^{n_k}) \, dr - \tilde{U}^{n_k}(0)\\
	\label{eq:xtne}
	&\hphantom{= \int_0^T \bigg\Vert} - \int_0^\cdot \int_0^r K_\varepsilon(r-u) \sigma^{n_k}(\tilde{U}^{n_k}(u)) \, du \, d\tilde{W}^{n_k}(r) \bigg\Vert_{V'}^2 \, ds.
\end{align}
From the definition of the stochastic integral w.r.t.\ a cylindrical Wiener process, the stochastic Fubini theorem, see e.g.\ \cite[Section 4.5]{dpz}, and the stochastic integration by parts
\begin{align*}
	\int_0^s \int_0^r K_\varepsilon(r-u) &\sigma^{n_k}\left(U^{n_k}(u)\right) \, du \, dW_r\\
	&= \sum_{\ell=1}^\infty \sum_{j=1}^{n_k} h_j \left( \int_0^s \int_u^s K_\varepsilon(r-u) \left( \sigma\left(U^{n_k}(u)\right) e_\ell, h_j \right)_H \, dW_r^\ell \, du \right)\\
	&= \sum_{\ell=1}^\infty \sum_{j=1}^{n_k} h_j \bigg( \int_0^s \left( \sigma\left(U^{n_k}(u)\right) e_\ell, h_j \right)_H \bigg[ K_\varepsilon(s-u)W^\ell(s) - W^\ell(u)\\
	&\hphantom{= \sum_{\ell=1}^\infty \sum_{j=1}^{n_k} h_j \bigg( \int_0^s}+ \frac{1}{\varepsilon} \int_u^s W^\ell(r) K_\varepsilon(r-u) \, dr \bigg] \, du \bigg).
\end{align*}
Returning to \eqref{eq:xne} we observe that all the integrals involved are deterministic and therefore there exists a bounded continuous function $\phi^{{n_k}, \varepsilon}$ on $\Xc$ such that
\[
	\frac{X^{{n_k}, \varepsilon}}{1 + X^{{n_k}, \varepsilon}} = \phi^{{n_k}, \varepsilon}\left(U^{n_k}, W\right).
\]
Similarly from \eqref{eq:xtne} we get
\[
	\frac{\tilde{X}^{{n_k}, \varepsilon}}{1 + \tilde{X}^{{n_k}, \varepsilon}} = \phi^{{n_k}, \varepsilon}(\tilde{U}^{n_k}, \tilde{W}^{n_k}).
\]
The remaining part is straightforward. Since $\left(U^{n_k}, W\right)$ and $(\tilde{U}^{n_k}, \tilde{W}^{n_k})$ have the same law, we have
\begin{equation}
	\label{eq:tilde.laws}	
	\widetilde{\Eb} \left[ \frac{\tilde{X}^{{n_k}, \varepsilon}}{1 + \tilde{X}^{{n_k}, \varepsilon}} \right] = \widetilde{\Eb} \left[ \phi^{{n_k}, \varepsilon}(\tilde{U}^{n_k}, \tilde{W}^{n_k}) \right] = \Eb \left[ \phi^{{n_k}, \varepsilon}\left(W, U^{n_k}\right) \right] = \Eb \left[ \frac{X^{{n_k}, \varepsilon}}{1 + X^{{n_k}, \varepsilon}} \right] .
\end{equation}
Recalling that for $h_1$ and $h_2$ from some Hilbert space $\Hc$ the identity
\[
	\vert h_1 \vert_\Hc^2 - \vert h_2 \vert_\Hc^2 = \left( h_1 - h_2, h_1 + h_2 \right)_\Hc
\]
holds,we use the Burkholder-Davis-Gundy inequality \eqref{eq:bdg} and \eqref{eq:est.expected} we get
\begin{align*}
	\Eb \bigg| &\frac{X^{{n_k}, \varepsilon}}{1 + X^{{n_k}, \varepsilon}} - \frac{X^{n_k}}{1 + X^{n_k}} \bigg| \leq \Eb \left| \frac{X^{{n_k}, \varepsilon}}{1 + X^{{n_k}, \varepsilon}} - \frac{X^{n_k}}{1 + X^{{n_k}, \varepsilon}} \right| + \Eb \left| \frac{X^{n_k}}{1 + X^{{n_k}, \varepsilon}} - \frac{X^{n_k}}{1 + X^{n_k}} \right|\\
	&\leq 2 \Eb \left| X^{{n_k}, \varepsilon} - X^{n_k} \right|\\
	&\leq C \Eb \Bigg[ \int_0^t \bigg( \int_0^s \Kc_\varepsilon\left(\sigma^{n_k}\left(U^{n_k}\right)\right) - \sigma^{n_k}\left(U^{n_k}\right) \, dW,\\
	&\hphantom{\leq C \Eb \Bigg[ \int_0^t \bigg( \ } \int_0^s \Kc_\varepsilon\left(\sigma^{n_k}\left(U^{n_k}\right)\right) + \sigma^{n_k}\left(U^{n_k}\right) \, dW \bigg)_{V'} \, ds \Bigg]\\
	&\leq C \left\Vert \Kc_\varepsilon\left( \sigma^{n_k}\left(U^{n_k}\right)\right) - \sigma^{n_k}\left(U^{n_k}\right) \right\Vert^2_{L^2(0, t; L_2(\Uc, H))}.
\end{align*}
In a similar way we may establish the estimate
\[
	\Eb \left| \frac{\tilde{X}^{{n_k}, \varepsilon}}{1 + \tilde{X}^{{n_k}, \varepsilon}} - \frac{\tilde{X}^{n_k}}{1 + \tilde{X}^{n_k}} \right| \leq  C \left\Vert \Kc_\varepsilon\left(\sigma^{n_k}\left(U^{n_k}\right)\right) - \sigma^{n_k}\left(U^{n_k}\right) \right\Vert^2_{L^2\left(0, t; L_2\left(\Uc, H\right)\right)}.
\]
Then by \eqref{eq:xn.zero.as}, \eqref{eq:tilde.laws}, \eqref{eq:modif.convergence} and the above estimates we get
\begin{align*}
	\widetilde{\Eb} \left| \frac{\tilde{X}^{n_k}}{1+\tilde{X}^{n_k}} \right| &\leq \widetilde{\Eb} \left| \frac{\tilde{X}^{n_k}}{1+\tilde{X}^{n_k}}  - \frac{\tilde{X}^{{n_k}, \varepsilon}}{1+\tilde{X}^{{n_k}, \varepsilon}} \right| + \widetilde{\Eb} \left| \frac{\tilde{X}^{{n_k}, \varepsilon}}{1+\tilde{X}^{{n_k}, \varepsilon}} \right|\\
	&\leq \widetilde{\Eb} \left| \frac{\tilde{X}^{n_k}}{1+\tilde{X}^{n_k}}  - \frac{\tilde{X}^{{n_k}, \varepsilon}}{1+\tilde{X}^{{n_k}, \varepsilon}} \right| + \Eb \left| \frac{X^{n_k}}{1+X^{n_k}}  - \frac{X^{{n_k}, \varepsilon}}{1+X^{{n_k}, \varepsilon}} \right|\\
	&\leq C \left\Vert \Kc_\varepsilon\left(\sigma^{n_k}\left(U^{n_k}\right)\right) - \sigma^{n_k}\left(U^{n_k}\right) \right\Vert^2_{L^2\left(0, t; L_2\left(\Uc, H\right)\right)} \to 0, \qquad \varepsilon \to 0+.
\end{align*}
It follows that
\[
	\tilde{U}^{n_k}(s) + \int_0^s A\tilde{U}^{n_k} + B^{n_k}(\tilde{U}^{n_k}) + F^{n_k}(\tilde{U}^{n_k}) \, dr = \tilde{U}^{n_k}(0) + \int_0^s \sigma^{n_k}(\tilde{U}^{n_k}) \, d\tilde{W}^{n_k}	
\]
for almost surely on $\tilde{\Omega} \times [0, t]$. By the continuity of the functions in $V'$ the above stochastic differential equation holds $\tilde{\Pb}$-almost surely.
\end{proof}

\begin{proposition}
\label{prop:global.martingale.existence}
Let $(\tilde{U}^{n_k}, \tilde{W}^{n_k})$ be a sequence of $\Xc$-valued random variables on a probability space $(\tilde{\Omega}, \tilde{\Fc}, \tilde{\Pb})$ such that
\begin{enumerate}
	\item $(\tilde{U}^{n_k}, \tilde{W}^{n_k}) \to (\tilde{U}, \tilde{W})$ in $\Xc$ $\tilde{\Pb}$-a.s., i.e.
	\begin{equation*}
		\tilde{U}^{n_k} \to \tilde{U} \ \text{in} \ L^2\left(0, t; V\right) \cap C\left(\left[0, t \right], V'\right), \ \tilde{W}^{n_k} \to \tilde{W} \ \text{in} \ C\left(\left[0, t\right]; \Uc_0 \right), \quad \Pb\text{-a.s.},
	\end{equation*}
	\item $\tilde{W}^{n_k}$ is a cylindrical Wiener process with reproducing kernel Hilbert space $\Uc$ adapted to the filtration $\left( \Fc_t^{n_k} \right)_{t \geq 0}$ that contains the $\sigma$-algebra $\sigma(\tilde{W}^{n_k}, \tilde{U}^{n_k}; s \leq t)$,
	\item each pair $(\tilde{U}^{n_k}, \tilde{W}^{n_k})$ satisfies \eqref{eq:appr.after.skorohod} with $\tilde{U}_0 \in L^q\left( \Omega, V \right)$ for some $q > 4$.
\end{enumerate}
Let $\tilde{\Fc}_t$ be the completion of $\sigma(\tilde{W}(s), \tilde{U}(s), 0 \leq s \leq t)$ and $\tilde{\Sc} = (\tilde{\Omega}, \tilde{\Fc}, ( \tilde{\Fc}_t )_{t \geq 0}, \tilde{\Pb})$. Then $(\tilde{\Sc}, \tilde{W}, \tilde{U})$ is a global martingale solution to the approximating problem \eqref{eq:pe.modified}.

Moreover, for all $p > 2$ such that $q \geq 2p$ and for all $t \geq 0$ the solution $\tilde{U}$ satisfies
\begin{equation}
	\label{eq:martingale.solution.approx.regularity}
	\tilde{U} \in L^p\left( \Omega, C\left( [0, t], V \right) \right), \qquad \vert A\tilde{U} \vert^2 \Vert \tilde{U} \Vert^{p-2} \in L^1\left( \Omega, L^1\left( 0, t \right) \right).
\end{equation}
\end{proposition}

Note that the requirement $q > 4$ above is used to prove the convergence of the stochastic term of the approximating sequence $\tilde{U}^{n_k}$.

\begin{proof}
The proof follows the argument of \cite[Proposition 7.1]{debussche2011}. Arguing as in \cite[Section 7.1]{debussche2011} we may find
\begin{equation}
	\label{eq:tilde.u.regular}
	\tilde{U} \in L^2\left(\tilde{\Omega}, L^2\left(0, t; D(A)\right)\right) \cap L^2\left( \tilde{\Omega}, L^\infty\left(0, t; V\right)\right)
\end{equation}
such that
\begin{equation}
	\label{eq:tilde.u.convergence}
	\tilde{U}^{n_k} \harp \tilde{U} \ \text{in} \ L^2\left(\tilde{\Omega}, L^2\left(0, t; D(A)\right)\right), \qquad \tilde{U}^{n_k} \to \tilde{U} \ \text{in} \ L^2\left(\Omega, L^2\left(0, t; V\right)\right).
\end{equation}

Let $U^\sharp \in D(A)$ be fixed. Repeating the argument of \cite[Section 7.2]{debussche2011} we deduce the convergence of deterministic terms
\begin{multline}
	\label{eq:deterministic.convergence}
	\int_0^\cdot \left(A\tilde{U}^{n_k} + \theta(\Vert \tilde{U}^{n_k} - \tilde{U}^{n_k}_\ast \Vert) B^{n_k}(\tilde{U}^{n_k}) + F^{n_k}(\tilde{U}^{n_k}), U^\sharp\right) \, dr\\
	\to \int_0^\cdot \left( A\tilde{U} + \theta(\Vert \tilde{U} - \tilde{U}_\ast \Vert) B(\tilde{U}) + F(\tilde{U}), U^\sharp\right) \, dr
\end{multline}
in $L^r([0, t] \times \tilde{\Omega})$ for every $r \in [1, 2)$. To show the convergence of the stochastic term we use the Lipschitz continuity of $\sigma$ in $L_2\left(\Uc, H\right)$ \eqref{eq:sigma.lip.H} and the Poincaré type inequality \eqref{eq:inverse.poincare} to get
\begin{align*}
	\Vert \sigma^{n_k}(\tilde{U}^{n_k}) - \sigma(\tilde{U}) &\Vert_{L^2\left(0, t; L_2\left(\Uc, H\right)\right)}^2\\
	&\leq C \left( \Vert \sigma(\tilde{U}^{n_k}) - \sigma(\tilde{U}) \Vert_{L^2\left(0, t; L_2\left(\Uc, H\right)\right)}^2 + \Vert Q_{n_k} \sigma(\tilde{U}^{n_k}) \Vert_{L^2\left(0, t; L_2\left(\Uc, H\right)\right)}^2 \right)\\
	&\leq C \left( \Vert \tilde{U}^{n_k} - \tilde{U} \Vert_{L^2(0, t; V)}^2 + \frac{1}{\lambda_{n_k}} \Vert \sigma(\tilde{U}^{n_k}) \Vert_{L^2(0, t; L_2(\Uc, V))}^2 \right)\\
	&\leq C \left( \Vert \tilde{U}^{n_k} - \tilde{U} \Vert_{L^2\left(0, t; V\right)}^2 + \frac{1}{\lambda_{n_k}} \int_0^t 1 + \Vert \tilde{U}^{n_k} \Vert^2 + \vert A\tilde{U}^{n_k} \vert^2 \, ds \right)
\end{align*}
Recalling the uniform estimates \eqref{eq:un.Vp} from Lemma \ref{lemma:approximation.estimates} and the convergence \eqref{eq:tilde.u.convergence} we have $\sigma^{n_k}(\tilde{U}^{n_k}) \to \sigma(\tilde{U})$ in $L^2\left(0, t; L_2\left(\Uc, H\right)\right)$ and therefore
\[
	\Vert \sigma^{n_k}(\tilde{U}^{n_k}) - \sigma(\tilde{U}) \Vert_{L_2\left(\Uc, H\right)}^2 \to 0
\]
for almost all $(s, \omega) \in [0, t] \times \tilde{\Omega}$. Using the estimate \eqref{eq:un.Vp} again with the bound \eqref{eq:sigma.bnd.H} on $\sigma(U)$ in $L_2\left(\Uc, H\right)$ we get
\[
	\sup_{k \in \Nb} \Eb \left[ \int_0^t \Vert \sigma^{n_k}(\tilde{U}^{n_k}) \Vert_{L_2\left(\Uc, H\right)}^{q/2} \, ds \right] \leq C \sup_{k \in \Nb} \Eb \left[ 1 + \sup_{s \in [0, t]} \Vert \tilde{U}^{n_k} \Vert^{q/2} \right],
\]
which gives uniform integrability of $\Vert \sigma(\tilde{U}^{n_k}) \Vert_{L_2(\Uc, H)}$ in $L^{q_0}([0, t] \times \tilde{\Omega})$ with $q_0 \in [1, q/2)$ and $q/2 > 2$. By the Vitali convergence theorem we have
\begin{equation}
	\label{eq:sigma.nk.lp}
	\sigma^{n_k}(\tilde{U}^{n_k}) \to \sigma(\tilde{U}) \qquad \text{in} \quad L^{q/2}\left(\tilde{\Omega}, L^{q/2}\left(0, t; L_2\left(\Uc, H\right)\right)\right).
\end{equation}
From \cite[Lemma 2.1]{debussche2011} we obtain
\begin{equation}
	\label{eq:sigma.dw.convergence}
	\int_0^\cdot \sigma^{n_k}(\tilde{U}^{n_k}) \, d\tilde{W}^{n_k} \to \int_0^\cdot \sigma(\tilde{U}) \, d\tilde{W}
\end{equation}
in probability in $L^2\left(0, t; H\right)$. Using the Vitali Theorem once more with \eqref{eq:sigma.nk.lp} we observe that the convergence in \eqref{eq:sigma.dw.convergence} occurs in the space $L^2(\tilde{\Omega}, L^2\left(0, t; H\right))$.

From the convergence in \eqref{eq:deterministic.convergence} and the above we infer that for all $U^\sharp \in D(A)$ the equation
\begin{multline}
	\label{eq:variational.eq}
	\left(\tilde{U}(s), U^\sharp\right) + \int_0^s \left(A\tilde{U} + \theta(\Vert \tilde{U} - \tilde{U}_\ast \Vert) B(\tilde{U}) + F(\tilde{U}), U^\sharp\right) \, dr\\
	= \left(\tilde{U}_0, U^\sharp\right) + \int_0^s \left(\sigma(\tilde{U}) \, dW, U^\sharp\right)
\end{multline}
holds for almost all $(s, \omega) \in [0, t] \times \tilde{\Omega}$. By \eqref{eq:tilde.u.regular} and by the density of $D(A)$ in $H$, \eqref{eq:variational.eq} holds for $U^\sharp \in H$ which gives the analogue to \eqref{eq:solution.def} for the modified equation \eqref{eq:pe.modified}.

The proof of continuity of $\tilde{U}$ in time in the space $V$ follows by a maximal regularity type argument similarly as in \cite[Section 7.3]{debussche2011}, see also \cite{pardoux1979}. By the assumption on the growth of the $L_2\left(\Uc, V \right)$-norm of $\sigma$ in \eqref{eq:sigma.bnd.V} and the regularity of $\tilde{U}$ from \eqref{eq:tilde.u.regular} we have
\[
	\sigma(\tilde{U}) \in L^2\left(\tilde{\Omega}; L^2\left(0, t; L_2\left(\Uc, V\right)\right)\right),
\]
therefore the solution of
\[
	dZ + AZ \, dt = \sigma(U) \, d\tilde{W}, \qquad Z(0) = \tilde{U}_0,
\]
satisfies
\begin{equation}
	\label{eq:z.regularity}
	Z \in L^2\left(\tilde{\Omega}, C\left([0, t], V\right)\right) \cap L^2\left(\tilde{\Omega}; L^2\left(0, t; D(A)\right)\right).
\end{equation}
Then, defining $\bar{U} = \tilde{U} - Z$, by \eqref{eq:pe.modified} we have $\Pb$-almost surely
\begin{equation}
	\label{eq:bar.u}
	\tfrac{d}{dt} \bar{U} + A\bar{U} + \theta(\Vert \bar{U} + Z - \tilde{U}_\ast \Vert) B(\bar{U} + Z) + F(\bar{U} + Z) = 0, \qquad \bar{U}(0) = \tilde{U}_0.
\end{equation}
The regularity of $Z$ \eqref{eq:z.regularity} and $\tilde{U}$ \eqref{eq:tilde.u.regular} gives
\[
	A \bar{U}, \ \theta(\Vert \bar{U} + Z - \tilde{U}_\ast \Vert) B(\bar{U} + Z), \ F(\bar{U} + Z) \in L^2\left(\tilde{\Omega}, L^2\left(0, t; H\right)\right)
\]
and thus
\[
	\frac{d}{dt} A^{1/2} \bar{U} \in L^2\left(\tilde{\Omega}, L^2\left(0, t; V'\right)\right), \qquad A^{1/2} \bar{U} \in L^2\left(\tilde{\Omega}, L^2\left(0, t; V\right)\right).
\]
Finally, by the Lions-Magenes Lemma, see e.g.\ \cite[Lemma 1.2, Chapter 3]{temam1979}, we infer from \eqref{eq:bar.u} that $A^{1/2} \bar{U} \in C\left(\left[0, t \right], H\right)$ and therefore $\bar{U} \in C\left(\left[0, t \right], V\right)$, both $\Pb$-almost surely.

To establish \eqref{eq:martingale.solution.approx.regularity} it suffices to repeat the estimates in the first part of Lemma \ref{lemma:approximation.estimates} for $\tilde{U}$. This can be done by using the It\^{o} Lemma from Theorem \ref{thm:ito.lemma}. It is straightforward to check that the assumptions of the lemma are satisfied for $\psi(U) = \Vert U \Vert^p$. Indeed, the operator $D\psi(U)$ can be extended to $H$ if $U \in D(A)$ by $D\psi(U)(h) = p \Vert U \Vert^{p-2} \left( AU, h \right)$ and the required convergence property \eqref{eq:ito.D.psi.convergence} can be established in a direct way.
\end{proof}

\begin{corollary}
\label{cor:loc.mart.sol}
Let $(\tilde{\Sc}, \tilde{W}, \tilde{U})$ be the global martingale solution of \eqref{eq:pe.modified} from Proposition \ref{prop:global.martingale.existence} with $\tilde{U}_0 \in L^q\left( \Omega, V \right)$ for some $q \geq 8$. Let
\begin{equation*}
	\tau = \inf \left\lbrace t \geq 0 \mid \Vert \tilde{U} - U_\ast \Vert \geq \kappa \right\rbrace,
\end{equation*}
with $U_\star$ being the solution to the linear problem \eqref{eq:linear} and $\kappa > 0$ the constant from \eqref{eq:theta} and Lemma \ref{lemma:approximation.estimates}. Then $(\tilde{\Sc}, \tilde{W}, \tilde{U}, \tau)$ is a local martingale solution to the problem \eqref{eq:pe.abstract}.

Moreover, for all $p > 2$ such that $q \geq 2p$, then for all $t \geq 0$
\begin{equation}
	\label{eq:local.mart.sol.regularity}
	\tilde{U}\left( \cdot \wedge \tau \right) \in L^p \left( \Omega, C\left( [0, t], V \right) \right), \qquad \mathds{1}_{[0, \tau]} \vert A\tilde{U} \vert^2 \Vert \tilde{U} \Vert^{p-2} \in L^1\left( \Omega, L^1\left( 0, t \right) \right).
\end{equation}
\end{corollary}

\subsection{Existence of local pathwise solutions}

Implementing the same argument from \cite[Section 5]{debussche2011} relying on the the Gy\"{o}ngy-Krylov theorem, see \cite[Lemma 1.1]{gyongy1996}, we briefly sketch the proof of the existence of local pathwise solution of the original problem \eqref{eq:pe.v.reform}-\eqref{eq:pe.T.reform}. We provide more details mainly in those parts where the explicit smallness of constants in the growth estimates of $\sigma(U)$ in $L_2\left( \Uc, V \right)$ \eqref{eq:sigma.bnd.V} and the Lipschitz constant $\gamma$ in \eqref{eq:sigma.lip.V} may play a role.

First we state a pathwise uniqueness result with a proof following that of \cite[Proposition 5.1]{debussche2011} with only minor changes. From now on we assume that $F$ also satisfies \eqref{eq:F.lipschitz}.

\begin{proposition}
\label{prop:pathwise.uniqueness}
Let $U_0 \in L^q\left( \Omega; \Fc_0, V \right)$ with $q \geq 8$. Let $\Sc = \left(\Omega, \Fc, \Fct, \Pb \right)$ and $W$ be as in Theorem \ref{thm:maximal.existence} and let $\left(\Sc, W, U_1\right)$ and $\left(\Sc, W, U_2\right)$ be two global martingale solutions of the modified equation \eqref{eq:pe.modified}. Let $\Omega_0 = \left\lbrace U_1(0) = U_2(0) \right\rbrace \subseteq \Omega$. Then
\begin{equation*}
	\Pb \left( \left\lbrace \mathds{1}_{\Omega_0}\left(U^1(t) - U^2(t)\right) = 0 \ \text{for all} \ t \geq 0 \right\rbrace \right) = 1.
\end{equation*}
\end{proposition}

\begin{proof}
Let $R = \uone - \utwo$ and $\bar{R} = \mathds{1}_{\Omega_0} R$. Let $\tau^n$ be the stopping time defined by
\[
	\tau^n = \inf \left\lbrace t \geq 0 \mid \int_0^t \Vert \uone \Vert^2 \vert A\uone \vert^2 + \Vert \utwo \Vert^2 \vert A\utwo \vert^2 \, ds \geq n \right\rbrace.
\]
Since both $\uone$ and $\utwo$ are defined globally and $U_0 \in L^q\left( \Omega; \Fc_0, V\right)$ with $q \geq 8$, from the estimates from Lemma \ref{lemma:approximation.estimates} we deduce $\tau^n \to \infty$ $\Pb$-a.s.\ and therefore it suffices to show that
\[
	\Eb \sup_{s \in \left[0, \tau^n \wedge t \right]} \Vert \bar{R}(s) \Vert^2  = 0
\]
for all $n \in \Nb$ and $t > 0$. Subtracting the equations for $\uone$ and $\utwo$ we get
\begin{gather*}
	\begin{multlined}
		dR + \left[AR + \theta(\Vert \uone - U^{(1)}_\ast \Vert) B(\uone) - \theta(\Vert \utwo - U^{(2)}_\ast \Vert) B(\utwo) + F(\uone) - F(\utwo)\right] \, dt\\
		= \left[ \sigma(\uone) - \sigma(\utwo) \right] \, dW,
	\end{multlined}
	\\
	R(0) = \uone(0) - \utwo(0).
\end{gather*}
Fix $n \in \Nb$ and let $\tau_a$, $\tau_b$ be stopping times such that $0 \leq \tau_a \leq \tau_b \leq \tau^n$. Applying the It\^{o} Lemma from Theorem \ref{thm:ito.lemma}, multiplying by $\mathds{1}_{\Omega_0}$, using the bilinearity of $B$, integrating over $\left[\tau_a, \tau_b\right]$ and applying the expected value gives the estimate
\begin{align}
	\nonumber
	\Eb \bigg[ \sup_{s \in \left[\tau_a, \tau_b\right]} &\Vert \bar{R} \Vert^2 + 2 \int_{\tau_a}^{\tau_b} \vert A \bar{R} \vert^2 \, ds \bigg] \leq  \Eb \Vert \bar{R}(\tau_a) \Vert^2 \\
	\nonumber
	&\quad + 2 \Eb \int_{\tau_a}^{\tau_b} \left| \left( \left[\theta(\Vert \uone - U^{(1)}_\ast \Vert) - \theta(\Vert \utwo - U^{(2)}_\ast \Vert) \right] B(\uone), A\bar{R} \right) \right| \, ds\\
	\nonumber
	&\quad+ 2 \Eb \int_{\tau_a}^{\tau_b} \left| \left( B(\uone) - B(\utwo), A\bar{R}\right) \, ds \right| + 2 \Eb \int_{\tau_a}^{\tau_b} \left| \left( F(\uone) - F(\utwo), A\bar{R}\right) \right| \, ds \\
	\nonumber
	&\quad+ 2 \Eb \sup_{s \in \left[\tau_a, \tau_b\right]} \left| \int_{\tau_a}^{s} \left( \left[\sigma(\uone) - \sigma(\utwo) \right] \, dW, A\bar{R} \right) \right|\\
	\nonumber
	&\quad+ \Eb \int_{\tau_a}^{\tau_b} \mathds{1}_{\Omega_0}\Vert \sigma(\uone) - \sigma(\utwo) \Vert_{L_2\left(\Uc, V\right)}^2 \, ds \\
	\label{eq:R.estimate}
	&= \Eb \Vert \bar{R}(\tau_a) \Vert^2 + J_1 + J_2 + J_3 + J_4 + J_5.
\end{align}
Note that $\mathds{1}_{\Omega_0}(U^{(1)}_\ast - U^{(2)}_\ast) = 0$ almost surely. To estimate $J_1$, recall that $\theta$ is Lipschitz continuous and use the estimate \eqref{eq:b.estimate2} and the Young inequality to get
\begin{equation}
	J_1 \leq \frac{\varepsilon}{3} \Eb \int_{\tau_a}^{\tau_b} \vert A \bar{R} \vert^2 \, ds + C_\varepsilon \Eb \int_{\tau_2}^{\tau_b} \Vert \bar{R} \Vert^2 \Vert \uone \Vert^2 \vert A \uone \vert^2 \, ds
\end{equation}
for some $\varepsilon > 0$ precisely determined later. The term $J_2$ can be estimated by the bound \eqref{eq:b.estimate2} on $B$ and the Young inequality. We deduce
\begin{align}
	\nonumber
	J_2 &\leq 2 \Eb \int_{\tau_a}^{\tau_b} \left| \left(B(\uone) - B(\utwo) \pm B(\utwo, \uone), A\bar{R} \right) \right| \, ds\\
	&\leq \frac{\varepsilon}{3} \Eb \int_{\tau_a}^{\tau_b} \vert A \bar{R} \vert^2 \, ds + C_\varepsilon \Eb \int_{\tau_a}^{\tau_b} \Vert \bar{R} \Vert^2 \left( \Vert \uone \Vert^2 \vert A \uone \vert^2 + \Vert \utwo \Vert^2 \vert A \utwo \vert^2 \right) \, ds.
\end{align}
For $J_3$ we recall the Lipschitz continuity of $F$ \eqref{eq:F.lipschitz} and employ the Young inequality to obtain
\begin{equation}
	J_3 \leq \frac{\varepsilon}{3} \Eb \int_{\tau_a}^{\tau_b} \vert A \bar{R} \vert^2 \, ds + C_\varepsilon \Eb \int_{\tau_a}^{\tau_b} \Vert \bar{R} \Vert^2 \, ds.
\end{equation}
Regarding $J_4$, from the Burkholder-Davis-Gundy inequality \eqref{eq:bdg}, the Lipschitz continuity of $\sigma$ in $L_2\left(\Uc, V \right)$ \eqref{eq:sigma.lip.V} and the Young inequality we follow the same argument leading to \eqref{eq:Jp32} and infer
\begin{equation}
	\begin{split}
		J_4 &\leq 2 C_{BDG} \Eb \left( \int_{\tau_a}^{\tau_b} \Vert \sigma(\uone) - \sigma(\utwo) \Vert_{L_2\left(\Uc, V\right)}^2 \Vert \bar{R} \Vert^2 \, ds \right)^{1/2}\\
		&\leq \left( 1- \delta + \varepsilon \right) \Eb \sup_{s \in \left[\tau_a, \tau_b\right]} \Vert \bar{R} \Vert^2 + C_{\varepsilon, \delta} \Eb \int_{\tau_a}^{\tau_b} \left( 1 + \Vert \bar{R} \Vert^2 \right) \, ds + \frac{C_{BDG}^2 \gamma}{1-\delta} \Eb \int_{\tau_a}^{\tau_b} \vert A \bar{R} \vert^2 \, ds
	\end{split}
\end{equation}
for some $\delta > \varepsilon$ small. Finally, the integral $J_5$ is estimated using the Lipschitz continuity of $\sigma$ in $L_2\left(\Uc, V\right)$ \eqref{eq:sigma.lip.V}. We get
\begin{equation}
	\label{eq:R.estimate5}
	J_5 \leq \gamma \Eb \int_{\tau_a}^{\tau_b} \vert A \bar{R} \vert^2 \, ds + C \Eb \int_{\tau_a}^{\tau_b} \Vert \bar{R} \Vert^2 \, ds.
\end{equation}

Collecting the estimates \eqref{eq:R.estimate}-\eqref{eq:R.estimate5} we finally obtain
\begin{multline*}
	\Eb \left[ \left(\delta - \varepsilon \right) \sup_{s \in \left[\tau_a, \tau_b\right]} \Vert \bar{R} \Vert^2 + \left(2 - \varepsilon - \gamma - \frac{C_{BDG}^2 \gamma}{1-\delta} \right) \int_{\tau_a}^{\tau_b} \vert A \bar{R} \vert^2 \, ds \right]\\
	\leq C \Eb \left[ \Vert \bar{R}(\tau_a) \Vert^2 + \int_{\tau_a}^{\tau_b} \Vert \bar{R} \Vert^2 \left( 1 + \Vert \uone \Vert^2 \vert A \uone \vert^2 + \Vert \utwo \Vert^2 \vert A \utwo \vert^2 \right) \, ds \right].
\end{multline*}
Recalling that $\sigma$ satisfies the hypothesis $H_4$ \eqref{eq:small.constants.maximal} and choosing $\delta$ and $\varepsilon$ sufficiently small we may invoke the stochastic Gronwall Lemma from Proposition \ref{prop:stochastic.gronwall} to conclude the proof.
\end{proof}

\begin{proposition}
\label{prop:local.path.existence}
Let $\Sc$, $W$, $\sigma$ and $F_U$ be as in Theorem \ref{thm:maximal.existence}.
\begin{enumerate}
	\item Let $U_0 \in L^q\left( \Omega; \Fc_0, V \right)$ with $q \geq \max \lbrace 2p, 8 \rbrace$. Then there exists a unique global pathwise solution of \eqref{eq:pe.modified}. Moreover, for all $p > 2$ such that $q \geq 2p$, then \eqref{eq:martingale.solution.approx.regularity} holds for all $t > 0$ with $U$ instead of $\tilde{U}$.
	\item If $U_0 \in L^2\left( \Omega; \Fc_0, V \right)$, then there exists a unique a local pathwise solution $(U, \tau)$ of \eqref{eq:pe.abstract}. Moreover, if $U_0 \in L^p\left( \Omega; \Fc_0, V \right)$ for some $p >2$, then \eqref{eq:local.mart.sol.regularity} holds for all $t > 0$ with $U$ instead of $\tilde{U}$.
\end{enumerate}
\end{proposition}

\begin{proof}
The proof of the first part runs exactly as in \cite[Section 5.1]{debussche2011} and is therefore omitted. Also if $p > 2$ is such that $q \geq 2p$, the equivalent of the regularity \eqref{eq:martingale.solution.approx.regularity} can be established by repeating the estimates of Lemma \ref{lemma:approximation.estimates} by the means of the It\^{o} Lemma from Theorem \ref{thm:ito.lemma}.

Assuming $U_0 \in L^2\left( \Omega; \Fc_0, V \right)$, we may prove the second claim by the following the localization argument of \cite[Proposition 4.1]{glatt-holtz2009}. Let $M \geq 1$ be fixed and for $k \in \Nb \cup \lbrace 0 \rbrace$ let $\Omega_k = \left\lbrace k \leq \Vert U_0 \Vert < k+1 \right\rbrace$. By the above there exist local solutions $\left(U_k, \rho_k\right)$ of the problem \eqref{eq:pe.abstract} with initial data $\mathds{1}_{\Omega_k} U_0 \in L^\infty\left(\Omega, V\right)$. Let us define
\[
	\tau_k = \rho_k \wedge \inf \left\lbrace s \geq 0; \sup_{r \in \left[0, s \wedge \rho_k\right]} \Vert U_k \Vert^2 + \int_0^{s \wedge \rho_k} \vert AU_k \vert^2\, dr \geq M + \Vert U_0 \Vert^2 \right\rbrace.
\]
Clearly, by the continuity of $U_k$ in $V$ we have $\tau_k > 0$ almost surely. Let us define
\[
	U = \sum_{k = 0}^\infty U_k \mathds{1}_{\Omega_k}, \qquad \tau = \sum_{k = 0}^\infty \tau_k \mathds{1}_{\Omega_k}.
\]
Then by the definitions of $\tau$ and $\tau_k$ we have
\begin{align*}
	\Eb \left[ \sup_{s \in [0, \tau]} \Vert U \Vert^2 + \int_0^\tau \vert AU \vert^2 \, ds \right] &= \Eb \sum_{k = 0}^\infty \mathds{1}_{\Omega_k} \left( \sup_{t \in [0, \tau]} \Vert U_k \Vert^2 + \int_0^\tau \vert AU_k \vert^2 \, ds \right)\\
	&\leq \Eb \sum_{k = 0}^\infty \mathds{1}_{\Omega_k} \left( M + \Vert U_0 \Vert^2 \right) \leq M + \Eb \Vert U_0 \Vert^2 <\infty,
\end{align*}
hence $(U, \tau)$ is a local pathwise solution of  \eqref{eq:pe.abstract}.

Let now $U_0 \in L^p\left( \Omega; \Fc_0, V \right)$ for some $p > 2$ satisfying $q \geq 2p$ and let $\left( U, \tau \right)$ be the local pathwise solution to \eqref{eq:pe.abstract} from above. Then from the definition of $\tau$ we get
\begin{align*}
	\sup_{s \in [0, \tau]} \Vert U \Vert^p &+ \int_0^\tau \vert AU \vert^2 \Vert U \Vert^{p-2} \, ds\\
	&\leq \left( \sup_{s \in [0, \tau]} \Vert U \Vert^2 \right)^{p/2} + \left( \sup_{s \in [0, \tau]} \Vert U \Vert^2 \right)^{(p-2)/2} \int_0^\tau \vert AU \vert^2 \, ds\\
	&\leq \left( M + \Vert U_0 \Vert^2 \right)^{p/2} + \left( M + \Vert U_0 \Vert^2 \right)^{(p-2)/2} \left( M + \Vert U_0 \Vert^2 \right)\\
	&\leq C_{p, M} \left( 1 + \Vert U_0 \Vert^p \right)
\end{align*}
$\Pb$-almost surely, which leads to the desired equivalent of \eqref{eq:local.mart.sol.regularity}.
\end{proof}

\subsection{Proof of Theorem \ref{thm:maximal.existence}}
\label{sec:proof.maximal.existence}

The proofs in this section are an adaptation of \cite[Section 4]{glatt-holtz2009} and \cite[Theorem 3]{mikulevicius2004}, see also \cite[Theoreme 14.21]{jacod1979} and \cite[Chapter 7, Section 2]{elworthy1982}.

Let $\Tc$ be the set all of stopping times $\tau$ such that there exists a process $U$ such that the couple $\left(U, \tau\right)$ is a local pathwise solution of the problem \eqref{eq:pe.abstract}. In particular $(U, \tau)$ has the regularity \eqref{eq:solution.regularity} and satisfies the equation \eqref{eq:solution.def}. By Proposition \ref{prop:local.path.existence} $\Tc$ is non-empty. By \cite[Chapter 5, Section 18]{doob1994} there exits a stopping time $\xi$ such that $\xi \geq \tau$ a.s.\ for all $\tau \in \Tc$ and a sequence of stopping times $\left( \tau_N \right)_{N = 1}^\infty \in \Tc$ satisfying $\tau_N \nearrow \xi$ on a full-measure set\footnote{We emphasize that we do not claim $\xi \in \Tc$ as that would be incompatible with our local solution $\left( U, \tau \right)$ being considered on a \emph{closed} interval $[0, \tau]$. Thus, we do not need to invoke the Kuratowski-Zorn Lemma.} $\Omega^1 \subseteq \Omega$. Let $\left(U_N, \tau_N \right)$ be the respective local pathwise solutions. By the pathwise uniqueness result from Proposition \ref{prop:pathwise.uniqueness} there exists $\Omega^2 \subseteq \Omega^1 \subseteq \Omega$ of full measure such that for all $N, M \in \Nb$
\begin{equation}
	\label{eq:path.uniq.k}
	U_N\left(t \wedge \tau_N \wedge \tau_M, \omega\right) = U_M\left(t \wedge \tau_N \wedge \tau_M, \omega\right), \qquad t \geq 0, \quad \omega \in \Omega^2.
\end{equation}
For $\omega \in \Omega^2$ and $t \geq 0$ we define
\[
	U\left(t, \omega\right) = \lim_{N \to \infty} U_N\left(t \wedge \tau_N, \omega\right) \mathds{1}_{\left[0, \xi\right)}(\omega).
\]
Note that the limit exists since by \eqref{eq:path.uniq.k} the sequence $U_N$ it is constant for $k \geq k_0$ with $k_0 = k_0(\omega)$. It is straightforward to check that $U$, $\xi$ and $\tau_N$ have the required properties, cf.\ the amalgation argument in \cite[Lemmata III.6.A and III.6.B]{elworthy1982}, see also \cite{brzezniak1997}. If $U_0 \in L^p\left(\Omega; \Fc_0, V\right)$ with $p > 2$, the additional integrability \eqref{eq:solution.regularity.p} immediately follows from the construction of $\tau_N$, in particular from the fact that the stopping times $\tau_N$ are accessible by a finite number of extensions which, by the final argument of the proof of Proposition \ref{prop:local.path.existence}, preserves integrability in question.

It remains to establish the blow-up property \eqref{eq:explosion}. Let $\rho_R$ be the stopping time defined by
\[
	\rho_R = \inf \left\lbrace t \in \left[0, \xi\right) \mid \sup_{s \in [0, t]} \Vert U \Vert^2 + \int_0^t \vert AU \vert^2 \, ds \geq R \right\rbrace \wedge \xi.
\]
Let us observe that $\left\lbrace \xi < \infty \right\rbrace = \Omega_1 \cup \Omega_2$, where the disjoint sets $\Omega_1$ and $\Omega_2$ are given by
\[
	\Omega_1 = \left\lbrace \xi < \infty \right\rbrace \cap \left\lbrace \rho_R < \xi \ \text{for all} \ R > 0 \right\rbrace, \quad \Omega_2 = \left\lbrace \xi < \infty \right\rbrace \cap \left\lbrace \rho_R = \xi \ \text{for some} \ R > 0 \right\rbrace.
\]
Clearly, on the set $\Omega_1$ the blow-up property \eqref{eq:explosion} holds. Let us show that $\Pb \left( \Omega_2 \right) = 0$. If $\Pb \left( \Omega_2 \right) > 0$, observing
\[
	\Omega_2 = \bigcup_{R \in \Nb} \left\lbrace \xi < \infty \right\rbrace \cap \left\lbrace \rho_R = \xi \right\rbrace = \bigcup_{R \in \Nb} \Omega_2^R,
\]
we find $R_0 \in \Nb$ such that $\Pb\left( \Omega_2^{R_0} \right) > 0$. In particular, we have
\begin{equation}
	\label{eq:maximal.existence.bound}
	\sup_{s \in [0, \xi)} \Vert U \Vert^2 + \int_0^\xi \vert AU \vert^2 \, ds \leq R_0
\end{equation}
on a set of positive measure. From the definition it is immediate that $\left\lbrace \xi < \infty \right\rbrace \in \Fc_{\xi}$. Also by e.g.\ \cite[Lemma I.2.15]{karatzas1991} $\left\lbrace \rho_{R_0} = \xi \right\rbrace \in \Fc_\xi$ and therefore $\Omega_{2}^{R_0} \in \Fc_\xi$. Let $X$ be the stochastic process defined by
\begin{equation}
	\label{eq:blowup.x}
	X(t) = U_0 + \int_0^{t} \mathds{1}_{\left[ 0, \xi \right)} \mathds{1}_{\Omega_2^{R_0}} \left( - AU - B(U) - F(U) \right) \, ds + \int_0^t \mathds{1}_{\left[ 0, \xi \right)} \mathds{1}_{\Omega_2^{R_0}} \sigma(U) \, dW
\end{equation}
for $t \in [0, \infty)$ and $\omega \in \Omega$. The stochastic integral is well defined due to the bound \eqref{eq:maximal.existence.bound} and the assumption \eqref{eq:sigma.bnd.V} and takes values in $V$. The deterministic integral can be definied pathwise. From \eqref{eq:blowup.x} we observe
\begin{equation}
	\label{eq:U.and.X}
	X(t) = U(t) \quad \text{for all} \ 0 \leq t < \xi \ \Pb\text{-a.s.\ on the set} \ \Omega_2^{R_0}.
\end{equation}
We want to establish that the trajectories of the process $X$ are continuous in $V$ $\Pb$-almost surely. On $\Omega \setminus \Omega_2^{R_0}$ this is immediate. Indeed, by the It\^{o} Lemma, see e.g.\ Theorem \ref{thm:ito.lemma}, we get
\begin{multline*}
	d\Vert X \Vert^2 + 2 \mathds{1}_{\Omega_2^{R_0}} \mathds{1}_{[0, \xi)} \left[ \left( AU, AX \right) + \left(B(U), AX \right) + \left(F(U), AX \right) \right] \, dt\\
	\leq \mathds{1}_{\Omega_2^{R_0}} \mathds{1}_{[0, \xi)} \Vert \sigma(U) \Vert_{L_2\left(\Uc, V \right)}^2 \, dt + 2 \left( \mathds{1}_{\Omega_2^{R_0}} \mathds{1}_{[0, \xi)} A^{1/2} \sigma(U) \, dW, A^{1/2} X \right).
\end{multline*}
Let $t > 0$. In the following estimates we implicitly use the bound by $R_0$ \eqref{eq:maximal.existence.bound} and the equivalence of $X$ and $U$ on $\Omega^{R_0}_2$ from \eqref{eq:U.and.X}. Clearly
\[
	2 \Eb \int_0^t \mathds{1}_{\Omega_2^{R_0}} \mathds{1}_{[0, \xi)} \left( AU, AX \right) \, ds = 2 \Eb \int_0^{t \wedge \xi} \mathds{1}_{\Omega_2^{R_0}} \vert AU \vert^2 \, ds \leq 2 R_0 < \infty.
\]
By the estimate on $B$ \eqref{eq:b.estimate1} we have
\[
2 \Eb \int_0^t \mathds{1}_{\Omega_2^{R_0}} \mathds{1}_{[0, \xi)} \left| \left(B(U), AX \right) \right| \, ds \leq C \Eb \int_0^{t \wedge \xi} \mathds{1}_{\Omega_2^{R_0}} \Vert U \Vert \vert AU \vert^2 \, ds \leq C R_0^{1 + 1/2} < \infty.
\]
From the bound \eqref{eq:F.bounded} on $F$ we deduce
\begin{equation*}
	2 \Eb \int_0^t \mathds{1}_{\Omega_2^{R_0}} \mathds{1}_{[0, \xi)} \left| \left(F(U), AX \right) \right| \, ds \leq C \Eb \int_0^{t \wedge \xi} \mathds{1}_{\Omega_2^{R_0}} \left(1 + \Vert U \Vert \right) \vert AU \vert \, ds \leq C_t \left(1 + R_0 \right) < \infty.
\end{equation*}
The estimate \eqref{eq:sigma.bnd.V} on $\sigma(U)$ in $L_2\left( \Uc, V \right)$ leads to
\begin{equation*}
	\Eb \int_0^t \mathds{1}_{\Omega_2^{R_0}} \mathds{1}_{[0, \xi)} \Vert \sigma(U) \Vert^2 \, ds	\leq C \Eb \int_0^{t \wedge \xi} \mathds{1}_{\Omega_2^{R_0}} \left( 1 + \Vert U \Vert^2 + \vert AU \vert^2 \right) \, ds \leq C_t \left(1 + R_0 \right) < \infty.
\end{equation*}
Finally we use the Burkholder-Davis-Gundy inequality \eqref{eq:bdg} and \eqref{eq:sigma.bnd.V} once more to deduce
\begin{multline*}
	2 \Eb \sup_{s \in [0, t]} \left| \int_0^s \mathds{1}_{\Omega_2^{R_0}} \mathds{1}_{[0, \xi)} \left( A^{1/2} \sigma(U) \, dW, A^{1/2} X \right) \right|\\
	\leq C \Eb \left( \int_0^{t \wedge \xi} \mathds{1}_{\Omega_2^{R_0}} \Vert U \Vert^{2} \left( 1 + \Vert U \Vert^2 + \vert AU \vert^2 \right) \, ds \right)^{1/2} \leq C_t \left(1 + R_0 \right) < \infty.
\end{multline*}
Recalling the estimates above we may repeat the deterministic and stochastic maximal regularity-type argument as in the proof of Proposition \ref{prop:global.martingale.existence} to show that for all $t > 0$
\[
	X \in C\left( [0, t]; V \right) \cap L^2 \left( 0, t; D(A) \right) \quad \Pb\text{-a.s.\ on} \ \Omega_2^{R_0}.
\]
Therefore, by \eqref{eq:U.and.X} the limit
\[
	\lim_{t \to \xi-} U(t) = \lim_{t \to \xi-} X(t) = X\left(\xi\right)
\]
exists $\Pb$-almost surely on the set $\Omega_2^{R_0}$. Let us define
\[
	\tilde{U}_0(\omega) =
		\begin{cases}
			\lim_{t \to \xi(\omega)-} U(t, \omega), & \omega \in \Omega_2^{R_0} \ \text{and the limit exists},\\
			0, & \text{otherwise}.
		\end{cases}					
\]
We claim that the function $\tilde{U}_0$ is $\Fc_\xi$-measurable. Indeed, the convergence $\tau_N \nearrow \xi$ $\Pb$-a.s.\ gives
\[
	\tilde{U}_0(\omega) = \lim_{t \to \xi-} \mathds{1}_{\Omega^{R_0}_2}(\omega) U(t, \omega) = \lim_{N \to \infty} \mathds{1}_{\Omega^{R_0}_2}(\omega) U\left( \tau_N(\omega), \omega \right) \quad \Pb\text{-a.s.\ on} \ \Omega_2^{R_0}.
\]
Since $U(\tau_n(\cdot), \cdot)$ is $\Fc_{\tau_N}$-measurable, the random variable $\tilde{U}_0$ is $\Fc_\xi$-measurable. By Proposition \ref{prop:local.path.existence} applied to the stochastic basis $\left( \Omega, \Fc, \lbrace \Fc_{t+\xi} \rbrace_{t \geq 0}, \Pb\right)$ and the $\lbrace \Fc_{t+\xi} \rbrace_{t \geq 0}$-adapted cylindrical Wiener process $W_{\cdot + \xi}$ there exists a local pathwise solution $(\tilde{U}, \tilde{\tau})$ to \eqref{eq:pe.abstract} with initial data $\tilde{U}_0$. Let $N \in \Nb$ be fixed. It is straightforward to check that
\[
	\hat{U}(t, \omega) =
		\begin{cases}
			U(t, \omega), & \omega \in \Omega, \ 0 \leq t \leq \tau_N,\\
			U(t, \omega), & \omega \in \Omega_2^{R_0}, \ \tau_N \leq t < \xi,\\
			\tilde{U}(t - \xi(\omega), \omega), & \omega \in \Omega_2^{R_0}, \ \xi \leq t \leq \xi + \tilde{\tau},
		\end{cases}
\]
is a local pathwise solution of the problem \eqref{eq:pe.abstract} with initial data $U_0$, which is a contradiction with the definition of the maximal solution and the stopping time $\xi$. \qed

\section{Existence of global solution}
\label{sect:global.solution}

In this section let $\Sc$, $W$, $U_0$, $F_U$ and $\sigma$ be as in Theorem \ref{thm:global.existence} and let $\left(U, \xi\right)$ be the maximal solution from Theorem \ref{thm:maximal.existence} with initial data $U_0$. Let $\tau_N$ be the sequence of $\Fb$-adapted stopping times satisfying $\tau_N \nearrow \xi$ $\Pb$-almost surely from Theorem \ref{thm:maximal.existence}.

The strategy of the proof of global existence is similar to the one in \cite{cao2007}. We decompose the horizontal velocity $v$ into its baroclinic and barotropic modes. The equation for the baroclinic mode does not contain pressure and using the fact that the operator $-\laplace_3$ is dissipative in $L^6(\Mc)$ we may estimate $\vert \vt \vert_{L^6}^6$. This is done by employing the stochastic Gronwall Lemma from \cite{glatt-holtz2009} recalled in Proposition \ref{prop:stochastic.gronwall}. After that we use the above estimate to obtain estimates on $\Vert \vb \Vert_{\Vo}^2$, $\vert T \vert_{L^6}^6$ and $\vert \partial_z U \vert^2 \Vert \partial_z U \Vert^2$. Finally, following the argument from \cite{debussche2012} we establish the global existence.

The decomposition of the horizontal velocity into barotropic and baroclinic modes $\bar{v}$ and $\tilde{v}$ respectively is done by defining
\[
	\bar{v} = \Ac_2 v, \qquad \tilde{v} = \left(I - \Ac_3\right) v = \Rc v,
\]
where $\Ac_2, \Ac_3$ are the averaging operators defined in \eqref{eq:averaging.operators}. Using the computations from \cite{cao2007} we observe that the barotropic mode satisfies the equation
\begin{multline}
	\label{eq:vb.barotropic}
	d\vb + \bigg[ -\mu \laplace \vb + \left(\vb \cdot \grad\right) \vb + \Ac_2 \left( \left(\vt \cdot \grad\right) \vt + (\div \vt) \vt \right) + f \vec{k} \times \vb + \grad p_s\\
	- \beta_T g \Ac_2 \grad \int_{z}^0 T \, dz' \bigg] \, dt = \Ac_2 F_v \, dt + \Ac_2 \sigma_1(U) \, dW_1, \qquad \vb(0) = \Ac_2 v_0,
\end{multline}
with
\[
	\div \vb = 0 \ \text{in}\ \Mc_0, \qquad \vb = 0 \ \text{on} \ \partial \Mc_0,
\]
with the equation \eqref{eq:vb.barotropic} understood in the space $H_f$. On the other hand, the baroclinic mode of the horizontal velocity solves the equation
\begin{equation}
	\label{eq:vt.baroclinic}
	\begin{split}
	d\vt &+ \bigg[ -\mu \laplace \vt - \nu \partial_{zz} \vt + \left(\vt \cdot \grad\right) \vt +w(\vt) \partial_z \vt + \left(\vt \cdot \grad\right) \vb + \left(\vb \cdot \grad\right) \vt\\
	&\hphantom{+ \bigg[ \ }- \Ac_3\left( \left(\vt \cdot \grad\right) \vt + \left(\div \vt\right) \vt \right) + f \vec{k} \times \vt - \beta_T g \Rc \grad \left( \int_{z}^0 T \, dz' \right) \bigg] \, dt\\
	&= \Rc F_v \, dt + \Rc \sigma_1(U) \, dW_1, \qquad \vt(0) = \Rc v_0,
	\end{split}
\end{equation}
with the boundary conditions
\[
	\partial_z \vt = 0 \ \text{on} \ \Gamma_i \cup \Gamma_b, \quad \vt = 0 \ \text{on} \ \Gamma_l.
\]

Before we proceed to the estimates let us state the following simple lemma. For a different argument see e.g.\ \cite[Proposition A.1]{glatt-holtz2011}.

\begin{lemma}
\label{lemma:stopping.time.infinity}
Let $f: \Omega \times [0, \infty) \to [0, \infty]$ be such that the function $f(\cdot, \omega)$ is non-decreasing continuous for $\Pb$-almost all $\omega \in \Omega$ and $f(t, \cdot)$ is measurable and almost surely finite for all $t \geq 0$. For $K \geq 0$ let $\rho_K(\omega) = \inf \left\lbrace t \geq 0 \mid f(t, \omega) \geq K \right\rbrace$. Then $\lim_{K \to \infty} \rho_K = \infty$ $\Pb$-almost surely.
\end{lemma}

\begin{proof}
For $N \in \Nb$ let $\Omega_N = \left\lbrace \omega \in \Omega \mid f(N, \omega) < \infty \right\rbrace$ and $\tilde{\Omega} = \bigcap_{N = 1}^{\infty} \Omega_N$. By the assumptions the set $\tilde{\Omega}$ is of full measure. Observe that for $0 \leq K_1 \leq K_2$ we have $\rho_{K_1} \leq \rho_{K_2}$, i.e.\ $\rho_K$ is monotone in $K$. Assume that $\rho_K \nrightarrow \infty$ on some measurable set $\Omega^1$ such that $\Pb \left( \Omega^1 \right) > 0$. By the monotonicity of $\rho_K$ there exists $t_0 \in \Nb$ such that $\rho_K \leq t_0$ for all $K > 0$ on $\Omega^2 \subseteq \Omega^1$ such that $\Pb \left( \Omega^2 \right) > 0$. From the definition of $\rho_K$ and the continuity of $f$ we get $f\left(\omega, \rho_K(\omega) \right) = K$ for all $K > 0$ and $\omega \ni \Omega_2$. Then since $f$ is continuous and non-decreasing, we have $f(t_0) = +\infty$ on a set of non-zero measure, a contradiction.
\end{proof}

\begin{proposition}
\label{prop:l2.estimate}
Let $p \geq 2$ and let $U_0 \in L^p\left(\Omega; \Fc_0, H\right)$. Let us the constant $\eta_0$ from \eqref{eq:sigma.bnd.H} is such that
\begin{equation}
	\label{eq:eta.0}
	\eta_0 < \frac{2}{p\left( 1 + \frac{C_{BDG}^2}{2} \right) - 1}.
\end{equation}
Then for all $t > 0$ we have
\begin{equation}
	\label{eq:l2.estimate}
	\Eb \left[ \sup_{s \in \left[0, t \wedge \xi\right)} \vert U \vert^p + \int_0^{t \wedge \xi} \vert U \vert^{p-2} \Vert U \Vert^2 \, ds \right] \leq C_t \Eb \left[ \vert U_0 \vert^p + 1 \right].
\end{equation}
If moreover $U_0 \in L^4\left( \Omega; \Fc_0, H \right)$ and $\eta_0$ is as in Theorem \ref{thm:global.existence}, so that in particular the condition \eqref{eq:eta.0} holds with $p = 4$, the stopping time $\tau_K^w$ defined by
\begin{equation}
	\label{eq:weak.stopping.time}
	\tau_K^w = \inf \left\lbrace s \geq 0 \mid \ \int_0^{s \wedge \xi} \vert U \vert^{2} \Vert U \Vert^2 + \Vert U \Vert^2 + \vert F_U \vert^2  + \vert F_T \vert^2_{H^{1/2}\left( \Mc \right)} \, dr \geq K \right\rbrace
\end{equation}
satisfies $\tau_K^w \to \infty$ $\Pb$-almost surely as $K \to \infty$.
\end{proposition}

\begin{proof}
We employ the It\^{o} Lemma from Theorem \ref{thm:ito.lemma} in $H$. We use the cancellation property of $b$ \eqref{eq:b.zero} and the cross product, the self-adjointness of $A$, the Lipschitz continuity of $F$ and the bound on $\sigma$ \eqref{eq:sigma.bnd.H} in $L_2(\Uc, H)$ to obtain
\begin{multline*}
	\left( \delta - \varepsilon \right) \Eb \sup_{s \in \left[0, t \wedge \tau_N \right]} \vert U \vert^p + p \left( 1 - \varepsilon - \frac{p-1}{2} - \frac{p c_{BDG}^2 \eta_0}{4(1-\delta)} \right) \Eb \int_0^{t \wedge \tau_N} \Vert U \Vert^2 \vert U \vert^{p-2} \, ds\\
	\leq C_{\varepsilon} \Eb \left[ \vert U_0 \vert^p + \int_0^{t \wedge \tau_N} 1 + \vert U \vert^p + \vert F_U \vert^2 \, ds \right]
\end{multline*}
for some $\varepsilon > 0$ and $\varepsilon < \delta < 1$. Recalling that $\eta_0$ satisfies \eqref{eq:eta.0}, the standard Gronwall Lemma leads to
\[
	\Eb \left[ \sup_{s \in \left[0, t \wedge \tau_N \right]} \vert U \vert^p + \int_0^{t \wedge \rho_M} \vert U \vert^{p-2} \Vert U \Vert^2 \, ds \right] \leq C_t \left( \Eb \vert U_0 \vert^p + \int_0^{t \wedge \tau_N} 1 + \vert F_U \vert^2 \, ds \right).
\]
The desired bound \eqref{eq:l2.estimate} then follows by passing to the limit w.r.t.\ $N \to \infty$ justified by the assumption on $F_U$ \eqref{eq:f.source} and the monotone convergence theorem. The convergence $\lim_{K \to \infty} \tau_K^w = \infty$ $\Pb$-a.s.\ now follows immediately from Lemma \ref{lemma:stopping.time.infinity}.
\end{proof}

\subsection{$L^6$ estimates}

\begin{proposition}
\label{prop:vt.l6.estimate}
Under the assumptions of Theorem \ref{thm:global.existence}, for all $t > 0$ and $K, N \in \Nb$ the baroclinic mode $\vt$ from \eqref{eq:vt.baroclinic} satisfies
\begin{multline}
	\label{eq:vt.l6.estimate}
	\Eb \left[ \sup_{s \in \left[0, t \wedge \tau^w_K \wedge \tau_N \right]} \vert \vt \vert^6_6 + \int_0^{t \wedge \tau_K^w \wedge \tau_N} \int_{\Mc} \vert \grad_3 \vt \vert^2 \vert \vt \vert^4 \, d\Mc \, ds \right]\\
	\leq C_{t, K} \Eb \left[ \vert \vt(0) \vert_{L^6}^6 + \int_0^{t \wedge \tau_K^w \wedge \tau_N} 1 + \Vert U \Vert^2 + \Vert U \Vert^2 \vert U \vert^2 + \vert F_v \vert^2 \, ds \right].
\end{multline}
Moreover, the stopping time $\tau^{\vt}_{K}$ defined by
\begin{equation}
	\label{eq:vt.stopping.time}
	\tau^{\vt}_{K} = \inf \left\lbrace s \geq 0 \mid \ \int_0^{s \wedge \xi} \vert \vt \vert^6_6 + \left( \int_{\Mc} \vert \grad_3 \vt \vert^2 \vert \vt \vert^4 \, d\Mc \right) \, dr \geq K \right\rbrace
\end{equation}
satisfies $\tau^{\vt}_{K} \to \infty$ $\Pb$-almost surely.
\end{proposition}

For another argument using dissipative properties of the Laplacian in the context of stochastic fluid mechanics see e.g.\ the paper \cite{brzezniak2001} on the stochastic 2D Euler equations.

\begin{proof}
Using the Gagliardo-Nirenberg inequality it is straightforward to check that for $u \in H^2$ and $h \in L^2$
\begin{align}
	\nonumber
	\left| D\vert u \vert_{L^6}^6 (h) \right| &\leq C \vert u^5 \vert_{L^2} \vert h \vert \leq C \vert u^3 \vert_{L^{10/3}}^{5/3} \vert h \vert \leq C \vert u^3 \vert^{2/3} \vert \grad_3 u^3 \vert \vert h \vert\\
	\label{eq:vt.l6.extension}
	&\leq C \vert u \vert_{L^6}^2 \vert \grad_3 u^3 \vert \vert h \vert \leq C \vert u \vert_{L^6}^4 \vert u \vert_{H^2} \vert h \vert
\end{align}
and therefore the operator $D\vert u \vert^6_{L^6}$ can be continuously extended to $L^2$ for $u \in D(A_1)$. Moreover, if $u_n \to u$ in $C\left([0, t], V\right)$ and $u_n$ is bounded\footnote{For $\psi(v) = \vert v \vert^6_{L^6}$ it is sufficient to assume boundedness of $u_n$ instead of convergence in the space $L^2\left( 0, T; W \right)$ to get the required convergence \eqref{eq:ito.D.psi.convergence}.} in $L^2\left(0, t; D(A)\right)$, then by the Gagliardo-Nirenberg inequality
\begin{align*}
	\bigg| \int_0^t [ &D\vert u_n \vert^6_{L^6} - D\vert u \vert^6_{L^6}] (h) \, ds \bigg| \leq C \int_0^t \int_{\Mc} \left|u_n - u \right| \left( \vert u_n \vert^4 + \vert u \vert^4 \right) \vert h \vert \, d\Mc \, ds\\
	&\leq C \int_0^t \left( \vert u_n \vert_{L^{12}}^4 + \vert u \vert^4_{L^{12}} \right) \vert u - u_n \vert_{L^6} \vert h \vert \, ds\\
	&\leq C \int_0^t \left( \vert u_n \vert_{H^2} + \vert u \vert_{H^2} \right) \left( \Vert u_n \Vert^3 + \Vert u \Vert^3 \right) \Vert u - u_n \Vert \vert h \vert \, ds\\
	&\leq C \sup_{s \in [0, t]} \left[ \Vert u_n - u \Vert \left( \Vert u_n \Vert^3 + \Vert u \Vert^3 \right)\right] \left( \Vert u \Vert_{L^2\left(0, t; H^2\right)} + \Vert u_n \Vert_{L^2\left(0, t; H^2\right)} \right) \vert h \vert_{L^2\left(0, t; H \right)} \to 0,
\end{align*}
and therefore the assumptions of Theorem \ref{thm:ito.lemma} are met. Thus, by the It\^{o} Lemma from Theorem \ref{thm:ito.lemma} applied to the equation \eqref{eq:vt.baroclinic} and the function $\vert \cdot \vert^6_{L^6}$ and the usual cancellations we may use integration by parts to get
\begin{align}
	\nonumber
	d \vert \vt \vert_{L^6}^6 &+ 6 \int_{\Mc} \mu \vert \grad \vt \vert^2 \vert \vt \vert^4 + \nu \vert \partial_z \vt \vert^2 \vert \vt \vert^4 \, d\Mc \, dt\\
	\nonumber
	&= - 6 \int_{\Mc} \vert \vt \vert^4 \vt \cdot \left( (\vt \cdot \grad) \vb + \Ac_3\left( (\vt \cdot \grad) \vt + (\div \vt) \vt \right) + \beta_T g \Rc \grad \int_{z}^0 T \, dz' \right) \, d\Mc \, dt \\
	\nonumber
	&\hphantom{= \ } + 6 \int_{\Mc} \Rc F_v \cdot \vert \vb \vert^4 \vb \, d\Mc \, dt + 15 \sum_{k=1}^\infty \int_{\Mc} \vert \vt \vert^4 \left( \Rc \sigma_1(U) e_k \right)^2 \, d\Mc \, dt\\
	\nonumber
	&\hphantom{= \ }  + 6 \sum_{k=1}^\infty \int_{\Mc} \vert \vt \vert^4 \vt \Rc \sigma_1(U) e_k \, d\Mc \, dW_1^k\\
	\label{eq:vt.after.ito}
	&= I_1 \, dt + I_2 \, dt + I_3 \, dt + \sum_{k=1}^\infty I_4^k \, dW_1^k.
\end{align}

To estimate the integral $I_1$ we proceed as in \cite[Section 3.2]{cao2007} and use the boundedness of the operators $\Ac$ and $\Rc$ from \eqref{eq:averaging.operators} to get
\begin{equation}
	\left| I_1 \right| \leq (\mu \wedge \nu) \int_{\Mc} \vert  \grad_3 \vt \vert^2 \vert \vt \vert^4 \, d\Mc + C \vert \vt \vert_{L^6}^6 \left( \vert v \vert^2 + 1\right) \left( \Vert v \Vert^2 + 1 \right) + C \vert T \vert^2 \Vert T \Vert^2.
\end{equation}
Let $\tau_a$ and $\tau_b$ be stopping times such that $0 \leq \tau_a \leq \tau_b \leq t \wedge \tau_N \wedge \tau_K^w$. We estimate the integral $I_2$ by the Gagliardo-Nirenberg inequality similarly as in \eqref{eq:vt.l6.extension} by
\begin{equation}
	\begin{split}
		\left| \int_{\tau_a}^{\tau_b} I_2 \, ds \right| &\leq \int_{\tau_a}^{\tau_b} \vert \Rc F_v \vert \vert \vt^5 \vert \, ds = \int_{\tau_a}^{\tau_b} \vert \Rc F_v \vert \vert \vt^3 \vert_{L^{10/3}}^{5/3} \, ds\\
		&\leq (\mu \wedge \nu) \int_{\tau_a}^{\tau_b} \int_{\Mc} \vert \grad_3 \vt \vert^2 \vert \vt \vert^4 \, d\Mc + C \vert F_v \vert^2 \left( \vert \vt \vert_{L^6}^6 + 1 \right) \, ds.
	\end{split}
\end{equation}
The integral $I_3$ can be estimated using \eqref{eq:vt.l6.extension} as
\begin{equation}
	\begin{split}
		\left| I_3 \right| &\leq \vert \vt \vert^{4}_{L^6} \sum_{k = 1}^\infty \vert \Rc \sigma_1(U) e_k \vert_{L^6}^2 \leq C \vert \vt \vert^{4}_{L^6} \left( 1 + \vert \vt \vert^2_{L^6} \right) \left( 1 + \Vert U \Vert^2 \right)\\
		&\leq C \left( 1 + \Vert U \Vert^2 \right) \left( 1 + \vert \vt \vert_{L^6}^6 \right).
	\end{split}
\end{equation}
We deal with the stochastic integral by the means of the Burkholder-Davis-Gundy inequality \eqref{eq:bdg} and the structural assumption \eqref{eq:sigma.rc.vt}. We get
\begin{equation}
	\label{eq:vt.l6.i4}
	\begin{split}
		\Eb \sup_{s \in \left[\tau_a, \tau_b \right]} \Bigg| \int_{\tau_a}^{\tau_b} &\sum_{k=1}^\infty I_4 \, dW^k_1 \Bigg| \leq C \Eb \left( \int_{\tau_a}^{\tau_b} \sum_{k=1}^\infty \left( \int_{\Mc} \vert \vt \vert^5 \vert \Rc \sigma_1(U) e_k\vert \, d\Mc \right)^2 \, ds \right)^{1/2}\\
		&\leq C  \Eb \left( \int_{\tau_a}^{\tau_b} \vert \vt \vert_{L^6}^{10} \sum_{k=1}^\infty \vert \Rc \sigma_1(U) e_k \vert_{L^6}^2 \, ds \right)^{1/2}\\
		&\leq C \Eb \left[ \left( \sup_{s \in \left[\tau_a, \tau_b\right]} \vert \vt \vert_{L^6}^3 \right) \left( \int_{\tau_a}^{\tau_b} \vert \vt \vert^4_{L^6} \sum_{k=1}^\infty \vert \Rc \sigma_1(U) e_k \vert_{L^6}^2 \, ds \right)^{1/2} \right]\\
	&\leq  \Eb \left[ \frac12 \sup_{s \in \left[\tau_a, \tau_b\right]} \vert \vt \vert_{L^6}^6 + C \int_{\tau_a}^{\tau_b} \left( 1 + \vert \vt \vert_{L^6}^6 \right) \left(1 + \Vert U \Vert^2 \right) \, ds  \right].
	\end{split}
\end{equation}
Collecting the estimates \eqref{eq:vt.after.ito}-\eqref{eq:vt.l6.i4} we obtain
\begin{multline*}
	\Eb \Bigg[ \sup_{s \in \left[\tau_a, \tau_b\right]} \vert \vt \vert_{L^6}^6 + \int_{\tau_a}^{\tau_b} \int_{\Mc}\vert \grad_3 \vt \vert^2 \vert \vt \vert^4 \, d\Mc \, ds \Bigg]\\
	\leq C \Eb \left[ \vert \vt(\tau_a) \vert_{L^6}^6 + \int_{\tau_a}^{\tau_b} \left(1 + \vert \vt \vert_{L^6}^6 \right) \left( 1 + \Vert U \Vert^2 + \vert U \vert^2 \Vert U \Vert^2 + \vert F_v \vert^2 \right) \, ds \right]
\end{multline*}
and by the stochastic Gronwall Lemma from Proposition \ref{prop:stochastic.gronwall} we obtain \eqref{eq:vt.l6.estimate}. Indeed, the use of the stochastic Gronwall Lemma is justified by the regularity of the maximal solution in \eqref{eq:solution.regularity.p} with $U_0 \in L^6\left( \Omega; \Fc_0, V \right)$ and Hypothesis $H_6$, see \eqref{eq:small.constants.maximal}. The right-hand side of \eqref{eq:vt.l6.estimate} can be estimated by
\[
	C_{t, K} \Eb \left[ \vert \vt(0) \vert_{L^6}^6 + \int_0^{t \wedge \tau_K^w \wedge \xi} 1 + \Vert U \Vert^2 + \Vert U \Vert^2 \vert U \vert^2 + \vert F_v \vert^2 \, ds \right] < \infty,
\]
where the finiteness follows from Proposition \ref{prop:l2.estimate}, and thus we may use the monotone convergence theorem to pass to the limit w.r.t.\ $N \to \infty$. In particular, we obtain
\[
	\sup_{s \in \left[0, t \wedge \tau^w_K \wedge \xi\right)} \vert \vt \vert^6_6 + \int_0^{t \wedge \tau_K^w \wedge \xi} \int_{\Mc} \vert \grad_3 \vt \vert^2 \vert \vt \vert^4 \, d\Mc \, ds < \infty \quad \Pb\text{-a.s.}
\]
for all $K \in \Nb$ and $t > 0$. Since $\tau_K^w \to \infty$ $\Pb$-a.s., we have
\begin{equation*}
	\sup_{s \in \left[0, t \wedge \xi\right)} \vert \vt \vert^6_6 + \int_0^{t \wedge \xi} \int_{\Mc} \vert \grad_3 \vt \vert^2 \vert \vt \vert^4 \, d\Mc \, ds < \infty \quad \Pb\text{-a.s.}
\end{equation*}
for all $t > 0$. The convergence $\tau_K^{\vt} \to \infty$ $\Pb$-a.s.\ follows from the convergence of $\tilde{\tau}_K^{\vt} \to \infty$ $\Pb$-a.s., where
\[
	\tilde{\tau}^{\vt}_{K} = \inf \left\lbrace s \geq 0 \mid \ \sup_{r \in \left[0, s \wedge \xi\right)} \vert \vt \vert^6_6 + \int_0^{s \wedge \xi} \int_{\Mc} \vert \grad_3 \vt \vert^2 \vert \vt \vert^4 \, d\Mc \, dr \geq K \right\rbrace,
\]
which is established using Lemma \ref{lemma:stopping.time.infinity}.
\end{proof}

\subsection{$H^1$ estimates}

\begin{proposition}
\label{prop:grad.vb.estimate}
Let the assumptions of Theorem \ref{thm:global.existence} hold and let $\tau_K^1 = \tau_K^{w} \wedge \tau_K^{\vt}$ for $K > 0$. Then for all $t > 0$ and $K, N \in \Nb$ we have
\begin{multline}
	\label{eq:grad.vb.estimate}
	\Eb \left[ \sup_{s \in \left[0, t \wedge \tau_N \wedge \tau^1_{K} \right]} \Vert \vb \Vert^4 + \int_{0}^{t \wedge \tau_N \wedge \tau^1_{K}} \Vert \vb \Vert^2 \vert A_S \vb \vert^2 \, ds \right]\\
	\leq C_{t, K} \Eb  \left[ \Vert v_0 \Vert_V^4 + \int_0^{t \wedge \tau_N \wedge \tau_K^1} 1 + \Vert U \Vert_V^2 + \vert F_v \vert_H^2 \, ds  \right],
\end{multline}
where the symbols $\vert \cdot \vert$ and $\Vert \cdot \Vert$ denote the norms on $\Ho$ and $\Vo$, respectively. Moreover, the stopping time $\tau_K^{\grad \vb}$ defined by
\begin{equation}
	\label{eq:grad.vb.stopping.time}
	\tau^{\grad \vb}_K = \inf \left\lbrace s \geq 0 \mid \int_0^{s \wedge \xi} \Vert \vb \Vert_{H^1\left( \Mc_0, \Rb^2 \right)}^4 \, dr \geq K \right\rbrace
\end{equation}
satisfies $\tau^{\grad \vb}_{K} \to \infty$ $\Pb$-almost surely.
\end{proposition}

\begin{proof}
We apply the the It\^{o} Lemma from Theorem \ref{thm:ito.lemma} to the equation \eqref{eq:vb.barotropic} with the function $\vert A_S^{1/2} P_{\Ho} \cdot \vert^4$. Recalling that the Stokes operator $A_S$ is self-adjoint we obtain
\begin{equation}
	\label{eq:grad.vb.after.ito}
	\begin{split}
		d\Vert \vb \Vert^4 &+ 4 \Vert \vb \Vert^2 \vert A_S \vb \vert^2 \, dt\\
		&\leq 4 \Vert \vb \Vert^2 \left| \left( P_{\Ho} \Ac_2\left[ (\vt \cdot \grad) \vt + (\div \vt) \vt \right], A_S \vb \right) \right| \, dt + 4 \Vert \vb \Vert^2 \left| \left( P_{\Ho} [ f \vec{k} \times \vb ], A_S \vb \right) \right| \, dt\\
		&\hphantom{\leq \ } + 6 \Vert \vb \Vert^2 \Vert \Ac_2 \sigma_v(U) \Vert^2_{L_2(\Uc, \Vo)} \, dt + 4 \Vert \vb \Vert^2 \left| \left( P_{\Ho} \Ac_2 \int_{z}^0 T \, dz' , A_S \vb \right) \right| \, dt\\
		&\hphantom{\leq \ } + 4 \Vert \vb \Vert^2 \left| \left(P_{\Ho} \Ac_2 F_v, A_S \vb \right) \right| \, dt + 4 \Vert\vb \Vert^2 \left| \left( P_{\Ho} [(\vb \cdot \grad) \vt], A_S \vb \right) \right| \, dt\\
		&\hphantom{\leq \ } + 4 \Vert \vb \Vert^2 \left| \sum_{k=1}^\infty \left( A_S^{1/2} \Ac_2 \sigma_1(U) e_k, A_S^{1/2} \vb \right) \, dW^k_1 \right|\\
		&= \sum_{j=1}^{6} I_j \, dt + \left| \sum_{k=1}^{\infty} I_7^k   \, dW^k_1 \right|.
	\end{split}
\end{equation}
Let $\varepsilon > 0$ be fixed and precisely determined later. Recalling that $\vert \vt \vert_{L^2} \leq C \vert U \vert_H$ and $\Vert \vb \Vert, \vert \grad \vt \vert_{L^2} \leq C \Vert U \Vert_V$ we employ the argument of \cite[Section 3.3.1]{cao2007} we have
\begin{align}
		\nonumber
		I_1 &\leq C \Vert \vb \Vert^2 \vert \grad \vt  \vert_{L^2}^{1/2} \left( \int_{\Mc} \vert \vt \vert^4 \vert \grad_3 \vt \vert^2 \, d\Mc \right)^{1/4} \vert A_S \vb \vert\\
		&\leq \frac{\varepsilon}{3} \vert A_S \vb \vert^2 \Vert \vb \Vert^2 + C_\varepsilon \Vert U \Vert_V^2 + C_\varepsilon \Vert \vb \Vert^4 \left( \int_{\Mc} \vert \vt \vert^4 \vert \grad_3 \vt \vert^2 \, d\Mc \right),\\
		I_6 &\leq C \vert \vb \vert^{1/2} \Vert \vb \Vert^3 \vert A_S \vb \vert^{3/2} \leq \frac{\varepsilon}{3} \Vert \vb \Vert^2 \vert A_S \vb \vert^2 + C_\varepsilon \Vert U \Vert_V^2 \Vert \vb \Vert^4.
\end{align}
Let $\tau_a$ and $\tau_b$ be stopping times such that $0 \leq \tau_a \leq \tau_b \leq t \wedge \tau_N \wedge \tau^1_K$. The remaining deterministic terms can be estimated in a straightforward way by
\begin{multline}
	\int_{\tau_a}^{\tau_b} \sum_{j=2}^5 I_j \, ds \leq \left( \frac{\varepsilon}{3} + 6 \eta_2 \right) \int_{\tau_a}^{\tau_b} \Vert \vb \Vert^2 \vert A_S \vb \vert^2 \, ds\\
	+ C_\varepsilon \int_{\tau_a}^{\tau_b} \left( \Vert U \Vert_V^2 + \vert F_v \vert_H^2 \right) + C_\varepsilon \Vert \vb \Vert^4 \left( \Vert U \Vert_V^2 + \vert F_v \vert_H^2 \right) \, ds.
\end{multline}
Using estimates similar to the ones leading to \eqref{eq:Jp32}, the Burkholder-Davis-Gundy inequality and the assumption \eqref{eq:sigma.grad.ac} on $\Ac_2 \sigma_1$ we get for $\delta \in (0, 1)$
\begin{equation}
\label{eq:grad.vb.stochastic}
	\begin{split}
		4 \Eb \sup_{s \in \left[\tau_a, \tau_b\right]} &\Bigg| \int_{\tau_a}^{\tau_b} \sum_{k=1}^\infty I^k_7 \, dW^k_1 \Bigg| \leq 4 C_{BDG} \Eb \left( \int_{\tau_a}^{\tau_b} \Vert \vb \Vert^6 \Vert \Ac \sigma_1(U) \Vert^2_{L_2(\Uc, \Vo)} \, ds \right)^{1/2}\\	
		&\leq 4 C_{BDG} \Eb \left( \int_{\tau_a}^{\tau_b} \Vert \vb \Vert^6 \left( C \left(1 + \Vert U \Vert_V^2 \right) + \eta_2 \vert A_S \vb \vert^2 \right) \, ds \right)^{1/2}\\
		&\leq \left(1-\delta + \varepsilon \right) \Eb \sup_{s \in \left[\tau_a, \tau_b\right]} \Vert \vb \Vert^4 + \frac{4 C_{BDG}^2 \eta_2}{1-\delta} \Eb \int_{\tau_a}^{\tau_b} \Vert \vb \Vert^2 \vert A_S \vb \vert^2 \, ds\\
		&\hphantom{\leq \ } + C_{\varepsilon} \Eb \int_{\tau_a}^{\tau_b} \left( 1 + \Vert \vb \Vert^4 \right) \left(1 + \Vert U \Vert^2_V \right) \, ds
	\end{split}
\end{equation}
Collecting the estimates \eqref{eq:grad.vb.after.ito}-\eqref{eq:grad.vb.stochastic}, choosing $\delta$ and $\varepsilon$ sufficiently small we use assumption \eqref{eq:small.constants.global} on $\eta_2$ to deduce
\begin{multline*}
	\Eb \left[ \sup_{s \in \left[\tau_a, \tau_b\right]} \Vert \vb \Vert^4 + \int_{\tau_a}^{\tau_b} \Vert \vb \Vert^2 \vert A_S \vb \vert^2 \, ds \right] \leq C \Eb \Vert \vb(\tau_a) \Vert^4 + C \Eb \int_{\tau_a}^{\tau_b} \Vert U \Vert_V^2 + \vert F_v \vert_H^2 \, ds\\
	+ C \Eb \int_{\tau_a}^{\tau_b} \Vert \vb \Vert^4 \left( \Vert U \Vert_H^2 + \vert F_v \vert_H^2 + \int_{\Mc} \vert \vt \vert^4 \vert \grad \vt \vert^2 \, d\Mc \right) \, ds.
\end{multline*}
The estimate \eqref{eq:grad.vb.estimate} is then obtained by the stochastic Gronwall Lemma from Proposition \ref{prop:stochastic.gronwall} which is justified by the assumption \eqref{eq:F.bounded} on $F$ and the definitions of the stopping times $\tau^w_K$ and $\tau^{\vt}_K$ from \eqref{eq:weak.stopping.time} and \eqref{eq:vt.stopping.time}, respectively. The convergence $\tau_K^{\grad \vb} \to \infty$ $\Pb$-a.s.\ as $K \to \infty$ can be shown from the convergence $\tilde{\tau}_K^{\grad \vb} \to \infty$ $\Pb$-a.s., where
\[
	\tilde{\tau}^{\grad \vb}_K = \inf \left\lbrace s \geq 0 \mid \sup_{r \in \left[0, s \wedge \xi\right)} \Vert \vb \Vert^4 + \int_0^{s \wedge \xi} \Vert \vb \Vert^2 \vert A_S \vb \vert^2 \, dr \geq K \right\rbrace,	
\]
similarly as in the proof of Proposition \ref{prop:vt.l6.estimate} from the estimate \eqref{eq:grad.vb.estimate}, the monotone convergence theorem and Lemma \ref{lemma:stopping.time.infinity}.
\end{proof}

\begin{proposition}
\label{prop:vz.estimates}
Let the assumptions of Theorem \ref{thm:global.existence} hold and let $\tau^2_K = \tau_K^w \wedge \tau_K^{\vt} \wedge \tau_K^{\grad \vb}$ for $K > 0$. Then for all $t > 0$, $p \in [2, 4]$ and $K, N \in \Nb$ the following estimate holds
\begin{multline}
	\label{eq:vz.estimate}
	\Eb \left[ \sup_{s \in \left[0, t \wedge \tau_N \wedge \tau_K^{2} \right]} \vert \partial_z v \vert^p + \int_0^{t \wedge \tau_N \wedge \tau_K^{2}} \vert \partial_z v \vert^{p-2} \vert \grad_3 \partial_z v \vert^2 \, ds \right]\\
	\leq C_{t, K, p} \Eb \left[ \Vert v_0 \Vert^p + \int_0^{t \wedge \tau_N \wedge \tau_K^{2}} 1 + \vert F_v \vert^2 + \Vert U \Vert^2 \, ds \right].
\end{multline}
Moreover, the stopping time $\tau_K^{\partial_z v}$ defined by
\begin{equation}
	\label{eq:vz.stopping.time}
	\tau^{\partial_z v}_K = \inf \left\lbrace s \geq 0 \mid \sup_{r \in \left[0, s \wedge \xi\right)} \int_0^{s \wedge \xi} \vert \grad_3 \partial_z v \vert^2 + \vert \partial_z v \vert^2 \vert \grad_3 \partial_z v \vert^2 \, dr \geq K \right\rbrace
\end{equation}
satisfies $\tau_K^{\partial_z v} \to \infty$ as $K \to \infty$ $\Pb$-almost surely.
\end{proposition}

\begin{proof}
Using the identity
\[
	\partial_z \left[ \left(v \cdot \grad\right) v + w(v) \partial_z v \right] = \left(\partial_z v \cdot \grad\right) v + \left(v \cdot \grad\right) \partial_z v - \left(\div v \right) \partial_z v + w(v) \partial_{zz} v
\]
and the cancellation
\[
	\left( \left(v \cdot \grad \right) \partial_z v + w(v) \partial_{zz} v, \partial_z v \right) = 0,
\]
by the It\^{o} Lemma from Theorem \ref{thm:ito.lemma} applied to the equation for $v$ and $\vert \partial_z \cdot \vert_{L^2}^p $we get the estimate
\begin{equation}
	\label{eq:vz.after.ito}
	\begin{split}
		d\vert \partial_z v \vert^p &+ p \left( \mu \wedge \nu \right) \vert \partial_z v \vert^{p-2} \vert \grad_3 \partial_z v \vert^2 \, dt\\
		&\leq p \vert \partial_z v \vert^{p-2} \left| \left( \left( \partial_z v \cdot \grad \right) v, \partial_z v \right) \right| \, dt + p \vert \partial_z v \vert^{p-2} \left| \left(  \left( \div v \right) \partial_z v, \partial_z v \right) \right| \, dt\\
		&\hphantom{= \ } + p \vert \partial_z v \vert^{p-2} \left| \left( g\beta_T \grad T, \partial_z v \right) \right| \, dt + \tfrac{p(p-1)}{2} \vert \partial_z v \vert^{p-2} \Vert \partial_z \sigma_1(v, T, S) \Vert_{L_2(\Uc, L^2)}^2 \, dt\\
		&\hphantom{= \ }+ p \vert \partial_z v \vert^{p-2} \left| \left( \partial_z F_v, \partial_z v\right) \right| \, dt + p \vert \partial_z v \vert^{p-2} \left| \sum_{k=1}^\infty \int_{\Mc} \partial_z \sigma_1(U) e_k \partial_z v \, d\Mc \, dW_1^k \right| \\
		&= \sum_{j=1}^5 I_j \, dt + \left| \sum_{k=1}^\infty I_6^k \, dW_1^k \right|.
	\end{split}
\end{equation}
Let $\varepsilon > 0$ be fixed. Integrating by parts w.r.t.\ the horizontal coordinates and using the Gagliardo-Nirenberg inequality we get
\begin{equation}
	I_1 + I_2 \leq C \vert \grad_3 \partial_z v \vert^{3/2} \vert \partial_z v \vert^{p-2 + 1/2} \vert v \vert_{L^6} \leq \frac{\varepsilon}{2} \vert \grad_3 \partial_z v \vert^2 \vert \partial_z v \vert^{p-2} + C_\varepsilon \vert v \vert_{L^6}^4 \vert \partial_z v \vert^p.
\end{equation}
Let $\tau_a$ and $\tau_b$ be stopping times satisfying $0 \leq \tau_a \leq \tau_b \leq t \wedge \tau_N \wedge \tau_K^2$. From the bound \eqref{eq:sigma.dz} on $\partial_z \sigma_1(U)$ in $L_2\left(\Uc, L^2 \right)$ we readily deduce
\begin{multline}
	\int_{\tau_a}^{\tau_b} \sum_{j=3}^5  I_j \, ds \leq \left( \frac{\varepsilon}{2} + 6 \eta_3 \right) \int_{\tau_a}^{\tau_b} \vert \partial_z v \vert^{p-2} \vert \grad_3 \partial_z v \vert^2 \, ds\\
	+ C_\varepsilon \int_{\tau_a}^{\tau_b} \vert U \vert^2 + \vert F_v \vert^2 + \vert \partial_z v \vert^p \left( \vert F_v \vert^2 + \Vert U \Vert^2 \right) \, ds.
\end{multline}
Similarly as in \eqref{eq:grad.vb.stochastic} we estimate the stochastic integral by the Burkholder-Davis-Gundy inequality \eqref{eq:bdg} using the bound \eqref{eq:sigma.dz} on $\partial_z \sigma_1(U)$ once more and obtain
\begin{align}
	\nonumber
	p \Eb &\sup_{s \in \left[\tau_a, \tau_b\right]} \left| \int_{\tau_a}^{\tau_b} \sum_{k=1}^\infty I^k_6 \, dW^k_1 \right| \leq pc_{BDG} \Eb \left( \int_{\tau_a}^{\tau_b} \vert \partial_z v \vert^{2p-2} \Vert \partial_z \sigma_1(U) \Vert^2_{L_2(\Uc, L^2)} \, ds \right)^{1/2}\\
	\nonumber
	&\leq \left( 1 - \delta + \varepsilon \right) \Eb \sup_{s \in \left[\tau_a, \tau_b\right]} \vert \partial_z v \vert^p + C_\varepsilon \Eb \int_{\tau_a}^{\tau_b} \left( 1 + \vert \partial_z v \vert^p \right)\left( 1 + \Vert v \Vert^2 + \Vert T \Vert^2 \right) \, ds \\
	\label{eq:vz.stochastic}
	&\hphantom{\leq \ } + \frac{p^2c_{BDG}^2 \eta_3}{4(1-\delta)} \Eb \int_{\tau_a}^{\tau_b} \vert \partial_z v \vert^{p-2} \vert \grad_3 \partial_z v \vert^2 \, ds
\end{align}
for some $\delta \in (0, 1)$. Then we collect the estimates \eqref{eq:vz.after.ito}-\eqref{eq:vz.stochastic} and, recalling the bound \eqref{eq:small.constants.global} on $\eta_3$, we choose $\delta > \varepsilon > 0$ similarly as in the proof of Proposition \ref{prop:grad.vb.estimate} to get
\begin{multline*}
	\Eb \left[ \sup_{s \in [\tau_a, \tau_b]} \vert \partial_z v \vert^4 + \int_{\tau_a}^{\tau_b} \vert \partial_z v \vert^2 \vert \grad_3 \partial_z v \vert^2 \, ds \right] \leq C \Eb \vert \partial_z v(\tau_a) \vert^4 + C\Eb \int_{\tau_a}^{\tau_b} 1 + \vert F_v \vert^2 + \Vert U \Vert^2 \, ds\\
	+ C \Eb \int_{\tau_a}^{\tau_b} \vert \partial_z v \vert^4 \left( 1 + \Vert v \Vert^2 + \Vert T \Vert^2 + \vert F_v \vert^2 + \vert v \vert_{L^6}^4 \right) \, ds.
\end{multline*}
Since we can control $\vert v \vert_{L^6}^4 = \vert \vt + \vb \vert_{L^6}^6 \leq C( \vert \vt \vert_{L^6}^4 + \vert \grad \vb \vert^4)$, the estimate \eqref{eq:vz.estimate} is obtained similarly as in the proof of Proposition \ref{prop:l2.estimate} by the stochastic Gronwall Lemma, see Proposition \ref{prop:stochastic.gronwall}. The convergence $\tau^{\partial_z v}_K \to \infty$ $\Pb$-a.s.\ can be proven by first establishing the convergence $\tilde{\tau}^{\partial_z v}_K = \inf \left\lbrace s \geq 0 \mid f_p(s, \omega) \geq K \right\rbrace \to \infty$ $\Pb$-a.s.\ for $p = 2, 4$, where
\[
	f_p(s, \omega) = \sup_{r \in \left[0, s \wedge \xi \right)} \vert \partial_z v \vert^p + \int_0^{s \wedge \xi} \vert \partial_z v \vert^{p-2} \vert \grad_3 \partial_z v \vert^2 \, dr,
\]
similarly as in the previous proofs by the means of the monotone convergence theorem and Lemma \ref{lemma:stopping.time.infinity}.
\end{proof}

\begin{proposition}
\label{prop:T.estimates}
Let the assumptions of Theorem \ref{thm:global.existence} hold and for $K > 0$ let $\tau^3_K = \tau_K^w \wedge \tau_K^{\vt} \wedge \tau_{K}^{\grad \vb} \wedge \tau_K^{\partial_z v}$. Then for all $t > 0$ and $K, N \in \Nb$ the following estimate holds
\begin{multline}
	\label{eq:T.estimates}
	\Eb \left[ \sup_{s \in \left[0, t \wedge \tau_N \wedge \tau_K^3 \right]} \vert T \vert_{L^6}^6 + \sup_{s \in \left[0, t \wedge \tau_N \wedge \tau_K^3 \right]} \vert \partial_z T \vert^4 + \int_{0}^{t \wedge \tau_N \tau_K^3} \int_{\Mc} \vert \grad_3 T \vert^2 \vert T \vert^4 \, d\Mc \, ds \right]\\
	 + \Eb \left[ \int_{0}^{t \wedge \tau_N \tau_K^3} \vert T \vert_{L^6\left( \Gamma_i \right)}^6 + \vert \partial_z T \vert^2 \vert \grad_3 \partial_z T \vert^2 + \vert \partial_z T \vert^2 \vert \partial_z T \vert^2_{L^2\left( \Gamma_i \right)} \, ds \right]\\
	 \leq C_{t, K} \Eb \left[ \vert T(0) \vert_{L^6}^6 + \vert \partial_z T(0) \vert^4 + \int_0^{t \wedge \tau_N \wedge \tau_K^3} 1 + \Vert U \Vert^2 + \vert F_T \vert^2 + \vert F_T \vert^2_{L^2\left( \Gamma_i \right)} \, ds \right].
\end{multline}
Moreover, the stopping time $\tau_K^T$ defined by
\begin{equation}
	\label{eq:T.stopping.time}
	\tau^T_K = \inf \left\lbrace s \geq 0 \mid \int_0^{s \wedge \xi} \vert T \vert_{L^6}^6 + \vert \partial_z T \vert^2 \Vert \partial_z T \Vert^2 \, dr \geq K \right\rbrace
\end{equation}
satisfies $\tau^T_K \to \infty$ as $K \to \infty$ $\Pb$-almost surely.
\end{proposition}

\begin{proof}
Let $\tau_a$ and $\tau_b$ be stopping times such that $0 \leq \tau_a \leq \tau_b \leq t \wedge \tau_N \wedge \tau_K^3$. First, similarly as in the proof of Proposition \ref{prop:vt.l6.estimate} we deduce
\begin{multline}
	\label{eq:T.l6.estimate}
	\Eb \left[ \sup_{s \in [\tau_a, \tau_b]} \vert T \vert_{L^6}^6 + \int_{\tau_a}^{\tau_b} \left( \int_{\Mc} \vert \grad_3 T \vert^2 \vert T \vert^4 \, d\Mc \right) + \vert T \vert^6_{L^6 \left( \Gamma_i \right)} \, ds \right]\\
	 \leq C \Eb \left[ \vert T(\tau_a)\vert^6_{L^6} + \int_{\tau_a}^{\tau_b} \vert T \vert^6_{L^6} \left( 1 + \vert F_T \vert^2 + \Vert U \Vert^2 \right) + 1 + \vert F_T \vert^2 + \Vert U \Vert^2 \, ds \right].
\end{multline}
Secondly, by the It\^{o} Lemma from Theorem \ref{thm:ito.lemma} we get the estimate
\begin{align*}
	d\vert \partial_z T \vert^4 &+ 4 \mu \vert \partial_z T \vert^2 \vert \grad \partial_z T \vert^2 \, dt+ 4 \nu \vert \partial_z T \vert^2 \vert \partial_{zz} T \vert^2 \, dt + 4\alpha \vert \partial_z T \vert^2 \vert \partial_z T \vert^2_{L^2(\Gamma_i)} \, dt\\
	&\leq 4 \vert \partial_z T \vert^2 \left| \left( \partial_z \left( \left(v \cdot \grad \right) T \right), \partial_z T \right) \right| \, dt + 4 \vert \partial_z T \vert^2 \left| \left( \partial_z \left( w(v) \partial_z T \right), \partial_z T \right) \right| \, dt\\
	&\hphantom{\leq \ } + 4 \vert \partial_z T \vert^2 \left| \left( \partial_z F_T, \partial_z T \right) \right| \, dt + 6 \vert \partial_z T \vert^2 \Vert \partial \sigma_2(U) \Vert^2_{L_2(\Uc, L^2)} \, dt\\
	&\hphantom{\leq \ }+ 4 \vert \partial_z T \vert^2 \left| \sum_{k=1}^\infty \left( \partial_z \sigma_2(v, T) e_k, \partial_z T \right) \, dW_2^k \right|\\
	&= \sum_{j = 1}^{4} I_j \, dt + \left| \sum_{k=1}^\infty I_5^k \, dW_2^k \right|.
\end{align*}
Repeating the integration by parts procedure from of \cite[Proposition 5.3]{debussche2012} we use the Gagliardo-Nirenberg inequality to obtain
\begin{align}
	\nonumber
	I_1 &= 4 \vert \partial_z T \vert^2 \left| \sum_{j = 1}^2 \int_{\Mc} v_j \partial_{jz} T \partial_z T - \partial_{zj} v_j T \partial_z T - \partial_z v_j T \partial_{zj} T \, d\Mc \right|\\
	\nonumber
	&\leq 4 \vert \partial_z T \vert^2 \left( \vert v \vert_{L^6} \vert \grad \partial_z T \vert \vert \partial_z T \vert_{L^3} + \vert \grad \partial_z v \vert \vert T \vert_{L^6} \vert \partial_z T \vert_{L^3} + \vert \grad \partial_z T \vert \vert T \vert_{L^6} \vert \partial_z v \vert_{L^3} \right)\\
	\nonumber
	&\leq \frac{\varepsilon}{3} \vert \partial_z T \vert^2 \vert \partial_{zz} T \vert^2 + C_\varepsilon \left( \vert T \vert^4_{L^6} + \vert v \vert^4_{L^6} \right)\\
	&\hphantom{\leq \ }+ C_\varepsilon \vert \partial_z T \vert^4 \left( \vert v \vert^4_{L^6} + \vert \grad_3 \partial_z v \vert^2 + \vert \grad_3 \partial_z T \vert^2 \vert \partial_z v \vert^2 \right)
\end{align}
and
\begin{align}
	\nonumber
	I_2 &= 4 \vert \partial_z T \vert^2 \left| \sum_{j=1}^2 \int_{\Mc} v_j \partial_{jz} T \partial_z T \, d\Mc \right| \leq C \vert \grad_3 \partial_z T \vert^{3/2} \vert \partial_z T \vert^{5/2} \vert v \vert_{L^6}\\
	&\leq \frac{\varepsilon}{3} \vert \partial_z T \vert^2 \vert \grad_3 \partial_z T \vert^2 + C_\varepsilon \vert \partial_z T \vert^4 \vert v \vert^4_{L^6}.
\end{align}
Using integration by parts again we infer
\begin{equation}
	\begin{split}
		\int_{\tau_a}^{\tau_b} I_3 \, ds &= 4 \int_{\tau_a}^{\tau_b} \vert \partial_z T \vert^2 \left| \left( F, \partial_{zz} T \right) + \alpha \int_{\Gamma_i}  T F_T \, d\Gamma_i \right| \, ds\\
		&\leq \int_{\tau_a}^{\tau_b} \frac{\varepsilon}{3} \vert \partial_z T \vert^2 \vert \grad_3 \partial_z T \vert^2 + C_\varepsilon \left( \vert \partial_z T \vert^4 + 1 \right) \left(1 + \Vert U \Vert^2 + \vert F_T \vert^2 + \vert F_T \vert^2_{L^2\left( \Gamma_i \right)} \right) \, ds.
	\end{split}
\end{equation}
Similarly as in \eqref{eq:grad.vb.stochastic} we employ the Burkholder-Davis-Gundy inequality \eqref{eq:bdg} and the bound on $\partial_z \sigma_2(U)$ in $L_2\left( \Uc, L^2 \right)$ \eqref{eq:sigma.dz} to deduce
\begin{equation}
	\label{eq:Tz.stochastic}
	\begin{split}
	\Eb \sup_{s \in [\tau_a, \tau_b]} &\left| \sum_{k=1}^\infty I_5^k \, dW_2^k \right| \leq 4 C_{BDG} \Eb \left( \int_{\tau_a}^{\tau_b} \vert \partial_z T \vert^6 \Vert \partial_z \sigma_2(U) \Vert^2_{L_2(\Uc, L^2)} \, ds \right)^{1/2}\\
	&\leq C_t \Eb \int_{\tau_a}^{\tau_b} \left(1 + \vert \partial_z T \vert^4 \right) \left( 1 + \Vert U \Vert^2 \right) \, ds + \left( 1 - \delta + \varepsilon \right) \Eb \sup_{s \in [\tau_a, \tau_b]} \vert \partial_z T \vert^4\\
	&\hphantom{\leq \ } + \frac{4C_{BDG}^2 \eta_3}{1-\delta} \Eb \int_{\tau_a}^{\tau_b} \vert \partial_z T \vert^2 \vert \grad_3 \partial_z T \vert^2 \, ds.
	\end{split}
\end{equation}
Collecting the estimates \eqref{eq:T.l6.estimate}-\eqref{eq:Tz.stochastic} while recalling that $\eta_3$ satisfies \eqref{eq:small.constants.global} and choosing $\delta > \varepsilon > 0$ sufficiently small, we use the bound \eqref{eq:sigma.dz} once more to obtain
\begin{align*}
	\Eb &\left[ \sup_{s \in [\tau_a, \tau_b]} \left( \vert T \vert_{L^6}^6 + \vert \partial_z T \vert^4 \right) + \int_{\tau_a}^{\tau_b} \left( \int_{\Mc} \vert \grad_3 T \vert^2 \vert T \vert^4 \, d\Mc \right) + \vert T \vert^6_{L^6 \left( \Gamma_i \right)} \, ds \right]\\
	&\qquad + \Eb \int_{\tau_a}^{\tau_b} \vert \grad_3 \partial_z T \vert^2 \vert \partial_z T \vert^2 + \vert \partial_z T \vert^2 \vert \partial_z T \vert^2_{L^2\left( \Gamma_i \right)} \, ds\\
	&\leq C \Eb \left[ \vert T(\tau_a)\vert^6_{L^6} + \vert \partial_z T \vert^4 \right] + + C\Eb \int_{\tau_a}^{\tau_b} 1 + \vert F_T \vert^2 + \vert F_T \vert^2_{L^2\left( \Gamma_i \right)} + \Vert U \Vert^2 \, ds\\
	&\hphantom{\leq \ }+ C \Eb \int_{\tau_a}^{\tau_b} \left( \vert T \vert^6_{L^6} + \vert \partial_z T \vert^4 \right) \left( 1 + \vert F_T \vert^2 + \vert F_T \vert^2_{L^2\left( \Gamma_i \right)} + \Vert U \Vert^2 \right) \, ds\\
	 &\hphantom{\leq \ } + C \Eb \int_{\tau_a}^{\tau_b}  \left( \vert T \vert^6_{L^6} + \vert \partial_z T \vert^4 \right) \left( \vert v \vert^4_{L^6} + \vert \grad_3 \partial_z v \vert^2 + \vert \grad_3 \partial_z v \vert^2 \vert \partial_z v \vert^2 \right) \, ds.
\end{align*}
The estimate \eqref{eq:T.estimates} then follows by the stochastic Gronwall Lemma from Proposition \ref{prop:stochastic.gronwall}. The convergence of the stopping times $\tau_K^T$ follows from the same argument as in the previous proofs.
\end{proof}

\subsection{Proof of Theorem \ref{thm:global.existence}}
\label{sec:proof.global.existence}

The proof will be complete once we establish
\begin{equation}
	\label{eq:final.proof}
	\Pb(\lbrace \xi < \infty \rbrace) = 0.
\end{equation}
The technique of the proof comes from \cite[Theorem 3.2]{debussche2012}. For $K \in \Nb$ we define
\[
	\tau_K^U = \tau_K^w \wedge \tau_K^{\vt} \wedge \tau_{K}^{\grad \vb} \wedge \tau_K^{\partial_z v} \wedge \tau^T_K,
\]
where the stopping times $\tau_K^w$, $\tau_K^{\vt}$, $\tau_{K}^{\grad \vb}$, $\tau_K^{\partial_z v}$ and $\tau^T_K$ are defined in \eqref{eq:weak.stopping.time}, \eqref{eq:vt.stopping.time}, \eqref{eq:grad.vb.stopping.time}, \eqref{eq:vz.stopping.time} and \eqref{eq:T.stopping.time}, respectively. by Propositions \ref{prop:l2.estimate}, \ref{prop:vt.l6.estimate}, \ref{prop:grad.vb.estimate}, \ref{prop:vz.estimates} and \ref{prop:T.estimates} we infer that $\tau^U_K \to \infty$ $\Pb$-a.s.\ as $K \to \infty$.

Before we embark on proving \eqref{eq:final.proof} let us establish the estimate
\begin{multline}
	\label{eq:U.V.estimate}
	\Eb \left[ \sup_{t \in \left[0, t \wedge \tau_N \wedge \tau^U_K\right]} \Vert U \Vert^2 + \int_0^{t \wedge \tau_N \wedge \tau_K^U} \vert AU \vert^2 \, ds \right]\\
	\leq C_{t, K} \Eb \left[ \Vert U(0) \Vert^2	+ \int_0^{t \wedge \tau_N \wedge \tau_K^U} 1 + \vert F_U \vert^2 \, ds \right].
\end{multline}
for all $K, N \in \Nb$. By the It\^{o} Lemma from Theorem \ref{thm:ito.lemma} we have
\begin{equation}
	\label{eq:grad.v.after.ito}
	\begin{split}
		d\Vert U \Vert^2 &+ 2 \vert AU \vert^2 \, dt \leq 2 \left| b(U, U, AU) \right| \, dt + 2 \left| \left( A^{1/2} F(U), A^{1/2} U \right) \right| \, dt\\
		&\hphantom{= \ } + \Vert A^{1/2}\sigma(U) \Vert^2_{L_2(\Uc, H)} \, dt + 2 \left| \sum_{k=1}^\infty \left( A^{1/2} \sigma_1(U) e_k, A^{1/2} U \right) \, dW^k \right| \\
		&= \sum_{j = 1}^3 I_j \, dt + \left| \sum_{k=1}^\infty I_4^k \, dW^k_1 \right|.
	\end{split}
\end{equation}
Let $\varepsilon > 0$ be fixed. By the estimate on $B$ \eqref{eq:b.estimate3} we have
\begin{equation}
	I_1 \leq \frac{\varepsilon}{2} \vert AU \vert^2 + C_\varepsilon \Vert U \Vert^2 \left( \vert v \vert_{L^6}^4 + \vert \partial_z U \vert^2 \Vert \partial_z U \Vert^2 \right).
\end{equation}
From the definition of $F(U)$ \eqref{eq:F.definition} we deduce
\begin{equation}
	\int_{\tau_a}^{\tau_b} I_2 \, ds \leq \frac{\varepsilon}{2} \int_{\tau_a}^{\tau_b} \vert AU \vert^2 \, ds + C_\varepsilon \int_{\tau_a}^{\tau_b} 1 + \Vert U \Vert^2 + \vert F_U \vert^2 \, ds.
\end{equation}
Similarly as in the previous proofs we use the Burkholder-Davis-Gundy inequality \eqref{eq:bdg} and the bound on $\sigma(U)$ in $L_2\left( \Uc, V \right)$ \eqref{eq:sigma.bnd.V} to get
\begin{multline}
	\label{eq:U.V.stochastic}
	\Eb \sup_{s \in [\tau_a, \tau_b]} \left| \sum_{k=1}^\infty I_4^k \, dW^k \right| \leq C_{t, \varepsilon} \Eb \int_{\tau_a}^{\tau_b} 1 + \Vert U \Vert^2 \, ds\\
	+ \left( 1 - \delta + \varepsilon \right) \Eb \sup_{s \in [\tau_a, \tau_b]} \Vert U \Vert^2 + \frac{C_{BDG}^2 \eta_1}{1-\delta} \Eb \int_{\tau_a}^{\tau_b} \vert AU \vert^2 \, ds.
\end{multline}
Collecting the estimates \eqref{eq:grad.v.after.ito}-\eqref{eq:U.V.stochastic} and choosing $0 < \varepsilon < \delta$ sufficiently small we get
\begin{multline*}
	\Eb \left[ \sup_{s \in [\tau_a, \tau_b]} \Vert U \Vert^2 + \int_{\tau_a}^{\tau_b} \vert AU \vert^2 \, ds \right] \leq C \Eb \Vert U(\tau_a) \Vert^2\\
	+ C \Eb \left[ \int_{\tau_a}^{\tau_b} \Vert U \Vert^2 \left( 1 + \vert v \vert^4_{L^6} + \vert \partial_z U \vert^2 \Vert \partial_z U \Vert^2 \right)\, ds + \int_{\tau_a}^{\tau_b} 1 + \vert F \vert^2 \, ds \right].
\end{multline*}
The estimate \eqref{eq:U.V.estimate} is then established by the stochastic Gronwall Lemma from Proposition \ref{prop:stochastic.gronwall}. Since the constant on the right-hand side of \eqref{eq:U.V.estimate} is independent of $N$, we may use the monotone convergence theorem to infer that for all $t > 0$ and $K \in \Nb$
\begin{equation*}
	\Eb \left[ \sup_{0 \in \left[0, t \wedge \xi \wedge \tau_K^U \right)} \Vert U \Vert^2 + \int_0^{t \wedge \xi \wedge \tau_K^U } \vert AU \vert^2 \, ds \right] \leq C_{t, K} \Eb \left[  \Vert U(0) \Vert^2 + \int_0^{t \wedge \xi} 1 + \vert F \vert^2 \, ds \right] < \infty.
\end{equation*}
In particular we deduce that for all $t > 0$ and $K \in \Nb$
\begin{equation}
	\label{eq:U.V.contradiction}
	\sup_{s \in \left[0, t \wedge \xi \wedge \tau_K^U \right)} \Vert U \Vert^2 + \int_0^{t \wedge \xi \wedge \tau_K^U } \vert AU \vert^2 \, ds < \infty \qquad \Pb\text{-a.s.}.
\end{equation}

Finally, we are ready to prove \eqref{eq:final.proof}. Since from the definition $\tau_K^U \to \infty$ $\Pb$-almost surely, it suffices to establish that for all $K \in \Nb$ $\tau^U_K \leq \xi$ $\Pb$-almost surely. Arguing by contradiction, let us assume that $\Pb \left( \left\lbrace \tau_K^U > \xi \right\rbrace \right) > 0$ for some $K \in \Nb$. Since
\[
	\left\lbrace \tau_K^U > \xi \right\rbrace = \bigcup_{t \in \Nb}^\infty	\left\lbrace \tau_K^U \wedge t > \xi \right\rbrace,
\]
we deduce that $\Pb \left( \left\lbrace \tau_K^U \wedge t > \xi \right\rbrace \right) > 0$ for some $t \in \Nb$. Now on the set $\left\lbrace \tau_K^U \wedge t > \xi \right\rbrace$ the explosion property \eqref{eq:explosion} on $\xi < t < \infty$ implies that
\[
	\sup_{s \in \left[0, t_0 \wedge \xi \wedge \tau_K^U \right)} \Vert U \Vert^2 + \int_0^{t_0 \wedge \xi \wedge \tau_K^U } \vert AU \vert^2 \, ds \geq \sup_{s \in \left[0, \xi \right)} \Vert U \Vert^2 + \int_0^{\xi} \vert AU \vert^2 \, ds = +\infty
\]
on a set of non-zero measure. This contradicts \eqref{eq:U.V.contradiction}. \qed

\section*{Acknowledgement}
\addcontentsline{toc}{section}{Acknowledgement}

J.S.\ wishes to express gratitude for the kind hospitality of University of York, where the research leading to this work has been done, and to the Programme for research and mobility support of young researchers of the Czech Academy of Sciences, which made the stay in York possible.

\appendix

\section{The It\^{o} Lemma}
\label{app:ito}

In this Section we prove a generalization of the It\^{o} Lemma  formulated and proved by  Pardoux in \cite[Theorem 1.2]{pardoux1979}. We have used this  version of the notoriously known result to establish the higher integrability of solutions in Section \ref{sect:maximal.solution} and the estimates necessary for the global existence in Section \ref{sect:global.solution}.

Let $\left( \Omega, \Fc, \Fb, \Pb \right)$ be a stochastic basis with filtration $\Fb = \Fct$ and let $\Uc$ be a separable Hilbert space with an orthonormal basis $\lbrace e_k \rbrace_{k = 1}^\infty$. Let $W$ be an $\Fb$-adapted cylindrical Wiener process with reproducing kernel Hilbert space $\Uc$. Let $V$ and $H$ be separable Hilbert spaces such that the embedding $V \hook H $ is dense and compact and let $A: D(A) \to H$ be an unbounded self-adjoint densely defined bijective operator on $H$ such that $(AU, U^\sharp)_H = (U, U^\sharp)_V$ for $U, U^\sharp \in V$. Following a standard argument one can show that there exists an orthonormal basis $\lbrace E_k \rbrace_{k = 1}^\infty$ of $H$ consisting of eigenvaules of $A$, in particular $E_k \in D(A)$ and $AE_k = \lambda_k E_k$ for all $k \in \Nb$. We may then define the fractional powers of $A$ by
\[
	D\left( A^{\alpha} \right) = \left\lbrace U \in H \mid \sum_{k = 1}^\infty \lambda_k^{2\alpha} \left| \left(U, E_k \right)_H \right|^2 < \infty \right\rbrace, \quad A^\alpha U = \sum_{k = 1}^\infty \lambda^\alpha_k \left( U, E_k \right) E_k, \ U \in D\left(A^\alpha\right).
\]
For simplicity let use denote $X_{\alpha} = D(A^{\alpha/2})$ and $\Vert U \Vert_{\alpha} = \Vert U \Vert_{D(A^{\alpha/2})}$. One can identify $X_0 = H$, $X_1 = V$ and $X_2 = D(A)$.

For an auxiliary separable Hilbert space $K$ and $T > 0$ let $M^2(0, T; K)$ be the Hilbert space consisting of all equivalence classes of progressively measurable $K$-valued processes $\psi$ such that $\Eb \int_0^T \Vert \psi_t \Vert_K^2 < \infty$.

\begin{theorem}
\label{thm:ito.lemma}
Let $\alpha \geq 0$ and $p \geq 2$ and let $U^0 \in L^p\left( \Omega; \Fc_0, X_{\alpha + 1} \right)$. Let $T > 0$ be fixed and let $\tau$ be a stopping time such that $\tau \leq T$. Let $v: [0, T] \times \Omega \to X_\alpha$ and $g: [0, T] \times \Omega \to L_2(\Uc, X_{\alpha+1})$ be progressively measurable processes such that
\begin{equation}
	\label{eq:ito.v.g.regularity}
	\Eb \int_0^T \mathds{1}_{[0, \tau]}(s) \left( \Vert v_s \Vert^2_{\alpha} + \Vert g_s \Vert^2_{L_2(\Uc, X_{\alpha+1})} \right) \, ds < \infty.
\end{equation}
Let $U: [0, T] \times \Omega \to V$ be a progressively measurable stochastic process such that
\begin{equation}
	\label{eq:ito.U.regularity.1}
	U(\cdot \wedge \tau) \in L^2\left( \Omega, C\left( [0, T], X_{\alpha+1} \right) \right), \qquad \mathds{1}_{[0, \tau]}U  \in L^2\left( \Omega, L^2\left(0, T; X_{\alpha+2} \right) \right),
\end{equation}
and let $U$ satisfy the equation in the space $X_\alpha$
\begin{equation}
	\label{eq:ito.spde}
	U_{t \wedge \tau} + \int_0^{t \wedge \tau} A U_s + v_s \, ds = U_0 +  \int_0^{t \wedge \tau} g_s \, dW_s, \ U_0 = U^0, \ \Pb\text{-a.s.\ for all} \ t \in [0, T].
\end{equation}
Let $\psi: X_{\alpha+1} \to \Rb$ be such that
\begin{enumerate}
	\item $\psi \in C^2(X_{\alpha+1}, \Rb)$,
	\item \label{item:ito.assumption2}$\psi$ and the Fr\'{e}chet derivatives $D \psi$ and $D^2 \psi$ are uniformly continuous and bounded on balls in $X_{\alpha+1}$,
	\item if $u \in X_{\alpha+2}$, then $D\psi(u) \in L(X_{\alpha+1}, \Rb)$ can be extended to $D\psi(u) \in L(X_\alpha, \Rb)$. Moreover,
	\begin{enumerate}
		\item for all $t > 0$ and $R \geq 0$ exists $C_{R, t}$ such that if $u \in C\left( [0, t], X_{\alpha+1} \right) \cap L^2\left(0, t; X_{\alpha+2} \right)$ is such that
		\[
			\Vert u \Vert_{C\left( [0, t], X_{\alpha+1} \right)} + \Vert u \Vert_{L^2\left( 0, t; X_{\alpha+2} \right)} \leq R,
		\]
		then one has
		\begin{equation}
			\label{eq:ito.D.psi.bound}
			\sup_{s \in [0, t]} \Vert D\psi(u_s) \Vert_{L\left( X_\alpha, \Rb \right)} \leq C_{R, t},
		\end{equation}
		\item for all $t > 0$ if $u^n \to u$ in $C\left( [0, t], X_{\alpha+1} \right) \cap L^2\left(0, t; X_{\alpha+2} \right)$, then for all $h \in L^2\left(0, t; X_{\alpha} \right)$
		\begin{equation}
			\label{eq:ito.D.psi.convergence}
			\int_0^{s} D\psi\left(u^n_r\right)(h_r) \, dr \to \int_0^{s} D\psi\left(u_r \right)(h_r) \, dr, \quad s \in [0, t].
		\end{equation}
	\end{enumerate}	
\end{enumerate}
Then for all $t \in [0, T]$ $\Pb$-almost surely
\begin{multline}
	\label{eq:ito.claim}
	\psi\left(U_{t \wedge \tau}\right) = \psi\left( U_0 \right) + \int_0^{t \wedge \tau} D\psi\left( U_s \right) \left( AU_s + v_s \right) \, ds + \int_0^{t \wedge \tau} D\psi\left( U_s \right)\left( g_s \, dW_s \right)\\
	+ \frac{1}{2} \sum_{k=1}^\infty \int_0^{t \wedge \tau} D^2 \psi\left( U_s \right) \left( g_s e_k, g_s e_k \right) \, ds.
\end{multline}
\end{theorem}

We will need the following version of It\^{o} Lemma for processes with bounded variation from \cite[Lemma 1.3]{pardoux1979}.

\begin{lemma}
\label{lemma:ito.pardoux}
Let $\tilde{g} \in M^2\left( 0, T; L_2\left( \Uc, X_{\alpha+2} \right) \right)$. Let $\Vc: [0, T] \times \Omega \to X_{\alpha+1}$ be stochastic process with trajectories of bounded variation and let $M_t = \int_0^t \tilde{g}_s \, dW_s$, $t\in [0,T]$. Let $\psi \in C^2 \left( X_{\alpha+1}, \Rb \right)$ satisfy the assumptions 1.-3.\ from Theorem \ref{thm:ito.lemma}. Then for all $t \in [0, T]$ $\Pb$-almost surely
\begin{multline*}
	\psi\left(\Vc_t + M_t\right) = \psi(\Vc_0) + \int_0^t D\psi\left(\Vc_s + M_s \right)\left( d\Vc_s \right) + \int_0^t D\psi\left(\Vc_s + M_s \right) \left( \tilde{g}_s \, dW_s \right)\\
	+ \frac12 \sum_{k=1}^\infty \int_0^t D^2 \psi\left( \Vc_s + M_s \right) \left(\tilde{g}_s e_k,\tilde{g}_s e_k \right) \, ds.
\end{multline*}
\end{lemma}

\begin{proof}[Proof of Theorem \ref{thm:ito.lemma}]
\emph{Step 1}. First, let us prove the claim with the additional assumption
\[
	\Eb \int_0^T \mathds{1}_{[0, \tau]}(s) \Vert g_s \Vert_{L_2(\Uc, X_{\alpha+2})}^2 < \infty.
\]
Let $M_t = \int_0^t \mathds{1}_{[0, \tau]}(s) g_s \, dW_s$ for $t \in [0, T]$, then we have $M \in M^2 \left( 0, T; X_{\alpha+2} \right)$. Defining $\tilde{U}_t = U_{t \wedge \tau} - M_t$ for $t \in [0, T]$ we observe that
\[
	\tfrac{d}{ds} \tilde{U}_s = \mathds{1}_{[0, \tau]}(s) \left( AU_s + v_s \right) \ \Pb\text{-a.s.\ for a.a.} \ s \in (0, t).
\]
We infer that
\[
	\tilde{U} \in M^2\left(0, T; X_{\alpha+2} \right) \quad \text{and} \quad \tfrac{d}{dt} \tilde{U} \in M^2 \left( 0, T; X_{\alpha} \right).
\]
Let $\tilde{U}^n \in M^2 \left( 0, T; X_{\alpha+2} \right)$ be such that $\tilde{U}^n \in C^1\left( [0, T], X_\alpha \right)$ $\Pb$-almost surely, $\tilde{U}^n_0 = U_0$ and
\begin{equation}
	\label{eq:ito.convergence}
	\tilde{U}^n \to \tilde{U} \ \text{in} \ M^2\left( 0, T; X_{\alpha+2} \right), \quad \tfrac{d}{dt} \tilde{U}^n =: \tilde{v}^n\to \mathds{1}_{[0, \tau]}\left(AU + v\right) \ \text{in} \ M^2\left(0, T; X_{\alpha} \right).
\end{equation}
In particular, $\tilde{U}^n$ have bounded variation. Such a sequence can be constructed by defining $\tilde{U}^n = \Kc_{1/n} \tilde{U}$, where $\Kc_{1/n}$ is a convolution operator similar to the one in \eqref{eq:convolution.operator}. Assuming we extend $U|_{(-\infty, 0)} = U_0$, which is justified by the continuity of $\tilde{U}$, the requirement $\tilde{U}^n_0 = U_0$ is satisfied. The resulting processes are measurable since the convolution operator $\Kc_{1/n}$ is continuous on $L^2(0, T; X_{\alpha+2})$ for all $n \in \Nb$. By passing to a (not relabelled) subsequence we may assume that in fact
\begin{equation}
	\label{eq:ito.convergence.as}
	\tilde{U}^n \to \tilde{U} \ \text{in} \ L^2\left( 0, T; X_{\alpha+2} \right) \ \Pb\text{-a.s.}, \quad \tilde{v}^n \to \mathds{1}_{[0, \tau]} \left( AU + v \right) \ \text{in} \ L^2\left(0, T; X_{\alpha} \right) \ \Pb\text{-a.s.}
\end{equation}
Moreover from the Lions-Magenes Lemma, see e.g.\ \cite[Lemma 3.1.2]{temam1979}, and the dominated convergence theoremwe deduce that 
\begin{equation}
	\label{eq:ito.continuous.convergence}
	\tilde{U}^n \to \tilde{U} \quad \text{in} \quad L^2 \left( \Omega,  C\left( [0, T], X_{\alpha+1} \right) \right).
\end{equation}
Furthermore, by choosing a suitable (again not relabelled) subsequence we can assume that the convergence in \eqref{eq:ito.continuous.convergence} is sufficiently fast so that
\begin{equation}
	\label{eq:ito.continuous.convergence.2}
	\sum_{n = 1}^\infty \Eb \sup_{t \in [0, T]} \Vert \tilde{U}^n_t - \tilde{U}_t \Vert_{\alpha+1}^2 < \infty.
\end{equation}
This in turn implies that 
\begin{equation}
	\label{eq:ito.continuous.convergence.as}
	\tilde{U}^n \to \tilde{U} \quad \Pb\text{-a.s.}\ \text{in} \ C([0, T], X_{\alpha+1}).
\end{equation}
For the later use let $\tilde{\Omega} \subseteq \Omega$ be the set of full-measure on which the convergences \eqref{eq:ito.convergence.as} and \eqref{eq:ito.continuous.convergence.as} hold. By the It\^{o} Lemma for processes with bounded variation, see Lemma \ref{lemma:ito.pardoux},  recalling the definition of $M$, we  get $\Pb$-almost surely for all $t \in [0, T]$
\begin{multline}
	\label{eq:ito.U.approximation.1}
	\psi\left( \tilde{U}^n_t + M_t \right) = \psi \left( U_0 \right) + \int_0^t D\psi\left( \tilde{U}^n_s + M_s \right) \left(\tilde{v}^n_s \right) \, ds + \int_0^{t \wedge \tau} D\psi\left(\tilde{U}^n_s + M_s \right) \left( g_s \, dW_s \right)\\
	+ \frac12 \sum_{k = 1}^\infty \int_0^{t \wedge \tau} D^2 \psi\left( \tilde{U}^n_s + M_s \right)\left(g_s e_k, g_s e_k \right) \, ds.
\end{multline}

It remains to pass to the limit w.r.t.\ $n \to \infty$ in \eqref{eq:ito.U.approximation.1}. Assume that $t \leq \tau(\omega)$ for some $\omega \in \tilde{\Omega}$; the case $t > \tau(\omega)$ being similar. Then
\begin{multline*}
	\left\vert \int_0^t D\psi\left( \tilde{U}^n_s + M_s \right)\left( \tilde{v}^n_s \right) - D\psi\left( \tilde{U}_s + M_s \right)\left( AU_s + v_s \right) \, ds \right\vert\\
	\leq \left\vert \int_0^t D\psi\left( \tilde{U}^n_s + M_s \right) \left( \tilde{v}^n_s - AU_s - v_s \right) \, ds \right\vert\\
	+ \left\vert \int_0^t \left[ D\psi\left( \tilde{U}^n_s + M_s \right) - D\psi\left( \tilde{U}_s + M_s \right) \right] \left( AU_s + v_s \right) \, ds \right\vert.
\end{multline*}
The first term converges to $0$ by the convergence of $v^n$ \eqref{eq:ito.convergence.as} localized on $(0, \tau(\omega))$ and by \eqref{eq:ito.D.psi.bound}. The second term converges to $0$ by the boundedness of $\tilde{U}^n$ in $ C\left( [0, \tau(\omega)], X_{\alpha+1} \right) \cap L^2\left( 0, \tau(\omega); X_{\alpha+2} \right)$ and \eqref{eq:ito.D.psi.convergence}.

Regarding the convergence of the stochastic term, let
\[
	\Omega = \bigcup_{R \in \Nb} \left\lbrace \tilde{U}^n,  \tilde{U}, M \in B_R\left( C\left( [0, T], X_{\alpha+1} \right) \cap L^2\left(0, T; X_{\alpha+2} \right) \right) \right\rbrace \equiv \bigcup_{R \in \Nb} \Omega_R,
\]
where $B_R(Y)$ denotes a ball of radius $R$ in a Banach space $Y$. Then for $R \in \Nb$ arbitrary we use the Fubini Theorem, the Burhkolder-Davis-Gundy inequality \eqref{eq:bdg}, the ideal property of Hilbert-Schmidt operators, the Lipschitz continuity of $D\psi$ on balls in $X_{\alpha+1}$, i.e.\ the assumption \ref{item:ito.assumption2}, and finally the convergence property \eqref{eq:ito.continuous.convergence.2} to deduce
\begin{align*}
	\Eb &\sum_{n= 1}^\infty \sup_{t \in [0, T]} \mathds{1}_{\Omega_R} \left| \int_0^{t} \left[ D\psi\left( \tilde{U}^n_s + M_s \right) - D\psi\left( \tilde{U}_s + M_s \right) \right] \left( \mathds{1}_{[0, \tau]}(s) g_s \, dW_s \right) \right|\\
	&\leq C \sum_{n=1}^\infty \Eb \left( \mathds{1}_{\Omega_R} \int_0^\tau \Vert D\psi\left( \tilde{U}^n_s + M_s \right) - D\psi\left( \tilde{U}_s + M_s \right) \Vert_{L\left(X_{\alpha+1}, \Rb\right)}^2 \Vert g_s \Vert_{L_2\left(\Uc, X_{\alpha+1}\right)}^2 \, ds \right)^{1/2}\\
	&\leq C_R \Eb \left[ \int_0^\tau \Vert g_s \Vert_{L_2\left(\Uc, X_{\alpha+1}\right)}^2 \, ds \right] \sum_{n = 1}^\infty \Eb \left[ \Vert \tilde{U}^n - \tilde{U} \Vert_{C\left([0, T], X_{\alpha+1}\right)}^2 \right] < \infty.
\end{align*}
This implies that
\[
	\sup_{t \in [0, \tau]} \left| \int_0^{t} \left[ D\psi\left( \tilde{U}^n_s + M_s \right) - D\psi\left( \tilde{U}_s + M_s \right) \right] \left( \mathds{1}_{[0, \tau]}(s) g_s \, dW_s \right) \right| \to 0 \quad \Pb\text{-a.s.\ on} \ \Omega_R.
\]
Since $R \in \Nb$ was arbitrary, we get the desired $\Pb$-a.s.\ convergence.

Finally, let $\omega \in \tilde{\Omega}$ and $t \leq \tau(\omega)$.
The convergence \eqref{eq:ito.continuous.convergence.as} and the uniform continuity of $D^2\psi$ on balls in $X_{\alpha+1}$ imply
\begin{multline*}
	\left| \int_0^t \sum_{k=1}^\infty \left[ D^2\psi\left( \tilde{U}^n_s + M_s \right) - D^2\psi\left( \tilde{U}_s + M_s \right) \right] \left( g_s e_k, g_s e_k \right) \, ds \right|\\
	\leq \sup_{s \in [0, t]} \left\Vert D^2\psi\left( \tilde{U}^n_s + M_s \right) - D^2\psi\left( \tilde{U}_s + M_s \right) \right\Vert_{L\left( X_{\alpha+1} \times X_{\alpha+1}, \Rb \right)} \int_0^\tau \Vert g_s \Vert_{L_2\left( \Uc, X_{\alpha+1} \right)}^2 \, ds \to 0
\end{multline*}
on $\tilde{\Omega}$. We are therefore able to pass to the limit in \eqref{eq:ito.U.approximation.1} and obtain \eqref{eq:ito.claim}.

\emph{Step 2}. To prove the general case, let $g^n \in M^2\left( 0, T; L_2\left( \Uc, X_{\alpha+2} \right) \right)$ be such that $g^n \to \mathds{1}_{[0, \tau]} g$ in $M^2 \left(0, T; L_2\left( \Uc, X_{\alpha+1} \right) \right)$ and
\begin{equation}
	\label{eq:ito.g.convergence}
	\sum_{n=1}^\infty \Eb \int_0^T \Vert g^n_s - \mathds{1}_{[0, \tau]}(s) g_s \Vert_{L_2\left( \Uc, X_{\alpha+1} \right)}^2 \, ds < \infty.
\end{equation}
Let $U^n$ be the solution of the equation in the space $X_{\alpha}$
\begin{equation}
	\label{eq:ito.sde.approximation}
	U^n_t + \int_0^t \mathds{1}_{[0, \tau]}(s) \left( AU^n_s + v_s \right) \, ds = \int_0^t g^n_s \, dW_s, \quad t \in [0, T], \qquad U^n_0 = U_0.
\end{equation}
Following a standard argument using the Burkholder-Davis-Gundy inequality \eqref{eq:bdg} and the It\^{o} Lemma from e.g.\  \cite[Theorem 1.2]{pardoux1979} we may show that
\[
	U^n \in L^2 \left( \Omega, C\left( [0, T], X_{\alpha+1} \right) \right), \qquad \mathds{1}_{[0, \tau]} U^n \in L^2\left( \Omega, L^2\left( 0, T; X_{\alpha+2} \right) \right).
\]
Therefore by Step 1 we have $\Pb$-almost surely
\begin{multline}
	\label{eq:ito.approximation.2}
	\psi\left( U^n_t \right) = \psi\left( U_0 \right) + \int_0^{t \wedge \tau} D\psi\left( U^n_s \right) \left( AU^n_s + v_s \right) \, ds + \int_0^t D\psi\left(U^n_s\right)\left( g^n_s \, dW_s \right)\\
	+ \frac12 \sum_{k=1}^\infty \int_0^t D^2\psi\left( U^n_s \right) \left( g^n_s e_k, g^n_s e_k \right) \, ds, \qquad t \in [0, T].
\end{multline}
Subtracting \eqref{eq:ito.sde.approximation} from \eqref{eq:ito.spde} we may repeating the argument above relying on the Burkholder-Davis-Gundy inequality and the It\^{o} Lemma by from \cite{pardoux1979} to get
\begin{equation}
	\label{eq:ito.second.convergence}
	\begin{gathered}
		U^n \to U( \cdot \wedge \tau) \ \text{in} \ L^2\left( \Omega, C\left( [0, T], X_{\alpha+1} \right) \right),\\
		\mathds{1}_{[0, \tau]} U^n \to \mathds{1}_{[0, \tau]} U \ \text{in} \ L^2\left( \Omega, L^2\left(0, T; X_{\alpha+2} \right) \right).
	\end{gathered}	
\end{equation}
Similarly as in Step 1 we may also use \eqref{eq:ito.g.convergence} and \eqref{eq:ito.second.convergence} to obtain
\[
	U^n \to U( \cdot \wedge \tau) \ \text{in} \ C\left( [0, T], X_{\alpha+1} \right) \ \Pb\text{-a.s.}, \quad \mathds{1}_{[0, \tau]} U^n \to \mathds{1}_{[0, \tau]} U \ \text{in} \ L^2\left(0, T; X_{\alpha+2} \right) \ \Pb\text{-a.s.}
\]
The passage to the limit in \eqref{eq:ito.approximation.2} now follows similarly as in Step 1.
\end{proof}

\addcontentsline{toc}{section}{References}
\bibliography{bibliography}
\bibliographystyle{plain}

\end{document}